\documentclass[12pt,a4paper]{article}
\usepackage{amsfonts}
\usepackage{amsthm}
\usepackage{amsmath}
\usepackage{amscd}
\usepackage{amssymb}
\usepackage{dsfont}
\usepackage[latin2]{inputenc}
\usepackage{t1enc}
\usepackage[mathscr]{eucal}
\usepackage{indentfirst}
\usepackage{graphicx}
\usepackage{graphics}
\usepackage{subcaption}
\usepackage{pict2e}
\usepackage{epic}
\numberwithin{equation}{section}
\usepackage[margin=2cm]{geometry}
\usepackage{epstopdf} 
\usepackage[english]{babel}
\usepackage{hyperref}
\usepackage{bbm}
\usepackage{tikz,pgf}
\usepackage{cite}
\usepackage{enumerate}
\usepackage[shortlabels]{enumitem}

\parskip \medskipamount
\definecolor{myBlue}{rgb}{0, 0, 0.7}
\hypersetup{
   colorlinks=true,
   linkcolor=myBlue, 
   citecolor=myBlue,
}

\newtheorem{theorem}{Theorem}[section]
\newtheorem{lemma}[theorem]{Lemma}

\newtheorem{corollary}[theorem]{Corollary}
\theoremstyle{definition}
\newtheorem{remark}[theorem]{Remark}
\newtheorem{definition}[theorem]{Definition}

\newcommand{\ind}[1]{\text{$\mathbbm{1}$}\left(#1\right)}
\newcommand{\PR}{\mathbb{P}}
\newcommand{\E}{\mathbb{E}}
\newcommand{\Var}{\mathrm{Var}}
\newcommand{\lr}[1]{\left(#1\right)}
\newcommand{\lrb}[1]{\left[#1\right]}
\newcommand{\lrc}[1]{\left\{#1\right\}}

\newcommand{\Z}{\mathbb{Z}}
\newcommand{\T}{\mathbb{T}}
\newcommand{\R}{\mathbb{R}}
\newcommand{\cC}{\mathcal{C}}
\newcommand{\fF}{\mathfrak{F}}
\newcommand{\cJ}{\mathcal{J}}

\newcommand{\cB}{\mathcal{B}}

\newcommand{\cE}{\mathcal{E}}
\newcommand{\cF}{\mathcal{F}}
\newcommand{\cP}{\mathcal{P}}

\newcommand{\cG}{\mathcal{G}}
\newcommand{\cS}{\mathcal{S}}
\newcommand{\cT}{\mathcal{T}}

\newcommand{\bT}{\mathbb{T}}

\newcommand{\rs}{\mathrm{s}}
\newcommand{\rt}{\mathrm{t}}

\newcommand{\diam}{\mathrm{diam}}

\newcommand{\dbar}[1]{\overline{\overline{#1}}}
\renewcommand{\bar}[1]{\overline{#1}}
\newcommand{\mix}{\mathrm{mix}}
\newcommand{\rel}{\mathrm{rel}}
\newcommand{\core}{\mathrm{core}}
\mathchardef\mhyphen="2D
\newcommand{\score}{\mathrm{s\mhyphen core}}
\newcommand{\inn}{\mathrm{inn}}
\newcommand{\defn}{=}
\newcommand{\rmin}{\mathrm{min}}
\newcommand{\rmax}{\mathrm{max}}
\newcommand{\cyl}{\mathrm{cyl}}
\newcommand{\qand}{\quad\text{and}\quad}

\newcommand{\TV}{\mathrm{TV}}
\newcommand{\crit}{\mathrm{c}}
\newcommand{\compl}{\mathrm{c}}
\newcommand{\enl}{\mathrm{enl1}}
\newcommand{\enlb}{\mathrm{enl2}}
\newcommand{\benl}{{\partial\enl}}
\newcommand{\SRWM}[2]{\mathbb{I}_{(#1,#2)}^{{\scalebox{0.7}{SRWM}}}}
\newcommand{\coup}{\mathrm{coup}}

\newcommand{\Mid}{\,\middle\vert\,}

\renewcommand{\phi}{\varphi}
\renewcommand{\hat}{\widehat} 

\begin{document}
\newcounter{cte}
\setcounter{cte}{0}
\newcommand{\newcte}[1]{{\refstepcounter{cte}\arabic{cte}\label{#1}}}
\newcommand{\cref}[1]{{\ref*{#1}}}
\newcommand{\newcteH}[1]{{\refstepcounter{cte}\label{#1}}}

\title{Mixing time of random walk on dynamical random cluster}

\author{Andrea Lelli\footnote{Most of this work was done while the author was affiliated with the University of Bath, Department of Mathematical Sciences, supported by a scholarship from the EPSRC Centre for Doctoral Training in Statistical Applied Mathematics at Bath (SAMBa), under the project EP/L015684/1.} \and Alexandre Stauffer\footnote{a.stauffer@bath.ac.uk, University of Bath, Department of Mathematical Sciences, supported by EPSRC Fellowship EP/N004566/1.}}

\maketitle
\begin{abstract}
   We study the mixing time of a random walker who moves inside a dynamical random cluster model on the $d$-dimensional torus of side-length $n$. 
   In this model, edges switch at rate $\mu$ between \emph{open} and \emph{closed}, following a Glauber dynamics for the random cluster 
   model with parameters $p,q$. At the same time, the walker jumps at rate $1$ as a simple random walk on the torus, but is only allowed to traverse open edges. 
   We show that for small enough $p$ the mixing time of the random walker is of order $n^2/\mu$.
   In our proof we construct of a non-Markovian coupling through a multi-scale analysis of the environment, which we believe could be more widely applicable. 
\end{abstract}

\section{Introduction}
We study the mixing time of a random walk on a dynamical random cluster model in $\bT_n^d$, the 
$d$-dimensional torus of side-length $n$. 
In this model, each edge of $\bT_n^d$ can be in either of two states: \emph{open} or \emph{closed}. 
At time $0$, we take the state of the edges to be distributed according to the 
random cluster measure with parameters $p\in(0,1)$ and $q>0$. That is, for any subset of edges $\omega\subset E(\bT_n^d)$, 
with $E(\bT_n^d)$ denoting the set of edges of the torus, the probability that the set of open edges at time $0$ is equal to $\omega$ is
\begin{equation}
   \upsilon(\omega)\defn\frac{1}{Z}p^{|\omega|}(1-p)^{|E(\bT_n^d)\setminus \omega|}q^{\kappa(\omega)},
   \label{eq:rc}
\end{equation}
where $\kappa(\omega)$ is the number of connected components obtained in the graph with vertex set $\bT_n^d$ and edge set $\omega$, and $Z=Z(d,p,q)>0$ is just a normalizing constant so that the above is 
a probability measure. Instead of representing the state of the edges by the set $\omega$ of open edges, 
we will often represent it by an element $\eta\in \lrc{0,1}^{E(\bT_n^d)}$, with $\eta(e)=0$ meaning that the edge $e\in E(\bT_n^d)$ is closed 
and $\eta(e)=1$ meaning that $e$ is open. Thus, given $\eta$, we have $\omega=\lrc{e\in E(\bT_n^d) \colon \eta(e)=1}$.

From time $0$, edges change their state following a continuous-time Glauber dynamics. 
Thus, given a parameter $\mu>0$, each edge $e\in E(\bT_n^d)$ has a Poisson clock of rate $\mu$, and when the clock of $e$ rings, 
the state of $e$ is resampled (open or closed) according 
to the conditional probability obtained from $\upsilon$ in~\eqref{eq:rc} conditioned on the states of all the other edges.
This resampling can be easily described: if the clock of $e$ rings at time $t$, then the probability that $e$ becomes open at time $t$ is equal to 
\begin{equation}
   \begin{array}{rl}
   p, &  \text{ if $e$ is not a \emph{cut-edge} at time $t-$},\\
   \frac{p}{p+(1-p)q}, & \text{ if $e$ is a \emph{cut-edge} at time $t-$},
   \end{array}\label{eq:rcupdate}
\end{equation}
where an edge $e$ is called a \emph{cut-edge} if modifying the state of $e$ (while keeping the state of the other edges unaltered) 
causes a change in the number of connected components in the configuration. Note that whether an edge $e$ is a cut-edge for a configuration $\eta$ is, in fact, independent of 
$\eta(e)$.
We let $\eta_t\in \lrc{0,1}^{E(\bT_n^d)}$ denote the configuration that gives the state of the edges at time $t$.

On top of this dynamic environment we place a random walker which starts from the origin of $\bT_n^d$ and has a Poisson clock of rate $1$. When the clock of the walker rings, the walker chooses  
an edge uniformly at random from the set of edges that are adjacent to its current location, regardless of their states. 
If the chosen edge is open at that time, then the walker traverses the edge, otherwise the walker stays put.
We denote by $X_t\in \bT_n^d$ the position of the walker at time $t$, and let 
\begin{equation*}
    \{M_t\}_{t\ge0} \defn \{X_t, \eta_t\}_{t\ge0},
\end{equation*}
denote the \emph{full system} composed of the walker $\lrc{X_t}_{t\geq 0}$ and the environment $\lrc{\eta_t}_{t\geq 0}$.
We note that $\lrc{M_t}_{t\geq 0}$ and $\lrc{\eta_t}_{t\geq 0}$ are Markov chains, while 
$\lrc{X_t}_{t\geq 0}$ is not.

One can check (for example, by reversibility) 
that if $\pi$ denotes the uniform probability measure on $\bT_n^d$ 
then $\pi \times \upsilon$ is the unique stationary distribution of $\lrc{M_t}_t$. 

Let $T_\mix$ denote the mixing time of the full system, starting from the worst-case initial state. 
In other words, given $x\in \T_n^d$ and $\xi\in\lrc{0,1}^{E(\T_n^d)}$, let 
$T_\mix^{x,\xi}$ be the smallest $t$ such that, starting from $M_0=(x,\xi)$, the total variation distance between the 
distribution of $M_t$ and $\pi \times \upsilon$ is smaller than a given constant, which for concreteness we take to be $1/4$.
Then $T_\mix = \min_{x,\xi}T_\mix^{x,\xi}$.

Our main result establishes that the mixing time is of order $n^2/\mu$ for all small enough $p$. 
We remark that $p$ and $q$ are considered to be constants independent of $n$, 
while $\mu$ may depend on $n$; in particular, a natural case in the context of dynamic networks is that $\mu\to 0$ as $n\to\infty$.
\begin{theorem}
   \label{thm:main}
   Given any $q>0$ and any dimension $d\geq 1$, there exists a positive $p_0>0$ so that for all $p\in (0,p_0)$ there exists $C_\newcte{cte:thmmain}=C_{\ref*{cte:thmmain}}(d,p,q)>0$ for which
   $$
       T_\mix\le C_{\ref*{cte:thmmain}}\frac{n^2}{\mu},\quad \text{for all $\mu=\mu(n)>0$ and all $n\geq 1$}.
   $$
\end{theorem}

The proof of Theorem~\ref{thm:main} goes via the construction of a \emph{non-Markovian} coupling using a multi-scale analysis of the environment. We believe this idea can be more widely applicable to analyze the 
mixing time of random walks on particle systems, and we regard it as one of our main contributions. 
Another main contribution of our work is to initiate the analysis of the mixing time of a random walk in a dynamic environment where edge updates are not 
independent of one another; see the related works in Section~\ref{sec:relwork}. 
We will employ a multi-scale analysis exactly to control the evolution of the environment. 
We will give a thorough description of the main ideas of the proof in Section~\ref{sec:pf}, since first, in Section~\ref{sec:spaceext}, 
we will need to introduce an auxiliary process. 

\subsection{Lower bounds on the mixing time}\label{sec:sublb}
We also derive matching lower bounds on the mixing time. 
We start by stating a straightforward generalization of the lower bound from~\cite{peres2015random}. 

We consider a larger class of models, which we refer to 
as \emph{continuous-time random walks on general dynamical percolation}, where the word general is to mean that the percolation process may not be independent. 
Let $\lrc{X_t,\eta_t}_{t\geq 0}$ be a continuous-time Markov chain where the walker $X_t$ jumps at rate $1$ and can only traverse open edges of $\T_n^d$, and the environment $\lrc{\eta_t}_t$ is a Markov chain on 
$\lrc{0,1}^{E(\T_n^d)}$ where edges refresh their state at rate $\mu$ independently of the walker. As usual $\mu$ may depend on $n$. 
Let $\pi$ be the uniform distribution on $\T_n^d$ and let $\nu$ be the stationary distribution
of the Markov chain $\lrc{\eta_t}_t$. 

We recall some fundamental definitions. The \emph{spectral gap} $\gamma$ of a reversible Markov chain is defined as 
$$
   \gamma\defn \inf_{f} \frac{\cE(f,f)}{\Var(f)},
$$
where the infimum is over all functions $f$ from the state space to $\R$ with $\Var(f)\neq 0$, the variance $\Var(f)$ being with respect to
the stationary distribution of the chain, and $\cE(f,f)$ is the so-called Dirichlet form. The \emph{relaxation time} of the said Markov chain is defined as 
$$
   T_\rel = \gamma^{-1}.
$$ 

Given a time interval $I\subset \R_+$, we say that an edge is $I$-open, if it is open at some time during $I$. 
Then, for any vertex $x\in\T_n^d$ and any time interval $I\subset \R_+$, we let $\cC_x(I)$ denote the connected component of $I$-open edges from $x$. Finally, given a subset $S\subset\T_n^d$, let 
$\diam(S)=\max_{x,y\in S}\|x-y\|_1$ be the diameter of $S$, where $\|x-y\|_1$ is the $L_1$ distance (or, equivalently, the length of the shortest path) between $x$ and $y$ in $\T_n^d$.
We require the following two assumptions from the process $\lrc{X_t,\eta_t}_{t\geq 0}$:
\begin{equation}
   \pi \times \nu \text{ is the stationary distribution of $\lrc{X_t,\eta_t}_{t\geq 0}$}, 
   \label{eq:assump1}
\end{equation}
and \newcteH{cte:assump2}
\begin{equation}
   \text{$\exists \delta>0$ and $C_{\ref*{cte:assump2}}>0$ such that for any $x\in\T_n^d$ we have $\E_\nu\lr{D_{x,\delta}^2}\leq C_{\ref*{cte:assump2}}$},
   \label{eq:assump2}
\end{equation}
where $D_{x,\delta}=\diam\lr{\cC_x\lr{[0,\delta]}}$ and $\E_\nu$ denotes the expectation with respect to the stationary measure of the environment.
The assumption in~\eqref{eq:assump1} just says that the stationary measure of the walker is uniform, while~\eqref{eq:assump2} gives that the environment is strictly subcritical.
\begin{theorem}
   \label{thm:lb}
   Let $\lrc{X_t,\eta_t}_{t\geq 0}$ be a random walk in a general dynamical percolation satisfying~\eqref{eq:assump1} and~\eqref{eq:assump2} above. 
   Then, there exist a constant $C_\newcte{cte:lb}>0$ depending only on $d$ such that 
   $$
      T_\rel \geq \frac{C_{\ref*{cte:lb}}\delta n^2}{C_\cref{cte:assump2}}.
   $$
\end{theorem}

A natural setting is when the environment starts from its stationary distribution. 
For this, let $\upsilon_t$ stand for the distribution of $(X_t,\eta_t)$ where the walker starts from the origin and the environment starts from stationarity (that is, $\eta_0$ is distributed as $\nu$).
Then, $\|\upsilon_t - \pi\times\nu\|_\TV$ is the total variation distance between $\upsilon_t$ and the stationary measure of $\lrc{X_t,\eta_t}_t$.
\begin{theorem}
   \label{thm:lb2}
   Let $\lrc{X_t,\eta_t}_{t\geq 0}$ be a random walk in a general dynamical percolation satisfying~\eqref{eq:assump1} and~\eqref{eq:assump2} above. 
   Then, there exists a constant $C_\newcte{cte:lb2}>0$ depending only on $d$ such that, for any $\epsilon>0$, 
   $$
      \text{if $t\leq \tfrac{C_\cref{cte:lb2}}{C_\cref{cte:assump2}}\epsilon^{\frac{d+2}{d}}\delta n^2$ then $\|\upsilon_t - \pi\times\nu\|_\TV\geq 1-\epsilon$.}
   $$
    Moreover, there exists a constant $C_\newcte{cte:sqd}>0$ depending only on $d$ such that, for any $t\geq \delta$, 
   $$
      \E_{\upsilon_t}\lr{\|X_t-X_0\|_1^2}\leq C_\cref{cte:sqd}C_\cref{cte:assump2}\frac{t}{\delta},
   $$
   where $\E_{\upsilon_t}$ stands for the expectation with respect to $\upsilon_t$.
\end{theorem}

The proofs of Theorems~\ref{thm:lb} and~\ref{thm:lb2} are identical to the ones for random walk on dynamical percolation from~\cite{peres2015random}. For the sake of completeness, we add the proofs
in Section~\ref{prooflb}.

We want to apply the above theorems to derive lower bounds on the mixing time of a random walk in dynamical random cluster. 
It is clear that~\eqref{eq:assump1} holds in this case. 
We will show in Section~\ref{sec:cor} that~\eqref{eq:assump2} also holds, obtaining the corollary below.
For any $q$, let $p_\crit^q$ be the critical probability for the appearance of an infinite cluster in the random cluster model on $\Z^d$. 
\begin{corollary}
   \label{cor:lb}
   If $\lrc{X_t,\eta_t}_{t\geq 0}$ is a random walker in the dynamic random cluster model, then for any $q\geq 1$ and any $p<p_\crit^q$, there exists a constant $c=c(d,q,p)>0$ such that 
   the relaxation time of the full system and the mixing time starting from a stationary environment is at least $c n^2/\mu$. If $q<1$, then for all small enough $p$ the same conclusion holds.
\end{corollary}

The proof of the lower bound is much simpler than that of the upper bound, allowing us to derive it in the whole subcritical regime when $q\geq 1$. 
In fact, when $q\geq 1$, the proof follows by using a sprinkling lemma to compare two random clusters configurations with densities $p<p'$ (Lemma~\ref{lem:sprinkle}), and the exponential decay 
of cluster sizes in the subcritical regime. 
When $q<1$, exponential decay of cluster sizes is only known for small enough $p$, preventing us to establish~\eqref{eq:assump2} in the whole subcritical regime.

We expect the upper bound of order $n^2/\mu$ to hold in the whole subcritical regime as well, however our proof technique requires the percolation process to be a small perturbation of subcritical independent percolation, 
in a sense that we better explain in Remark~\ref{rem:whysub}, after introducing the $\star$-process.

\subsection{Related works}\label{sec:relwork}
We will restrict our discussion to works dealing with the mixing time of random walks on \emph{dynamic} environments, 
as otherwise there is simply a plethora of works. 
We also remark that, if the environment is allowed to evolve in an arbitrary fashion (for example, by taking any sequence of graphs on a fixed 
vertex set), then several problems may arise. For example, there may not be a stationary distribution for the walker. Moreover, even if there is a stationary distribution, 
the distribution of the walker may not converge to stationarity, 
or the total variation distance to stationarity may not be monotone in time.

\emph{Random walk on dynamical percolation on $\bT_n^d$.} This model is equivalent to the model we described restricted to $q=1$. 
This special case is already quite challenging but some results have been obtained recently. 
First note that, when $q=1$, the two probabilities in~\eqref{eq:rcupdate} become equal, and when an edge updates, it does so \emph{independently} of the other edges, becoming open 
with probability $p$ or closed with probability $1-p$. 
Though in this case edges evolve independently of one another, there are strong dependences between the location of the walker and the state of the edges (especially if $\mu\to 0$ as $n\to\infty$, 
since edges update very slowly in comparison to the rate of jump of the walker).

The random walk on dynamical percolation model was introduced by Peres, Stauffer and Steif~\cite{peres2015random}, where it is shown that, 
in the whole subcritical regime\footnote{%
   That is, for any $p<p_\crit$ with $p_\crit=p_\crit(d)$ being the critical 
   probability for the existence of an infinite cluster in percolation in $\Z^d$}, 
the mixing time is of order $n^2/\mu$. 
We remark that in~\cite{peres2015random} both upper and lower bounds of order $n^2/\mu$ were derived for $T_\mix$. 
Recall that $n^2$ is the order of the mixing time of a simple random walk on the \emph{static} torus (that is, where all edges are open at all times). 
So, in a subcritical dynamical percolation, the walker is delayed by a factor of $1/\mu$, which is the 
expected time that a single edge takes to refresh. 

Later, Peres, Sousi and Steif~\cite{peresSousiSteif} analyzed the supercritical regime and showed that, for $p$ large enough, the mixing time 
is at most $(\log n)^a \lr{n^2+\frac{1}{\mu}}$ for some constant $a>0$. Their upper bound is not believed to be tight: 
one expects that, in the whole supercritical regime, the mixing time is of order $n^2+\frac{1}{\mu}$. This remains an interesting open problem.
Their proof makes strong use of isoperimetric properties of the infinite cluster of supercritical percolation, which are only known for $q=1$.  
With regard to the \emph{critical} regime, the only known result is that the mixing time is of order at most $\frac{n^2}{\mu}$, which is the mixing time in the subcritical regime~\cite{hermonSousi}. 
It is not inconceivable that the mixing time in the critical case is in fact of smaller order than $\frac{n^2}{\mu}$.

\emph{Random walk on dynamical percolation on other graphs.}
Sousi and Thomas~\cite{sousiThomas} studied the case where the torus is replaced by the \emph{complete graph}.
This is a simpler case due to the lack of an underlying geometry, but for which a more detailed analysis can be carried out.  
They established the order of the mixing time in that case, and also the occurrence of a cut-off phenomenon.
We remark that if the walker is at some vertex $v$ and we know that an edge incident to $v$ is updating to open, but we refrain from observing which of the edges incident to $v$ is updating, 
then the other endpoint of this edge is uniformly at random
among all vertices (but $v$). So, by traversing this edge (call it $e$), after one additional step, the walker can find itself in a location that is essentially uniformly at random, so very close to stationarity. 
Though suggestive, this is not enough to establish the mixing time, as one still needs to control that the walker ``forgets'' that $e$ is now open (that is, the walker may be close to stationarity, but 
the full system is not). Still, this illustrates the kind of simplification that the 
lack of an underlying geometry brings.

The last work we mention for the random walk on dynamical percolation model is a recent result by Hermon and Sousi~\cite{hermonSousi}. They developed a comparison principle and showed that, 
for \emph{any} graph $G$, 
the so-called \emph{spectral profile mixing time} for the random walk on dynamical percolation on $G$ 
is at most $\frac{1}{\mu}$ times the spectral profile mixing time of simple random walk on (the static graph) $G$. 

In all the above results, it was crucial that when $q=1$ edges update independently of one another. 
The main objective of our work is to develop a technique that can go beyond the dynamical percolation case and which 
can deal with environments whose edge updates may depend on one another, including the case of unbounded dependences such as in the dynamical random cluster.

\emph{Other models.} We end this section by mentioning two lines of work.
In the first one, Avena et al.\ \cite{avenaGuldasHofstadHollander,avenaGuldasHofstadHollander2} studied a different dynamic on the environment, where instead of dynamical percolation 
one has a \emph{dynamic configuration model}. 
This model has some intuitive similarities with the dynamical percolation on the \emph{complete graph}, in the sense that it also lacks an underlying geometry. They studied the mixing time and the occurrence of 
a cut-off phenomenon in this setting, but restricted to a random walker that is non-backtracking. This helps the walker to move away from its current location, strongly reducing dependences between the walker and the environment. 

Finally, the second line of work we mention is that of~\cite{sauerwaldZanetti,Shimizu2021Feb}. They considered the case of a discrete-time random walk on a graph with a fixed set of vertices, but which evolves
over time by means of an arbitrary sequence of graphs on that vertex set. The goal of their work is much different than ours; for example, they want to understand which conditions on the sequence of graphs one 
can impose to guarantee that the mixing time is polynomial. They also derive results for the hitting time and cover time. We refer to~\cite{sauerwaldZanetti,Shimizu2021Feb} and 
references therein for a list of known results on dynamic graphs that go beyond the mixing time. We also refer the reader to~\cite{Cai2020Jul} for results on a model similar to random walks on dynamical percolation on the 
complete graph.

\subsection{The $\star$-process: retaining some randomness}
\label{sec:spaceext}
Before giving the ideas of our proof, we need to describe a different representation of the full system, which is inspired by~\cite{peres2015random}. 
Recall that each edge has a Poisson clock of rate $\mu$ associated to it, which gives the times at which the edge is updated. 
To each update of an edge, we can decide whether the edge becomes open or closed by sampling an independent random variable $U$ with a uniform distribution in $(0,1)$, and then 
making the edge open if and only if the edge is not a cut-edge and $U<p$ or the edge is a cut-edge and $U<\frac{p}{p+(1-p)q}$. Now, let 
$$
   p_{\rmin} \defn \min\lrc{p,\frac{p}{p+(1-p)q}}
   \qand
   p_{\rmax} \defn \max\lrc{p,\frac{p}{p+(1-p)q}};
$$
thus, $p_\rmin= \frac{p}{p+(1-p)q}$ if $q>1$ and $p_\rmin=p$ if $q<1$.
Note that if $U$ turns out to be in the interval $(0,p_\rmin)\cup(p_\rmax,1)$ the outcome of the update (i.e., whether the edge becomes open or closed) 
is determined regardless of whether or not the edge is a cut-edge. In other words, the update is \emph{oblivious} to the current configuration, and we will refer to those updates as \emph{$\star$-updates}.
We then let 
\begin{equation*}
   p_\star\defn p_{\rmin}+1-p_\rmax \in (0,1)
\end{equation*}
be the probability that a given update is a $\star$-update. 

We now define the update of an edge $e$ in two stages. First, we sample an independent random variable $U_\star$, which is uniformly at random in $(0,1)$, so that 
if $U_\star< p_\star$, then the update is a $\star$-update, otherwise it is not a $\star$-update. 
Next, we use the random variable $U$ to determine whether $e$ updates to open or closed. 
More precisely, in the case of a $\star$-update, we make $e$ open if $U<\frac{p_{\min}}{p_\star}$, otherwise $e$ becomes closed. 
In the case of a non-$\star$-update, we need to inspect the current configuration to see whether $e$ is a cut-edge or not. In particular, we need to perform what we call an 
\emph{exploration of $e$}, which means that we perform a local search from the endpoints of $e$ that traverses only open edges in order to determine what are the open clusters of the endpoints of $e$. 
Hence, each update of $e$ will be represented by a tuple $(s,U_\star,U)$, 
where $s>0$ is the time at which the update occurs, $U_\star\in(0,1)$ is the variable used to decide whether 
the update is a $\star$-update, and $U\in(0,1)$ is the random variable governing whether the edge is to be updated open or closed. 

We use this to introduce another Markov process which we denote by $\{{M}^\star_t\}_{t\ge0} = \lrc{X_t, {\eta}_t^\star}_{t\ge0}$, and which we refer to as the 
\emph{$\star$-process}. 
This process will retain more randomness than $\{M_t\}_{t\ge0}$ 
and its state space will be 
\begin{equation*}
    \Omega^{\star} \defn \lrc{(v, {\eta^\star}) \in \bT_n^d \times \{0, 1, \star\}^{E(\bT_n^d)} \colon {\eta^\star}(e) \in \{0, 1\}\text{ for each edge }e\text{ adjacent to }v}.
\end{equation*}    
So an edge will be allowed to be in an additional state, called $\star$, which means that in its last update the edge underwent a $\star$-update. 
However, we do not allow that edges adjacent to the walker are in state $\star$. 

The $\star$-process evolves as follows.
If the Poisson clock of an edge $e$ rings, we look at the variable $U_\star$ associated with this update and determine whether the update is a $\star$-update. If the update is a $\star$-update and if $e$ is not 
currently adjacent to the walker, then we make the state of $e$ equal to $\star$. If $e$ is adjacent to the walker, then we look at the variable $U$ associated with this update and determine whether $e$ is open or closed.
Finally, if the update is not a $\star$-update, then we perform an exploration of $e$ as mentioned above. The difference is that, in such an exploration, we may run into edges that are in state $\star$. 
For each such edge, we immediately sample whether that edge is open or closed by using the random variable $U$ associated with 
its last update. We proceed in this way until the exploration ends and we have fully determined the cluster of each endpoint of $e$. 
At this moment, we know whether or not $e$ is a cut-edge, and we can use the random variable $U$ associated to the update of $e$ to determine whether
$e$ is to be made open or closed. 
There is still one final case to be described: when it is the clock of the walker that rings. Suppose this happens and the walker jumps from a vertex $v$ to a vertex $w$. 
Then, if there are edges adjacent to $w$ at state $\star$ we sample the state of such edges (using the random variables 
$U$ associated to their last update) and switch them to open or closed, appropriately. 

Note that, conditioned on the position of the walker and on the state $0$, $1$ or $\star$ of each edge, we gain no information concerning whether the edges in state $\star$ are open or closed. 
In particular, each such edge is open with probability $\frac{p_\rmin}{p_\star}$ (which is the probability that the random variable $U$ associated to their last update is at most 
$\frac{p_\rmin}{p_\star}$). Therefore, we do not need to keep track of the variables $U$ related to the last $\star$-update of each edge, since we can sample $U$ whenever needed independently of 
the whole trajectory of the process. The $\star$-process is thus a Markov process.

\begin{remark}\label{rem:whysub}
   When $q=1$, we have $p_\rmin=p_\rmax$, and so $p_\star=1$: all updates are $\star$-updates, as in this case the random cluster model reduces to dynamical percolation.
   If $q\neq 1$, then as $p\to 0$ we have that $p_\rmax-p_\rmin\to 0$ and so $p_\star\to1$. 
   Therefore, for any fixed $q$ and all small enough $p$, the dynamic random cluster model can be viewed as a small perturbation of dynamical percolation. 
   We also obtain that edges of state $\star$ are open with probability $\frac{p_\rmin}{p_\star}<p_\crit$, so they form a subcritical percolation process as well.
   Those are the properties that play an essential role in the constructions of the multi-scale analysis and the coupling used to establish the upper bound on the mixing time (Theorem~\ref{thm:main}).
\end{remark}

\subsection{Proof overview}\label{sec:pf}
We will only give an overview of the upper bound, which is our main result and by far the most involved proof. 
We start recalling the proof in~\cite{peres2015random} for the subcritical regime when $q=1$. 
There they also define the $\star$-process (which they denote by 
$\tilde M_t$). Recall that, when $q=1$, we have $p_\star=1$, so all updates are $\star$-updates. With this, they define a stopping time $\tau_0$ as the first time at which 
\begin{equation}
   \text{all edges adjacent to the walker are closed, and all remaining edges are in state $\star$.}
   \label{eq:srwmintro}
\end{equation}
Then, one can define a sequence of times $\tau_1,\tau_2,\ldots$ so that $\tau_i$ is the first time 
after $\tau_{i-1}+C/\mu$, for some fixed constant $C>0$, at which the event in~\eqref{eq:srwmintro} happens. These are regeneration times in the sense that the evolution of the full system from $\tau_i$ does not 
depend on what happened before $\tau_i$. Once the full system is at a regeneration time $\tau_i$, with positive probability the following sequence of events happen within time $\tau_i+C/\mu$: 
\begin{enumerate}[(i)]
   \item an edge $e$ adjacent to the walker opens
   \item when the walker jumps to the other endpoint of $e$, all the adjacent edges (which are in state $\star$) are sampled closed
   \item $e$ remains open for some time of order $1/\mu$
   \item $e$ closes before any of the other edges adjacent to $e$ open, thereby locking the walker in one of $e$'s endpoints, and 
   \item the edges adjacent to the other endpoint of $e$ (i.e., opposite to the location of the walker) refresh before the edges adjacent to the walker refresh.
\end{enumerate}
When these events occur, the walker does nothing more than a jump to a uniformly random neighbor, and immediately gets back to a regeneration time (so $\tau_{i+1}=\tau_i+C/\mu$); 
such a regeneration time is then called a \emph{simple random walk regeneration} since, at the end, what the walker did was just one step of a simple random walk in $\bT_n^d$. 

The proof in~\cite{peres2015random} then goes by showing that the
$\tau_{i+1}-\tau_i$ are of order $\frac{1}{\mu}$. Therefore, after time $\frac{n^2}{\mu}$, the walker underwent an order of $n^2$ regeneration times, a positive fraction of which being simple random walk regeneration.
So it is possible to couple the full system with another copy of the full system so that, whenever the walker does a simple random walk regeneration, we employ one of the standard couplings of simple random walks on the torus.
On the other hand, if the regeneration time is not a simple random walk regeneration, we couple the motion of the two walkers from one regeneration time to the next identically, so that the distance between the walkers does not 
change. Since an order of $n^2$ steps is necessary to 
couple two simple random walks on $\bT_n^d$, we get that performing an order of $n^2$ simple random walk regenerations is enough to couple the two processes, which translates to a mixing time of order $n^2/\mu$.

If we try to mimic the steps above for the case $q\neq 1$, 
we immediately run into the issue that the event~\eqref{eq:srwmintro} now occurs very rarely. In fact, since non-$\star$-updates occur with positive probability,
we will typically have a positive density of non-$\star$-edges. 
Therefore, it will take an exponential amount of time to reach a regeneration time as in~\eqref{eq:srwmintro}, rendering this strategy useless. 

We will devise a different strategy. 
We will, as before, construct a coupling between two copies of the full-system, where we see the edges ``from the point of view of the walker'' in the sense that whenever the edge 
$X_t+e$ updates at time $t$, 
where $X_t$ is the position of the walker in the first copy, 
then in the second copy we will do the same update to the edge $\bar{X}_t+e$, where $\bar{X}_t$ is the location of the walker in the second 
copy. Note that to establish the mixing time of the full system  we need to couple the environments and the walkers. 
For simplicity, we concentrate our discussion here in the coupling of the walkers (which is the most delicate bit), and assume for now that somehow we managed to couple the two environments: that is, 
the two copies are coupled modulo a translation of the walkers. 
Note that, from this moment, if we were to employ the \emph{identity coupling} (that is, the second copy mimics all the edge updates and jumps of the walker from the first copy) 
we would get that the environments will remain coupled (from the point of view of the walkers) but the distance 
between the walkers will not change, thereby not allowing the walkers to couple. 

Our idea is to observe ``a bit'' the environment and, whenever the environment looks ``favorable enough'', 
we attempt to do a coupling that could bring the walkers closer together, which will be a standard coupling of simple random walks. 
We will refer to such moments as \emph{simple random walk moments}, as an allusion to the simple random walk regenerations described above, but with the fundamental difference that they will not be regeneration times.
On the other hand, when the environment is not favorable enough, then doing a simple random walk moment is a bit too risky, 
so instead we resort to the identity coupling as a means to keeping the distance between the walkers unchanged and 
not spoiling the work done during the favorable regions of the environments.

But what does it mean for the environment to look favorable enough? In short terms, it will mean that the event~\eqref{eq:srwmintro} occurs \emph{locally}. That is, 
at such times, all edges adjacent to the walkers will be closed and all edges in a small region around the walkers will be $\star$ (for example, all edges inside a ball of radius 3 around the walkers, excluding the 
edges adjacent to the walkers). At such a time, with positive probability, the sequence of events described above for the simple random walk regeneration occurs, and therefore 
we could attemp to perform one of the standard couplings of simple 
random walks. However, there are two important caveats. 

The first caveat is that if we succeed in doing a simple random walk moment with a coupling of simple random walks, then the distance between the walkers will change. 
This means that the translation that maps the location of one walker to the location of the second walker will change, and this map is what we use to match the edges of the first copy to the edges of the second copy, when we view 
the edges from the point of view of the walkers.
As a consequence, the environments will immediately \emph{decouple}. Of course, if we only had $\star$-edges (besides the ones adjacent to the walkers, as in the case $q=1$), 
then the environments would not decouple since despite the change in the translation map, we would still match $\star$-edges in the first copy to $\star$-edges in the second copy, 
so we can easily maintain the environments coupled. 
But, since $q\neq 1$ implies a density of non-$\star$ edges, the environments will necessarily decouple. Moreover, if we decide to just wait the environments to recouple completely, this would take a time of order 
$\frac{\log n}{\mu}$, which is just too long: it will lead to an upper bound on the mixing time of $\frac{n^2\log n}{\mu}$. So we will not recouple the environments completely, but will work with partially coupled 
environments.

The second caveat is that a simple random walk moment occurs with \emph{positive} probability, so it is also possible that it turns out that a simple random walk moment does not take place. 
Then, what could happen in this case? 
If the environments were completely coupled, then we are guaranteed that we can perform identity coupling and keep the distance between the walkers unchanged. 
But we have just seen that the environments will typically not be fully coupled. Yet, if we knew that the environments are coupled in a neighborhood around the walkers and that the walkers will not exit this neighborhood,
then identity coupling is still doable. That will be our strategy, but to implement it we will require a more delicate definition of what a favorable enough environment means. 

We will use a multi-scale analysis to control the environment. 
This will reveal \emph{future information} regarding the environment; that is, we will observe some information about the environment from time $0$ to some time $t$, and then decide how to couple the walkers from 
time $0$. Therefore, this construction will lead to a non-Markovian coupling. 

A good picture to have in mind is that the environment is a process in space-time, where some regions are classified as favorable and others as unfavorable. We observe these regions from time $0$ to time $t$, and 
then start observing the walkers which are paths in space-time that start growing from time $0$.
Whenever we see that the walkers are passing through a favorable part of the environment, where favorable will also imply 
that the walkers will not move outside some neighborhood around their current locations, we will try to do a simple random walk moment. If successful, 
the distance between the walkers may change and the environments may decouple, but still using (the yet-to-be-defined properties of) favorability we 
will be able to recouple the environments within a neighborhood around the walkers. 
If, instead, the simple random walk moment is not successful, then the walkers may move more than just one step of a simple random walk, 
but favorability will also imply that the walkers will not move too far away, in 
particular they will remain within a region where we know the environments were coupled. This will 
translate to a successful application of the identity coupling. 

On the other hand, if we see that the walkers are approaching an unfavorable region of the environment, then we will want to do identity coupling but we will need to start preparing ourselves beforehand. 
The problem is that such an unfavorable region could be of an arbitrarily large scale, 
and the larger its size is, the earlier we need to start preparing for it. So when we see that in space-time the path of the walker is getting dangerously near an unfavorable region, 
we stop doing simple random walk moments even if in a smaller scale around the walkers the environment looks favorable. By switching off the simple random walk moments, 
we only apply identity coupling until the walkers reach the unfavorable region or are again far enough from any unfavourable region.  
We can show that such identity couplings will succeed and, since the translation map 
from one walker to the next will not change during this period, it will give enough time for the environments to couple in a region around the walkers that is as large as needed to contain the scale of the 
unfavorable region that the walkers are approaching. Then, with the environments properly coupled, 
if the walkers do enter the unfavorable region, they can move as wildly as the environment there allows because we can perform identity coupling throughout the unfavorable region. 
So the walkers survive the traversal of the unfavorable region without changing their distance. 

Then one can imagine that the proof ends by showing that $n^2$ instances of a simple random walk moment are enough to guarantee that we can couple the walkers. This is partially true. The fact 
is that, as mentioned above, we need to observe future information to carry out this coupling strategy. 
But in order to establish that the mixing time is at most $t$, we need to show that with 
a large enough probability the two copies of the full system are coupled at time $t$ without revealing any information that goes beyond time $t$. 
So our strategy to finalize the proof is to choose an appropriate time $t'\in(0,t)$, 
reveal the information up to time $t$ and do the coupling described above up to time $t'$, showing that within $t'$ we have carried out an order of $n^2$ simple random walk moments, and that we coupled the walkers 
at time $t'$ (the environments may, and typically will, be uncoupled except for a small region around the walkers). 
The whole analysis will be split into three phases, and the above will be carried out in the first two phases. We will be able to show that these first two phase succeed with positive probability.

Next, the goal is to try to do identity coupling from time $t'$ to $t$ in a similar manner as 
we were doing when approaching an unfavorable region. In this second phase, identity coupling can only fail due to information that we have not observed because we are limited to observe the environment up to time $t$.
We will show that, with positive probability, identity coupling will indeed succeed from $t'$ to $t$, leading to a coupling of the full system at time $t$. This is the content of the third phase.
If any of the three phases fail, then we just restart from scratch. 
We only need to repeat the phases a constant number of times to guarantee that the whole coupling succeeds with probability at least $3/4$. 

\subsection{Organization of the paper}
In Section~\ref{tessellation} we will introduce the multi-scale analysis that will allow to control the favorable regions of the environment.
Then in Section~\ref{sec:overview} we will give an more thorough overview of the three phases of the proof of the upper bound, which will better explain the constructions from the tessellation
of Section~\ref{tessellation}. Then in Sections~\ref{firstphase},~\ref{secondphase} and~\ref{thirdphase} we will give the three phases of the coupling, with the second phase in Section~\ref{secondphase} being the 
most delicate part where the non-Markovian coupling is developed. Then in Section~\ref{proofteo} we put all phases together to complete the proof of the upper bound (Theorem~\ref{thm:main}).
In Section~\ref{prooflb} we establish the general lower bounds from Theorems~\ref{thm:lb} and~\ref{thm:lb2}, but which are essentially the same proofs as in~\cite{peres2015random}; this section is added for the 
sake of completeness. 
Finally, in Section~\ref{sec:cor} we apply these theorems to derive the lower bounds on the mixing time and relaxation time of random walks on the dynamical random cluster model (Corollary~\ref{cor:lb}).

\section{Multi-scale setup}
\label{tessellation}

We start defining a multi-scale tessellation of $\bT_n^d$, which will consist of partitioning $\bT_n^d$ into boxes and defining 
the event that boxes are \emph{good} or \emph{bad}. Those events will be then used to define the favorable parts of the environment.

\subsection{Tessellation}\label{tessellation}
Let 
\begin{equation}
   \label{ell}
   \ell\defn p^{-\frac{1}{3d}},
\end{equation}
and $m$ be a sufficiently large integer.
For each $k\ge 1$ we tessellate $\mathbb{T}_n^d$ into cubes of length $\ell_k$ where
\begin{equation}
   \label{ellkdeff}
   \ell_1=\ell\quad\text{ and }\quad \ell_{k+1}=mk^2\ell_{k}.
\end{equation}
The cubes will be indexed by integer vectors $i\in\mathbb{Z}^d$, and denoted $S_k^{\core}(i)\subset\mathbb{T}_n^d$ with
\begin{equation*}
   S_k^\core(i)\defn i \ell_k + [0,\ell_k)^d. 
\end{equation*} 
We will consider a tiling of $\mathbb{T}_n^d$ with a hierarchy as each cube of scale $k$ is contained inside a unique cube of scale $k+1$. 
For simplicity we will assume $\ell_k$ divides $n$ for all $k$ we will consider\footnote{If that were not the case, one could consider for each $k$ some cubes to have length between $\ell_k$ and $2\ell_k$ to fully 
   tessellate the torus.}. 
Moreover for any subset $V$ of the vertices of $\mathbb{T}_n^d$, we denote by
\begin{equation*}
    E(V)=\{(v_1,v_2)\in E(\mathbb{T}_n^d):v_1,v_2\in V\}
\end{equation*}
the set of all edges incident only to vertices in $V$.

Now we define a multi-scale tessellation of time. At scale 1, we tessellate $\mathbb{R}$ into intervals of length $t_1=\frac{\sqrt{\ell}}{\mu}$ and then, for higher scales, we define
\begin{equation*}
   t_{k+1}\defn mk^2t_{k},\quad  k\geq  1 .
\end{equation*}
We index the time intervals by $\tau\in \Z$ and denote them by $T_k^\core(\tau)$, where
\begin{equation*}
    T_k^\core(\tau)=\left[\tau t_k,\, (\tau+1)t_k\right).
\end{equation*}

Now for any $i\in\mathbb{Z}^d$, $k\geq  1$, and $\tau\in\mathbb{Z}$, we define the core of the space-time $k$-box by
$$
    R_k^\core(i,\,\tau)\defn S_k^\core(i)\times T_k^\core(\tau).
$$
For any subset of $A \subset \Z^d$, we let $\partial A$ denote its inner boundary. Then, in space-time, we define the spatial boundary of $R_k^\core(i,\tau)$ by
\begin{equation}
    \partial_\rs R_k^\core(i,\tau)=\partial S_k^\core(i) \times T_k^\core(\tau).
    \label{eq:sbd}
\end{equation}
For the time dimension, we define two time boundaries, the boundary $\partial_\rt^+$ corresponding to the largest unit of time in the box and the boundary 
$\partial_\rt^-$ corresponding to the smallest unit of time in the box:
\begin{align}
    \partial_\rt^+ R_k^\core(i,\tau)=S_k^\core(i)\times \sup\lrc{T_k^\core(\tau)} = S_k^\core(i)\times \lrc{(\tau+1)t_k}, \text{ and}\nonumber\\
    \partial_\rt^-R_k^\core(i,\tau)=S_k^\core(i)\times\inf\lrc{T_k^\core(\tau)}=S_k^\core(i)\times\left\{\tau t_k\right\}.
    \label{eq:tbd}
\end{align}
For $k\geq 2$, each box $R_k^\core(i,\tau)$  will be the central part of a larger box
\begin{equation*}
    R_k(i,\,\tau)\defn \bigcup_{(j,\tau') \in\{-1,0,+1\}^{d+1}} R_k^\core(i+j,\tau+\tau')=S_k(i)\times T_k(\tau),
\end{equation*}
where we let
\begin{equation}
    S_k(i)\defn \bigcup_{j\in\{-1,0,+1\}^{d}} S_k^\core(i+j),
    \quad\text{and}\quad 
    T_k(\tau)\defn\bigcup_{\tau'\in\{-1,0,+1\}} T_k^\core(\tau+\tau').
    \label{eq:sktk}
\end{equation}
In words $R_k(i,\,\tau)$ is composed of a cube in space of side length $3\ell_k$ and a time interval of length $3t_k$, and it has $R_k^\core(i,\,\tau)$ as its central part (see Figure~\ref{fig:test}).
\begin{figure}[t]
    \centering
   \begin{subfigure}{.45\textwidth}
     \centering
     \includegraphics[scale=.7]{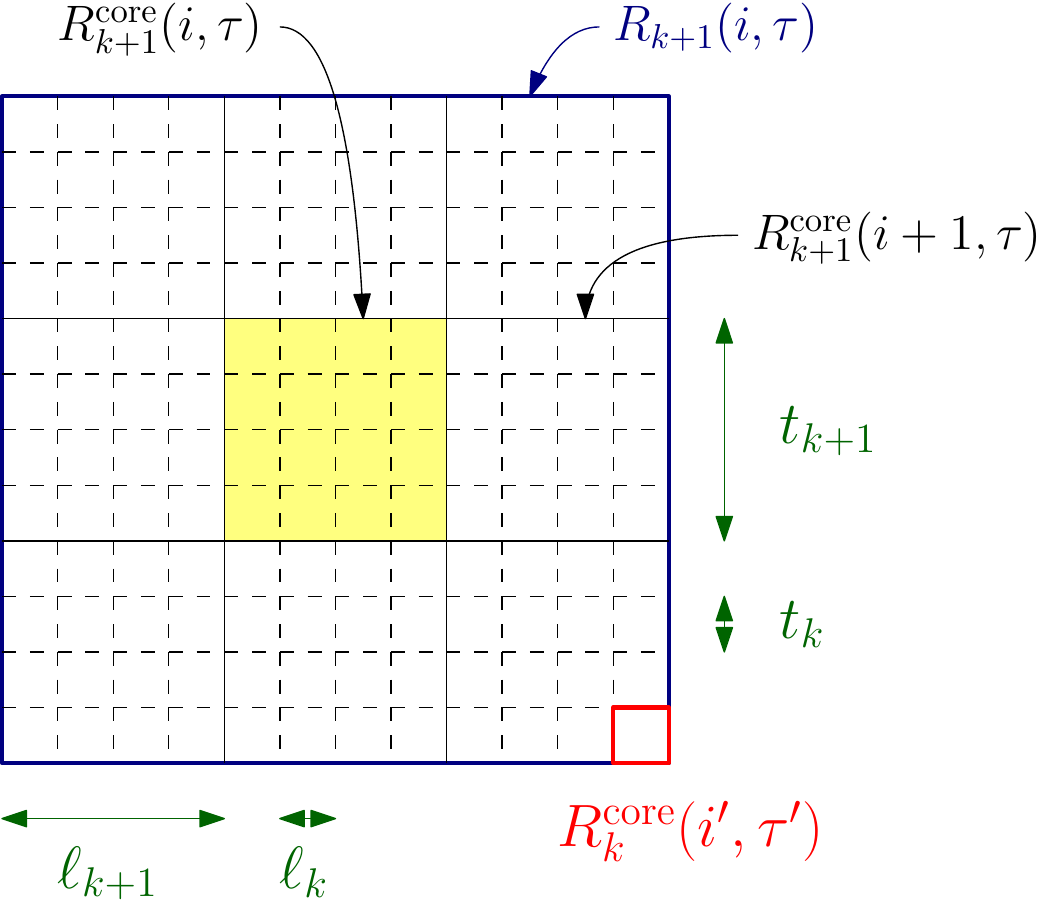}
   \end{subfigure}%
   \hspace{\stretch{1}}\begin{subfigure}{.45\textwidth}
     \centering
     \includegraphics[scale=.7]{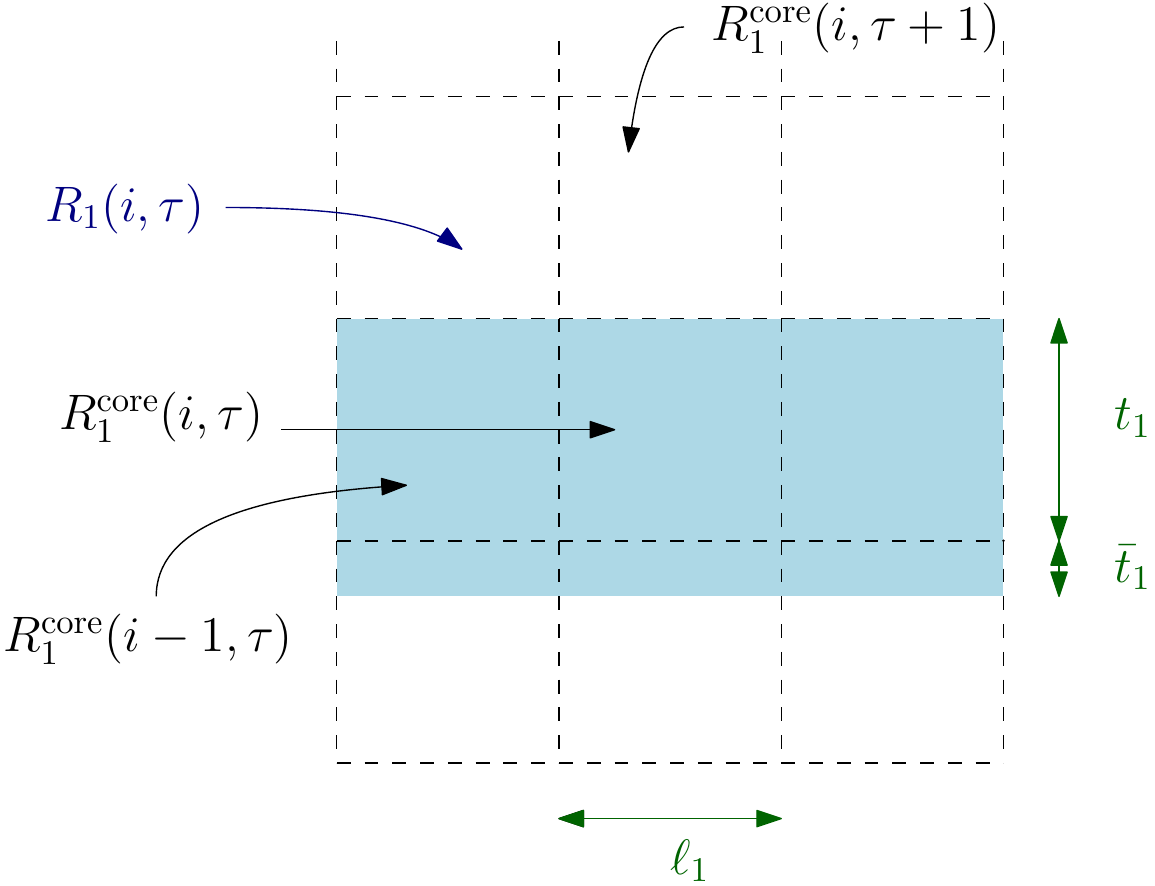}
   \end{subfigure}
   \caption{On the left a space-time box $R_{k+1}(i,\tau)$ represented by a blue square, its core highlighted in yellow, its partition into cores of scale $k+1$ represented by solid black lines, 
      and its partition into cores of scale $k$ in dashed lines.
      The horizontal axis represents space and the vertical axis represents time. 
      On the right a space-time box of scale 1 (highlighted in blue), and space-time cores of level $1$ in dashed lines.}
   \label{fig:test}
\end{figure}
For scale $k=1$, we will 
need a small intersection between the time dimension. For this, let 
\begin{equation}
   \bar{t}_1 \defn \frac{\log^2\ell}{\mu},\label{eq:bt1}
\end{equation}
and set for each $\tau\in\Z$
$$
   T_1(\tau) \defn \left[\tau t_1-\bar{t}_1,\, (\tau+1)t_1\right].
$$
With $S_1$ defined as in~\eqref{eq:sktk} we can define $R_1(i,\tau)=S_1(i)\times T_1(\tau)$.
We then define the space and time boundaries of $R_k(i,\tau)$ for each $k,i,\tau$ analogously to~\eqref{eq:sbd} and~\eqref{eq:tbd}.

Finally we denote by $S_1^\inn(i)$ the \emph{inner part} of $S_1(i)$ which is obtained by removing all the vertices within distance $\frac{\gamma}{6}\log^2\ell$ from the boundary of $S_1(i)$ 
($\gamma$ is a constant that will be clarified later in the definition); in symbols,
\begin{equation}
    S_1^\inn(i)\defn\lrc{v\in S_1(i)\colon \|v-w\|_1>\frac{\gamma}{6}\log^2\ell \text{ for all }w\in \partial S_1(i)}.
    \label{eq:sinn}
\end{equation}

\subsection{Good boxes at scale 1}
\begin{definition}
   We say that an event $A$ is restricted to a region $R\subset V(\mathbb{T}_n^d)$ and a time interval $[s_0,s_1]$ 
   if $A$ is measurable with respect to the $\sigma$-algebra generated by the updates of the edges from $E(R)$ from time $s_0$ to $s_1$, together with the 
   random variables $U,U'$ associated to such updates. 
\end{definition}

Denote with $\mathcal{C}_x(t)$ the connected component of open edges containing vertex $x\in V$ at time $t$. 
Given a time interval $[s,s']$, we say that an edge is $[s,s']$-open if there is at least one time during $[s,s']$ at which this edge is open.
Then, we denote with $\mathcal{C}_x(s,s')$ the connected component of $[s,s']$-open edges that contains $x$. If we denote $I=[s,s']$, then we employ the shorter notation
$\mathcal{C}_x(I)$.
Below we split the time interval of a box into two sets of sub-intervals, and then introduce the definition of \emph{good boxes}.

\begin{definition}\label{def:dbar}
   Recall that $\bar{t}_1=\frac{\log^2\ell}{\mu}$. We define two other tessellations of disjoint time intervals. 
   The first one has length $\frac{\log^2\ell}{\mu}$:
   $$
      \bar{T}_1(j) = [j \bar{t}_1,(j+1)\bar{t}_1), \quad \text{ for $j\in \Z_+$}.
   $$
   Moreover, we fix a constant $\gamma=\gamma(p,q,d)>0$ such that $\frac{p_\rmin}{p_\star}+\gamma < p_\crit$, where $p_\crit$ is the critical probability for independent bond percolation on $\Z^d$,
   and introduce a tessellation of time of length $\frac{\gamma}{\mu}$:
   $$
      \dbar{T}_1(j) = \left[j \tfrac{\gamma}{\mu},(j+1)\tfrac{\gamma}{\mu}\right), \quad \text{ for $j\in\Z_+$}.
   $$
\end{definition}

We assume throughout this paper that $\bar t_1$ divides $t_1$ and that $\gamma/\mu$ divides $\bar t_1$, so that $\dbar{T}_1$ is a finer tessellation than $\bar{T}_1$, which in turn is a finer tessellation than $T_1^\core$.
\begin{remark}\label{rem:gamma}
   Note that the larger $p$ is (that is, the closer $p$ is to $p_\crit$) the smaller we need to take $\gamma$. However, as we will need to take $p$ small enough in several places in the proof, we will set
   $\gamma$ first (for example, it is enough to take $\gamma=\frac{1}{100}$). Then we make $p$ small enough so that the condition on $\gamma$ is satisfied.
\end{remark}

The definition of good boxes will be done in steps. First we define some fundamental events that we will require from good boxes.
\begin{enumerate}
   \item[($\cG_1'$)] Given a box $R_1(i,\tau)$, let $\cG_1'(i,\tau)$ be the event that, for any $e\in E(S_1(i))$, there are no non-$\star$ update on $e$ during $T_1(\tau)\cap[0,\infty)$.
   \item[($\cG_{2}'$)] For each $j \in \Z_+$ and each spatial box $S_1(i)$, define $\cG_{2}'(i,j)$ the event that, for any $e\in E(S_1(i))$, during $\bar{T}_1(j)$ edge $e$ never gets a $\star$-update to become open.
   \item[($\cG_{3}'$)] For each $j \in \Z_+$ and each spatial box $S_1(i)$, let $\cG_{3}'(i,j)$ be the event that, for each $e\in E(S^\core_1(i))$, 
      the number of $\star$-updates on edge $e$ during $\bar{T}_1(j)$ is at least $\frac{1}{2}p_\star \log^2\ell$ (for the values of $\ell$ and $p$ we will consider this will always be at least 1).
   \item[($\cG_{4}'$)] For any $j\in\Z_+$ take the unique $\tau$ such that $\dbar{T}_1(j)\subset T_1^\core(\tau)$.
      For any site $x$ on the torus, 
      if we regard all edges closed at time $\tau t_1$ and we only consider $\star$-updates disregarding all non-$\star$-updates,
      then let $\cG_{4}'(x,j)$ be the event that 
      $
         \left|\mathcal{C}_x\left(\dbar{T}_1(j)\right)\right|<\log^2(\ell).
      $
\end{enumerate}
Now for a box $R_1(i,\tau)$, define
\begin{equation}
   j(\tau) \text{ the value $j$ such that } \min \bar{T}_1(j) = \tau t_1-\bar t_1;
   \label{eq:jtau}
\end{equation}
in other words, it is the value such that $\bar{T}_1(j(\tau))$ starts at the initial time of $R_1(i,\tau)$. Note that both $\bar{T}_1(j(\tau))$ and $\bar{T}_1(j(\tau+1))$ are contained in $T_1(\tau)$. 

The event that a box $R_1(i,\tau)$ is good will be composed of four events, which we denote by $\cG_1(i,\tau)$, $\cG_2(i,\tau)$, $\cG_3(i,\tau)$ and $\cG_4(i,\tau)$. 
The first event regards only non-$\star$ updates and is simply 
$$
   \cG_1(i,\tau) \defn \cG_1'(i,\tau).
$$
The second event regards the time intervals $\bar{T}_1(j(\tau))$ and $\bar{T}_1(j(\tau+1))$, and is defined as 
$$
      \cG_2(i,\tau) \defn \cG_{2}'(i,j(\tau))\cap\cG_{2}'(i,j(\tau+1))\bigcap_{{\substack{i'\colon S_1^\core(i')\subset S_1(i)}}} \cG_{3}'(i',j(\tau)) \cap \cG_{3}'(i',j(\tau+1)).
$$
The next two events are confined to the time interval $\T_1^\core(\tau)\setminus T_1(\tau+1)$:
$$
   \cG_3(i,\tau) \defn \bigcap_{\substack{i'\colon S_1^\core(i')\subset S_1(i)\\ j \colon \bar{T}_1(j)\subset \lr{T_1^\core(\tau)\setminus T_1(\tau+1)}}} \cG_3'(i',j)
$$
and
$$
   \cG_4(i,\tau) \defn \bigcap_{\substack{x\in S_1^\inn(i)\\j \colon \dbar{T}_1^j(\tau)\subset \lr{T_1^\core(\tau)\setminus T_1(\tau+1)}}}\cG_4'(x,j).
$$
For convenience, we write 
$$
   \cG_{12}(i,\tau)  = \cG_1(i,\tau) \cap \cG_2(i,\tau)
   \quad\text{and}\quad
   \cG_{34}(i,\tau)  = \cG_3(i,\tau) \cap \cG_4(i,\tau)
$$

\begin{lemma}
   \label{lem:badboxI}
   The family of events $\lrc{\cG_{1}(i,\tau)}_{(i,\tau)}$, $\lrc{\cG_{2}(i,\tau)}_{(i,\tau)}$ and $\lrc{\cG_{34}(i,\tau)}_{(i,\tau)}$ are independent of one another. 
\end{lemma}
\begin{proof}
   The events $\cG_1(\cdot,\cdot)$ depend only on non-$\star$ updates, so it is independent of $\cG_2(\cdot,\cdot)$, $\cG_3(\cdot,\cdot)$ and $\cG_4(\cdot,\cdot)$, 
   which only regard $\star$-updates. Then, note that $\cG_2(\cdot,\cdot)$ are events about $\star$-updates during $\bigcup_\tau \bar{T}_1(j(\tau))$, while $\cG_3(\cdot,\cdot)$ and 
   $\cG_4(\cdot,\cdot)$ regard $\star$-updates during $\bigcup_\tau T_1^\core(\tau)\setminus T_1(\tau+1)$. Since these two time intervals are disjoint, independence is obtained from standard properties of 
   Poisson processes.
\end{proof}

\begin{lemma}
   \label{lem:badboxpre}
   Let $R_1(i,\tau)$ be any box of scale $1$.
   There exist constants $c,c'>0$ so that for all small enough $p$ we obtain
   $$
       \PR\lr{\cG_{12}^\compl(i,\,\tau)}\leq  c\ell^{d+\frac{1}{2}}p_\rmax.
   $$
   and 
   $$
       \PR\lr{\cG_{34}^\compl(i,\tau)}\leq \exp\lr{-c'\log^2\ell}.
   $$
\end{lemma}
\begin{proof}
   We start with event $\cG_{34}(i,\tau)$. 
   For a given edge, an update that is not $\star$ occurs at rate $(1-p_\star)\mu$, a $\star$-update occurs at rate $p_\star\mu$, and
   a $\star$-update that opens an edge occurs at rate $p_\star \frac{p_\rmin}{p_\star}\mu=p_\rmin\mu$. 
   Moreover, there are at most $2d\,3^d\ell^d$ edges in $E(S_1(i))$.
   
   For $\cG_3$, note that for a given edge $e\in S_1(i)$ and a given $j$, the number of $\star$-updates on $e$ during $\bar{T}_1(j)$ is a Poisson random variable of 
   mean $p_\star \log^2\ell$. Therefore, using a standard Chernoff bound for Poisson random variables and the union bound over the edges in $S_1(i)$ and over the values of $j$, we 
   obtain a constant $c_1>0$ such that 
   \begin{align*}
       \PR\lr{\cG_3^\compl(i,\tau)}\leq& 2d3^{d}\ell^d\, \frac{(t_1-\bar t_1)\mu}{\log^2\ell}\exp\lr{-c_1p_\star \log^2\ell}.
   \end{align*}
   Note that as $p$ decreases to $0$ we have that $p_\star$ increases to $1$ and $\ell$ increases to $\infty$. So for all small enough $p$ we obtain 
   \begin{align*}
       \mathbb{P}(\cG_3^\compl(i,\tau))\leq& \exp\lr{-c_2\log^2\ell},
   \end{align*}
   for some constant $c_2$.
   Regarding $\cG_4(i,\tau)$, for any $j\geq 0$ with $\dbar{T}_1(j)\subset T_1^\core(\tau)\setminus T_1(\tau+1)$ and any edge $e$, 
   note that the probability that $e$ is open at the beginning of the interval $\dbar{T}_1(j)$, given that we only consider $\star$-updates and consider that all edges are closed at $\tau t_1$,
   is at most $\frac{p_\rmin}{p_\star}$, since this is the probability that the last $\star$-update of $e$ (if there was any) made $e$ open. Thus, using that $1-\exp\lr{-p_\star \gamma}$ is the probability 
   that $e$ has a $\star$-update during $\dbar{T}_1^j$, we obtain that the probability that 
   $e$ is $\dbar{T}_1^j$-open under the above assumptions is at most 
   $$
       \frac{p_\rmin}{p_\star}+1-\exp\lr{-p_\star \gamma}
       \leq \frac{p_\rmin}{p_\star}+p_\star \gamma < p_\crit.
   $$
   In other words, the set of $\dbar{T}_1^j$-open edges (under the above assumptions) forms a subcritical percolation cluster.  
   Therefore, using the exponential decay of cluster size for subcritical percolation, together with the union bound over all sites and values of $j$, we obtain a constant $c''$ such that 
   \begin{align*}
       \mathbb{P}(\cG_4^\compl(i,\tau))\leq \ell^d \frac{t_1\mu}{\gamma} \exp\lr{-c_3 \log^2\ell}.
   \end{align*}
   We note that $c_3$ increases as $p$ decreases. So the bound on $\PR\lr{\cG_{34}^\compl(i,\tau)}$ follows by taking $p$ small enough.

   Now we turn to $\cG_{12}(i,\tau)$. From the above considerations we obtain
   \begin{align*}
       \mathbb{P}(\cG_{12}^\compl(i,\tau))
       &\leq 1-\exp\lr{-2d3^d\ell^{d}\lr{(1-p_\star)\mu\lr{t_1+\bar t_1}+p_\rmin \log^2\ell}} +2d3^{d}\ell^d\, 2\exp\lr{-c_1p_\star \log^2\ell},
   \end{align*}
   where the term $1-\exp(\cdot)$ comes from $\cG_1'$ and $\cG_2'$, while the last term comes from $\cG_3'$ as in the case of $\cG_3$ above.
   Using that $t_1+\bar t_1\leq 2t_1$ and that $e^{-x}\geq 1-x$ for all $x\in\R$, we obtain
   \begin{align*}
       \mathbb{P}(\cG_{12}^\compl(i,\tau))
       &\leq 2d3^d\ell^{d}\lr{(1-p_\star)2\sqrt{\ell}+p_\rmin \log^2\ell+2\exp\lr{-c_1p_\star \log^2\ell}}\\
       &\leq 4d3^d\ell^{d+\frac{1}{2}}(1-p_\star+p_\rmin)\\
       &= 4d3^d\ell^{d+\frac{1}{2}}(1-p_\star+p_\rmin)
       =4d3^d\ell^{d+\frac{1}{2}}p_\rmax,
   \end{align*}
   where we use that $\sqrt{\ell}$ is the term that dominates inside the parenthesis as $p$ is made small enough (hence, $\ell$ is made large enough).
   So we can take $p$ small enough so that $p_\rmin \log^2\ell+2\exp\lr{-c_1p_\star \log^2\ell}\leq 2p_\rmin \log^2\ell\leq 2p_\rmin \sqrt{\ell}$.
\end{proof}

We need one more step to define good boxes of scale $1$. Using a standard result for percolation with bounded dependences~\cite{LSS}, we will replace $\cG_{34}(i,\tau)$ by 
a collection of i.i.d.\ Bernoulli random variables.
\begin{lemma}\label{lem:badboxlss}
   There exists a constant $C_\newcte{cte:bblss}=C_{\ref*{cte:bblss}}(d)>0$ such that letting $\lrc{\hat\cG_{34}(i,\tau)}_{(i,\tau)}$ be a collection of i.i.d.\ Bernoulli random variables of parameter $1-\exp\lr{-C_{\ref*{cte:bblss}}\log^2\ell}$,  
   Then for any $p$ small enough we obtain that $\lrc{\cG_{34}(i,\tau)}_{(i,\tau)}$ stochastically dominates $\lrc{\hat\cG_{34}(i,\tau)}_{(i,\tau)}$.
\end{lemma}
\begin{proof}
   First note that if we fix $i$, then $\cG_{34}(i,\tau)$ forms a collection of independent random variables as $\tau$ varies. 
   So $\cG_{34}(i,\tau)$ depends only on events $\cG_{34}(i',\tau)$ such that $S_1(i)\cap S_1(i')\neq\emptyset$. Note that this is true even for the events $\cG_4'$ since 
   for $x\in S_1^\inn(i)$, the event that the component of $x$ is of length $\log ^2\ell$ is measurable with respect to the edges in $S_1(i)$. 
   Given $i$, the number of $i'$ such that $S_1(i)\cap S_1(i')\neq\emptyset$ is strictly smaller than $\Delta=5^d$. 
   Then, from Lemma~\ref{lem:badboxpre} we obtain 
   that by taking $p$ small enough the marginal probability $w=\PR\lr{\cG_{34}^\compl(i,\tau)}$ can be made smaller than $\frac{(\Delta-1)^{\Delta-1}}{\Delta^\Delta}$, and so we can apply~\cite[Theorem~1.3]{LSS} to 
   deduce that the family $\cG_{34}(\cdot,\cdot)$ stochastically dominates a set of i.i.d.\ Bernoulli random variables with parameter 
   \begin{align*}
      \rho
      &=\lr{1-\frac{w^{1/\Delta}}{(\Delta-1)^{1-1/\Delta}}}\lr{1-\lr{w(\Delta-1)}^{1/\Delta}}\\
      &\geq 1- \frac{w^{1/\Delta}}{(\Delta-1)^{1-1/\Delta}}-\lr{w(\Delta-1)}^{1/\Delta}\\
      &= 1 - w^{1/\Delta}\lr{\frac{\Delta}{(\Delta-1)^{1-1/\Delta}}}.
   \end{align*}
   Applying Lemma~\ref{lem:badboxpre} for the value of $w$ completes the proof since $w^{1/\Delta}\lr{\frac{\Delta}{(\Delta-1)^{1-1/\Delta}}}\leq \exp(-C_{\ref*{cte:bblss}}\log^2\ell)$ for some constant $C_{\ref*{cte:bblss}}>0$. 
\end{proof}

Now we are ready to define good boxes at scale $1$.
\begin{definition}[Good boxes of scale 1]\label{def:goodbox}
   Let $i\in\Z^d$ and $\tau\geq 0$. 
   A box $R_1(i,\,\tau)$ is said to be \emph{good} if the following event happens
   $$
      \cG(i,\tau)
         \defn \cG_{12}(i,\tau) \cap \hat\cG_{34}(i,\tau).
   $$
   For convenience, we assume that for $\tau<0$ then $\cG(i,\tau)$ holds for all $i\in\Z^d$. 
   We also couple $\lrc{\hat\cG_{34}(i,\tau)}_{(i,\tau)}$ with $\lrc{\cG_{34}(i,\tau)}_{(i,\tau)}$ so that, 
   \begin{equation}
      \text{for each $(i,\tau)$, whenever $\hat\cG_{34}(i,\tau)$ holds, so does $\cG_{34}(i,\tau)$.}
      \label{eq:hatcoup}
   \end{equation}
\end{definition}
 
The following lemma bounds the probability that a box is bad, and follows directly from the previous lemmas.
\begin{lemma}
   \label{lem:badbox}
   Let $R_1(i,\tau)$ be any box of scale $1$.
   There exists a constant $C_\newcte{cte:badbox}>0$ so that for all small enough $p$ we obtain
   $$
       \PR\lr{\cG^\compl(i,\,\tau)}\leq  C_{\ref*{cte:badbox}}\ell^{d+\frac{1}{2}}p_\rmax.
   $$
\end{lemma}
\begin{proof}
   This follows from Lemmas~\ref{lem:badboxpre} and~\ref{lem:badboxlss}
\end{proof}

\begin{remark}
   Note that the event $\{R_1(i,\tau)\text{ is good}\}$ is restricted to the cube $S_1(i)$ and the interval $T_1(\tau)$. In fact, that is why in~$\cG_4'$ we assume that all edges are closed at time 
   $\tau t_1$, the initial time of the core $T_1^\core(\tau)$; note that the fact that all edges are closed at time $\tau t_1$ is implied by $\cG_2$, but by explicitly adding 
   the assumption in $\cG_4'$ we make $\cG_2(i,\tau)$ and $\cG_4(i,\tau)$ 
   independent of each other.
   Note also that the decision of whether a box is good is completely independent of the walker, it only depends on the updates of the dynamical random cluster process. 
\end{remark}

Recall that $X_t$ denotes the position of the random walker at time $t$. In the lemma below, we will show that if the walker happens to be inside a good box, then it cannot move very quickly. 
This will allow us to have a better control on where the walker can be.

\begin{lemma}
   \label{eventa3}
   Let $t\geq  0$ be any given time and suppose $(X_t,t)\in R_1^\core(i,\tau)$, where $R_1(i,\tau)$ is a good box. Then,
   \begin{equation*}
       \sup_{\substack{s \geq t\\ s\in T_1(\tau)}}\|X_t-X_{s}\|_1
       \leq  |\cC_{X_t}(t)| + \left\lceil\frac{(\tau+1) t_1 -t}{\gamma/\mu}\right\rceil \log^2\ell,
   \end{equation*}
   In particular, if $\tau>0$ then 
   \begin{equation*}
       \sup_{\substack{s \geq t\\ s\in T_1(\tau)}}\|X_t-X_{s}\|_1
       \leq  \log^2\ell + \left\lceil\frac{(\tau+1) t_1 -t}{\gamma/\mu}\right\rceil \log^2\ell
       \leq \log^2\ell + \frac{\sqrt{\ell}}{\gamma} \log^2\ell  \leq \frac{\ell}{10},
   \end{equation*}
   where the last inequality holds for all $p$ small enough (thus, $\ell$ large enough) and where $\|x - y\|_1$ denotes the $L^1$ distance in the torus between the positions $x,y \in \mathbb{T}_n^d$; 
   in particular, $\|x - y\|_1$ does not depend on whether edges are open or closed.
\end{lemma}
\begin{proof}
   This is a direct consequence of the event~$\cG_4'$ from the definition of good boxes, and the fact that good boxes do not have non-$\star$-updates. 
   There is just one caveat. The event $\cG_4'$ is not enforced in time intervals $\bar{T}_1(j)$ where $j=j(\tau)$ for some $\tau$; recall the definition of $j(\tau)$ from~\eqref{eq:jtau}.
   So for example, $\cG_4(i,\tau)$ does not include the event $\cG'_4(j)$ for $j$ such that $\bar{T}_1(j)=T_1(\tau)\cap T_1(\tau+1)$, which is the only time interval of the type $\bar{T}_1(\cdot)$ inside 
   $T_1(\tau)$. However, for such a $j$, we know that the connected components inside $\bar{T}_1(j-1)$ are of size at most $\log^2\ell$, and $\cG_2(i,\tau)$ implies that during $\bar{T}_1(j)$ no edge of $S_1(i)$ gets an
   update to open. Therefore, the size of the connected components can only decrease during $\bar{T}_1(j)$ and the result follows.
\end{proof}

\subsection{Good boxes at larger scales}

In this section we define the concept of good and bad boxes of scale larger than 1, but first we define a slightly relaxed version of \emph{intersection} of boxes.
\begin{definition}
   \label{intersection} 
   Since boxes are defined by semi-open intervals, we will consider boxes that are barely non-intersecting as intersecting. That is, we consider two boxes 
   $R_k(i,\tau)$ and $R_k(i',\tau')$ as \emph{non-intersecting} if and only if 
   $$
      \inf_{\substack{(j,s)\in R_k(i,\tau)\\(j',s')\in R_k(i',\tau')}}\|(j,s)-(j',s')\|_1\geq 2.
   $$
\end{definition}
\begin{definition}
   \label{badbox}
   A $k$-box $R_k(i,\,\tau)$ with $k\ge 2$ is said to be \emph{bad} if it contains at least two non-intersecting bad boxes of scale $k-1$. 
   Otherwise, $R_k(i,\,\tau)$ is called \emph{good}.
\end{definition}

\begin{remark}
   \label{rem1}
   The event $\{R_{k}(i,\tau) \text{ is bad}\}$ is strictly restricted to the cube $S_k(i)$ and the time interval $T_k(\tau)$. 
   Moreover, by translation invariance, for any pair $(i,\tau),\,(i',\tau')$ and any scale $k$ we have $\mathbb{P}(R_k(i,\,\tau)\text{ is bad})=\mathbb{P}(R_k(i',\,\tau')\text{ is bad})$. 
   Therefore if $R_k(i,\tau)$ and $R_k(i',\tau')$ are two non-intersecting boxes then
   \begin{equation*}
       \mathbb{P}(R_k(i,\,\tau)\text{ and }R_k(i',\,\tau')\text{ are bad})=\mathbb{P}(R_k(i,\,\tau)\text{ is bad})^2.
   \end{equation*}
\end{remark}

\begin{definition}
   Define $\rho_k$ as the probability that a $k$-box $R_k(i,\,\tau)$ is bad, that is, 
   $$
      \rho_k\defn\mathbb{P}\left(R_k(i,\,\tau)\text{ is bad}\right).
   $$
   As noted in Remark \ref{rem1}, this probability does not depend on $(i,\tau)$.
\end{definition}
Recall that $m$ is the variable that appears in the definition of $\ell_k$ from \eqref{ellkdeff}.
\begin{lemma}
   \label{lemrho}
   For any $m>0$, by setting $p$ small enough we obtain 
   \begin{equation*}
       \rho_k\le\rho_1^{2^{k-2}}
   \end{equation*}
\end{lemma}
\begin{proof}
   We prove the lemma in a slightly stronger version: we prove that we can set values $c_k$, satisfying $c_k\geq \frac{1}{2}$ for all $k$, so that
   \begin{equation*}
       \rho_k\leq\rho_1^{c_k2^{k-1}}.
   \end{equation*}
   We prove this by induction. For $k=1$ the statement is trivially satisfied by setting $c_1=1$. Assume the statement is true up to $k$. Now, by the definition of bad box we have
   \begin{align*}
       \rho_{k+1}&\leq \lr{\lr{\frac{3\ell_{k+1}}{\ell_k}}^d\frac{3t_{k+1}}{t_k}}^2\rho_k^2\\
       &= 3^{2d+2}(mk^2)^{2d+2}\rho_k^{2}\\
       &\leq 3^{2d+2}(mk^2)^{2d+2}\lr{\rho_1^{c_k2^{k-1}}}^2\\
       &=3^{2d+2}(mk^2)^{2d+2}\rho_1^{(c_k-c_{k+1})2^k}\rho_1^{c_{k+1}2^k}.
   \end{align*}
   Setting $c_{k+1}=c_k-\frac{1}{10k^2}$ gives that 
   \begin{equation*}
      3^{2d+2}(mk^2)^{2d+2}\rho_1^{(c_k-c_{k+1})2^k}=3^{2d+2}(mk^2)^{2d+2}\rho_1^\frac{2^k}{10k^2}\leq  1,
   \end{equation*}
   for all $k\geq  1$, provided $\rho_1$ is small enough with respect to $m$. Given $m$, $\rho_1$ can be made small enough by setting $p$ small enough, as in Lemma~\ref{lem:badbox}.
   Notice that $c_k>c_1-\sum_{i=1}^\infty\frac{1}{10i^2}\geq  \frac{1}{2}$, which proves the lemma.
\end{proof}

\subsection{Enlargement of boxes}
As we discussed in the proof overview, whenever the walkers are in a favorable region of the environment, we will try to use a simple random walk coupling to bring the walkers together. However, when the walkers are in an unfavorable 
region of the environment, which essentially means that the walkers are approaching a bad box (at some scale), then we will have to refrain from doing this simple random walk coupling, and will just do a na\"ive 
identity coupling in order to let the environments couple around the walkers before
they can reach the bad box. 
Here we will define two types of enlargements of boxes so that it is when the walker enters the enlargement of a bad box that we will need to stop doing the simple random walk coupling.

\begin{definition}[1-enlargement]
   The \emph{1-enlargement} of a box $R_k(i,\,\tau)$ of scale $k$, is the set of boxes 
   \begin{equation*}
       R_k^\enl(i,\,\tau)\defn\bigcup_{(j,\beta)\in\{-3,-2,\dots,3\}^{d+1}} R_k(i+j,\tau+\beta).
   \end{equation*}
   We also denote
   \begin{equation*}
       S_k^\enl(i)\defn\bigcup_{j\in\{-3,-2,\dots,3\}^{d}} S_k(i+j),\quad\text{and}\quad T_k^\enl(\tau)\defn\bigcup_{\beta\in\{-3,-2,\dots,3\}} T_k(\tau+\beta).
   \end{equation*}
\end{definition}

Note that $R_k^\enl(i,\tau)$ is a parallelogram of spatial length $9\ell_k$ and time length $9t_k$ for $d\geq 2$ and $7t_1+\bar t_1$ for $d=1$.

\begin{remark}
   \label{enl1shield}
   The 1-enlargement is a $(d+1)$-dimensional parallelogram centered in $R_k(i,\tau)$ defined to obtain the following property. Let 
   $R_k(i,\tau)$ be a bad box, whose whole 1-enlargement $R_k^\enl(i,\tau)$ is contained inside a good box of scale $k+1$. Then, we know that the only boxes of scale $k$ inside the $(k+1)$-box that 
   can be bad are those intersecting $R_k(i,\tau)$. Let $I$ be the set of tuples $(i',\tau')$ such that 
   $R_k(i,\tau)\cap R_k(i',\tau')\neq \emptyset$. Note that $R_k(i',\tau')\subset R_k^\enl(i,\tau)$ for all $(i',\tau')\in I$. 
   Moreover, the property that we get is that $\bigcup_{(i',\tau')\in I} R_k^\core(i',\tau')$ does not exhaust $R_k^\enl(i,\tau)$ in the sense
   that $\bigcup_{(i',\tau')\in I} R_k^\core(i',\tau')$ is separated from the outside of $R_k^\enl(i,\tau)$ by at least one layer of cores. We define this layer of cores as 
   $$
      R_k^\benl(i,\tau)=R_k^\enl(i,\tau)\setminus \lr{\bigcup_{(i',\tau')\in I} R_k^\core(i',\tau')}.
   $$
%
\end{remark}

\begin{definition}[2-enlargement]
The \emph{2-enlargement} of a box $R_k(i,\,\tau)$ of scale $k$ is the set of boxes 
\begin{equation*}
    R_k^\enlb(i,\,\tau)\defn\bigcup_{\substack{j\in\{-20,-19,\dots,20\}^d,\\\beta\in\{-18,-17,\dots,3\}}} R_k(i+j,\tau+\beta),
\end{equation*}
and we also denote
\begin{equation*}
    S_k^\enlb(i)\defn\bigcup_{j\in\{-20,-19,\dots,20\}^d} S_k(i+j),\quad\quad T_k^\enlb(\tau)\defn\bigcup_{\beta\in\{-18,-17,\dots,3\}} T_k(\tau+\beta).
\end{equation*}
\end{definition}

Note that the 2-enlargement is a larger $(d+1)$-dimensional parallelogram centered in $R_k(i,\tau)$ so that $\sup\lrc{ T_k^\enl(\tau)}=\sup\lrc{T_k^\enlb(\tau)}$; See Figure~\ref{fig:affenl}.

\begin{figure}
    \centering
    \includegraphics{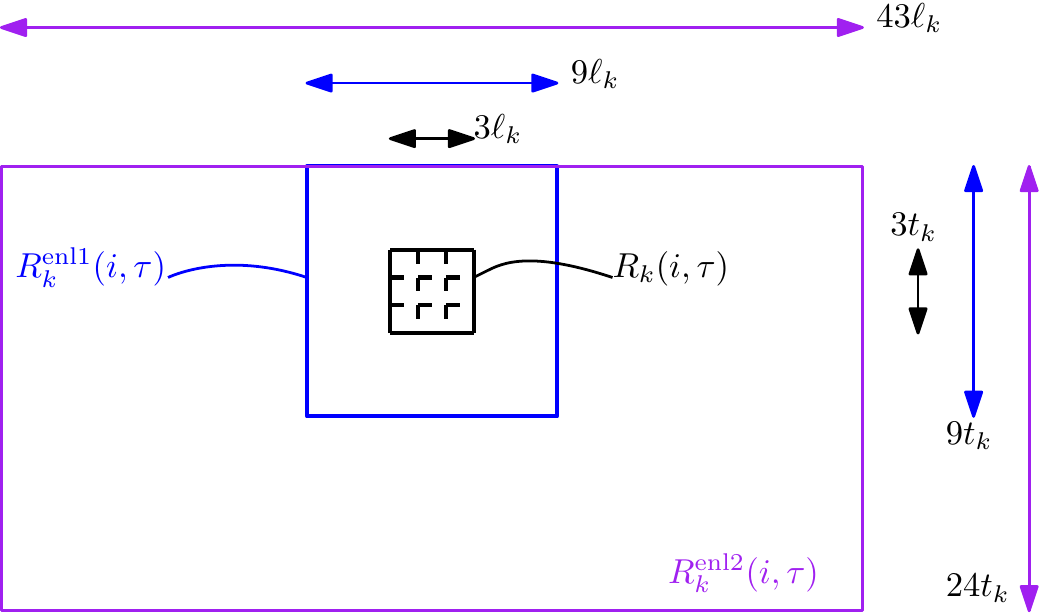}
    \caption{In black a box $R_k(i,\tau)$, with its 1-enlargement in blue and its 2-enlargement in purple. The figure is not to scale and illustrates the case $k\geq 2$; recall that the time intervals are defined differently at scale $1$.}
    \label{fig:affenl}
\end{figure}

We will require a different type of boundary for the second enlargement. 
\begin{definition}[2-enlargement boundary]\label{def:bdenl}
   We define $\partial_\rs^\enlb R_k^\enlb(i,\tau)$, the 2-enlargement boundary of $R_k^\enlb(i,\tau)$, 
   as the set of space-time points $(x,s)\in R_k^\enlb(i,\tau)$ such that $(x,s)\in R_k^\core(i',\tau')$ for some $(i',\tau')$ with $R_k(i',\tau')\subset R_k^\enlb(i,\tau)$ 
   and $R_k(i',\tau')\cap \partial_\rs R_k^\enlb(i,\tau)\neq \emptyset$.
\end{definition}


\subsection{Feasible Paths}
In this subsection we introduce the concept of feasible paths. For any graph $G=(V,E)$, we denote the neighbors of a vertex $v\in V$ by $\mathcal{N}_G(v)=\{w\in V:(v,w)\in E\}$.
A path $\mathcal{P}:\mathbb{R}^+\to V$ on a graph $G=(V,E)$ is a \textit{c\`adl\`ag} function of time such that for any $s\in\mathbb{R}^+$, 
if we take $s'$ to be the smallest value that is larger than $s$ and such that $\mathcal{P}(s')\ne\mathcal{P}(s)$ then $\mathcal{P}(s')\in\mathcal{N}_{G}(\mathcal{P}(s))$.
Note that a path, as defined above, does not consider whether edges are open or closed and is thus allowed to jump across closed edges. The same is true in the definition below.
Recall the definition of the time intervals $\dbar{T}_1^j$ from Definition~\ref{def:dbar} and the inner part $S_1^\inn(i)$ of box $i$ from~\eqref{eq:sinn}.


\begin{definition}[Feasible path]
   \label{feaspathdef}
   A path $\mathcal{P}$ is said to be \emph{feasible} if for any times $s,s'$ with $s,s'\in \dbar{T}_1^j(\tau)\subset T_1^\core(\tau)$ for some $j$ and $\tau>0$, and 
   such that $\mathcal{P}(s)\in S_1^\inn(i)$ for some $i\in\Z^d$ for which $R_1(i,\tau)$ is a good box, then 
   \begin{equation*}
       \|\mathcal{P}(s')-\mathcal{P}(s)\|_{1}<\log^2{\ell}.
   \end{equation*}
\end{definition}
Intuitively, a feasible path can move at most $\log^2 \ell$ during any interval $\dbar{T}_1^j$ in which it is inside good boxes. Even though the definition of feasible paths does not consider whether edges are open or closed,
this is aligned with the fact that in good boxes the clusters are of size at most $\log^2\ell$.

We will refer to a path that \textit{leaves the box $R_k(i,\tau)$ from the time boundary} 
as a path $\mathcal{P}$ such that $(\mathcal{P}(s),s)\in R_k^\core(i,\tau)$ for some $s\geq  0$ and if $s'>s$ is the smallest value such that 
$(\mathcal{P}(s'),s')\notin R_k(i,\tau)$ then $\mathcal{P}(s')\in S_k(i)$. 
In other words, it is a path that exits $R_{k}(i,\tau)$ through $\partial_\rt^+R_k(i,\tau)$. 
In the following two lemmas we will prove that a feasible path always leaves good boxes from the time boundary. 
Recall the definition of $\cC_x(I)$, the connected component of $x$ during a time interval $I$, from the paragraph preceding Definition~\ref{def:dbar}.

We define the \emph{spatial core} of a box as follows:
\begin{equation}
   R_k^\score(i,\tau) = S_k^\core(i) \times T_k(\tau).
   \label{eq:score}
\end{equation}
\begin{lemma}
   \label{timeboundlem1}
   For all $p$ small enough (hence, $\ell$ large enough) the following holds. 
   Let $(i,\tau)$ be such that $R_1(i,\tau)$ is a good box. 
   Let $\mathcal{P}$ be a feasible path such that $(\mathcal{P}(s),s)\in R_1^\score(i,\tau)$ for some $s$. 
   Assume that either $s\geq \bar t_1$ or $|\cC_{\cP(s)}(s)|\leq \sqrt{\ell}\log^2 \ell$.
   Then, $\mathcal{P}$ leaves $R_1(i,\tau)$ from the time boundary and 
   $$
      \sup_{s' \in \lrb{s,\max T_1(\tau)}} \|\cP(s)-\cP(s')\|_1 \leq \sqrt{\ell}\log^{3}\ell.
   $$    
\end{lemma}
\begin{proof}
   For any $(v,s)\in R_1^\score(i,\tau)$ and any $(v',s')\in\partial_\rs R_1(i,\tau)$, 
   $\|v-v'\|_1\geq  \ell$ as well as $|s'-s|\leq 3t_1$. 
   Recall that $T_1(\tau)$ is split into intervals $\dbar{T}_1^j(\tau)$ of length $\frac{\gamma}{\mu}$. 
   Assume first that $s\geq \bar t_1$. 
   Then the subinterval $\bar{T}_1^j(\tau)$ containing $s$ consists only of positive times; hence, $\cC_{\cP(s)}(s)$ has cardinality at most $\log^2\ell$. 
   From the definition of feasible paths we obtain
   \begin{equation*}
       \sup_{s'\in\lrb{s,\max T_1(\tau)}}\|\mathcal{P}(s)-\mathcal{P}(s')\|_1\leq  \frac{3t_1\mu}{\gamma} \log^2\ell.
   \end{equation*}
   For all small enough $p$ we have $\frac{3t_1\mu}{\gamma} \log^2\ell = \frac{3\sqrt{\ell}}{\gamma} \log^2\ell \leq \ell\log^3\ell$, recalling that $\gamma$ is set before we take $p$ small enough as 
   in Remark~\ref{rem:gamma}.
   
   If $s\in[0,\bar t_1)$ then $s\in \bar{T}_1^1(0)\subset T_1(0)$ and we use that during $T_1^1(0)$ no edge in $S_1(i)$ opens. Therefore,
   \begin{equation*}
       \sup_{s'\in\lrb{s,\max T_1(\tau)}}\|\mathcal{P}(s)-\mathcal{P}(s')\|_1\leq  |\cC_{\cP(s)}(s)|+\frac{3t_1\mu}{\gamma} \log^2\ell\leq \sqrt{\ell}\log^3\ell.
   \end{equation*}
\end{proof}

The next lemma is the analogous of Lemma~\ref{timeboundlem1} for higher scales.
\begin{lemma}
   \label{timeboundlem}
%
   For all $m$ large enough and all $p$ small enough with respect to $m$, the following holds. 
   Let $k\geq 2$ and 
   let $(i,\tau)$ be such that $R_k(i,\tau)$ is a good box. 
   Let $\mathcal{P}$ be a feasible path such that $(\mathcal{P}(s),s)\in R_k^\score(i,\tau)$ for some $s$.
   Assume that either $s\geq\bar t_1$ or $|\cC_{\cP(s)}(s)|\leq \sqrt{\ell}\log^2 \ell$.   
   Then $\mathcal{P}$ leaves $R_k(i,\tau)$ from the time boundary.
   Moreover, while the path is inside the box, from time $s$ up to time $(\tau+2)t_k$, the path must be within distance $\ell_k/9$ from $\mathcal{P}(s)$.
\end{lemma}
\begin{proof}
   For any $(v,s)\in R_k^\score(i,\tau)$ and any $(v',s')\in\partial_\rs R_k(i,\tau)$, $\|v-v'\|_1\geq  \ell_k$ as well as $|s'-s|\leq 3t_k$. We do a proof by induction on $k$. 
   The case $k=1$ is be a direct consequence of Lemma~\ref{timeboundlem1}. 
   We will actually assume a slightly stronger induction hypothesis. 
   Take $c_1=\frac{1}{100}$, $c_2=\frac{1}{10}$ and, for $j\geq 2$, set $c_{j+1}=c_{j}+\frac{11}{m j^2}$. 
   We take $m$ large enough so that $c_j\leq \frac{1}{9}$ for all $j\geq 1$.
   Now, for a scale $k$, assume that if $\mathcal{P}(s)\in R_k^\score(i,\tau)$ for some $(i,\tau)$ and some $s$ as in the statement of the lemma, 
   then $\sup_{s'\in\lrb{s,(\tau+2)t_k}}\|\mathcal{P}(s)-\mathcal{P}(s')\|_1\leq c_k\ell_k$.
   We want to prove the above for scale $k+1$. 
   
   We split into two cases, starting with $k\geq 2$ (thus, $k+1\geq 3$). 
   Let now $\mathcal{P}$ be a feasible path such that $(\mathcal{P}(s),s)\in R_{k+1}^{\score}(i,\tau)$, and $R_{k+1}(i,\tau)$ is a good box. 
   Thus there are no pairs of non intersecting bad boxes of scale $k$ inside $R_{k+1}(i,\tau)$. 
   By Remark~\ref{enl1shield}, if $R_{k+1}(i,\tau)$ contains at least one bad box $R_k(i',\tau')$, then all bad boxes contained in $R_{k+1}(i,\tau)$ are contained in 
   $R_k^\enl(i',\tau')$. Inside $R_k^{\enl}(i',\tau')$ a feasible path has no restriction on how quickly it can move and it could potentially traverse $S_k^{\enl}(i')$ instantaneously. 
   The remaining boxes of scale $k$ that are in $R_{k+1}(i,\tau)$ are good and by the inductive hypothesis we can use that in these ones the maximum displacement of the path is bounded above by $c_k\ell_k$, so for $k\geq 2$ it follows that
   \begin{align*}
       \sup_{s'\in\lrb{s,(\tau+2)t_{k+1}}}\|\mathcal{P}(s)-\mathcal{P}(s')\|_1\leq &9\ell_k+\lr{12+\frac{2t_{k+1}}{2t_k}}c_k\ell_k\\
       \leq & 9\ell_k+\lr{12+k^2m}c_k\ell_k\\
       =&\left(\frac{9}{mk^2}+\frac{12 c_k}{mk^2}+c_k\right)\ell_{k+1}\\
       &\leq c_{k+1}\ell_{k+1}.
   \end{align*}
   The term $\lr{12+\frac{2t_{k+1}}{2t_k}}$ accounts for the following boxes. 
   Each time the path $\mathcal{P}$ finds itself at the starting time of a $k$-core, 
   it spends at least time $2t_k$ inside the corresponding $k$-box, and after that amount of time finds itself at the starting time of 
   another $k$-core. 
   This gives at most $\frac{2 t_{k+1}}{2 t_k}$ $k$-cores for which we can apply the induction hypothesis. 
   There are situations, however, that we cannot guarantee that the path $\mathcal{P}$ is at the 
   \emph{starting} time of a $k$-core. One first situation is the very first box. We can still apply the induction hypothesis in such cases, 
   since the hypothesis requires only that the path is inside the spatial core, regardless of it being the starting time of a core or not, but 
   can give rise to at most $12$ additional boxes to the counting: the first box, the boxes right before and after $R_k^\enl(i',\tau')$, and the $9$ time intervals contained in $R_k^\enl(i',\tau')$.
   
   
   Now we turn to the last case, which is $k=1$ (that is, $k+1=2$). 
   We proceed in the same way as before, but taking care of the fact that boxes at scale $1$ have a different length in the time dimension. 
   We have
   \begin{align*}
       \sup_{s'\in\lrb{s,(\tau+2)t_{2}}}\|\mathcal{P}(s)-\mathcal{P}(s')\|_1\leq 
       &9\ell_1+\lr{12+\frac{2t_{2}}{t_1}}c_1\ell_1\\
       \leq & 9\ell_1+\lr{12+2m}c_1\ell_1\\
       =&\left(\frac{9}{m}+\frac{12c_1}{m}+2c_1\right)\ell_{2}\\
       &\leq c_{2}\ell_{2}.
   \end{align*}
   The proof is concluded by taking $m$ large enough to guarantee that 
   $c_{k+1}=\frac{1}{100}+\frac{1}{10}+\frac{20}{m}\sum_{i=2}^k\frac{1}{i^2}<\frac{1}{9}$ for all $k$.
\end{proof}


Next we will prove that if a feasible path enters the 2-enlargement of a box $R_k(i,\tau)$ from its spatial boundary, 
and all the $k$-boxes inside $R_k^\enlb(i,\tau)\setminus R_k^\enl(i,\tau)$ are good, then the path remains far from the box $R_k(i,\tau)$.
For this lemma, recall the definition of $\partial_\rs^\enlb R_k^\enlb(i,\tau)$ the 2-enlargement boundary of $R_k(i,\tau)$ from Definition~\ref{def:bdenl}.

\begin{lemma}
   \label{spaceenl}
   Let $\mathcal{P}$ be a feasible path such that $(\mathcal{P}(s),s)\in \partial_\rs^\enlb R_k^{\enlb}(i,\tau)$ for some $(i,\tau)$ and $s\geq0$. 
   Assume that either $s\geq\bar t_1$ or $|\cC_{\cP(s)}(s)|\leq \sqrt{\ell}\log^2 \ell$.
   Assume also that $R_k(i',\tau')$ is good for all $R_k(i',\tau')\subset R_k^\enlb(i,\tau)\setminus R_k^\enl(i,\tau)$. 
   Then,
   $$
      \inf_{\substack{v\in S_k^{\enl}(i)\\s'\geq s,\;s'\in T_k^{\enlb}(\tau)}}\|\mathcal{P}(s')-v\|_1\geq  12\ell_k.
   $$
\end{lemma}
\begin{proof}
   By hypothesis every box $R_k(\cdot,\cdot)\subset R_k^{\enlb}(i,\tau)\setminus R_k^{\enl}(i,\tau)$ is good.
   For any $s'\in T_k^{\enlb}(\tau)$, $|s-s'|\leq  24t_k$. 
   Assume without loss of generality that during $[s,s']$ the path never visits a box $R_k(\cdot,\cdot)$ which is not contained in $R_k^\enlb(i,\tau)$; otherwise we can carry out the proof separately to each portion of the 
   path that only traverses boxes $R_k(\cdot,\cdot)$ contained in $R_k^\enlb(i,\tau)$. 
   Let $(i',\tau')$ be such that $(\mathcal{P}(s),s)\in R_k^\core(i',\tau')$.
   Since $R_k(i',\tau')$ is a good box, letting $s''=\sup\lrc{T_k(\tau')}$ we have that 
   $\|\mathcal{P}(s)-\mathcal{P}(s'')\|_1\leq  \frac{\ell_k}{9}$ by Lemma~\ref{timeboundlem}. 
   If $s''< s'$, 
   we can iterate the above argument obtaining that 
   $\|\mathcal{P}(s)-\mathcal{P}(s')\|_1\leq  24 \frac{\ell_k}{9}$, where $24$ amounts for the largest number of iterations. 
   Since boxes have length $2t_k$ in the time dimension, it would be enough to replace $24$ by $12$ for $k\geq 2$, but we just use the larger bound $24$ to accommodate also the $k=1$ case, for which the 
   length of a box in the time dimension is smaller.
   
   Now since $\frac{24 \ell_k}{9}$ is smaller than $13\ell_k$, which is the distance between $\mathcal{P}(s)$ and the spatial boundary of $S_k^\enl(i)$ enlarged by all boxes $R_k(\cdot,\cdot)$ that intersects it, 
   the path can only traverse good $k$-boxes while inside $R_k^\enlb(i,\tau)$. 
   In addition, for any $ v\in S_k^{\enl}(i)$ one has 
   \begin{equation*}
      \|v-\mathcal{P}(s')\|_1
      \geq \|v-\mathcal{P}(s)\|_1-\frac{24\ell_k}{9} 
      \geq 15\ell_k-\frac{24\ell_k}{9} 
      \geq  12\ell_k.
   \end{equation*}
\end{proof}


\subsection{Great Boxes}

We will need a stroger notion for boxes of scale $1$, which we will call \emph{great} boxes.
\begin{definition}
A box $R_1(i,\tau)$ is said to be $k$-great if for all $k'\leq  k$, for all $R_{k'}(i',\tau')$ such that $R_{k'}^{\enlb}(i',\tau')$ intersects $R_1(i,\tau)$ then $R_{k'}(i',\tau')$ is good. Moreover, we define \begin{equation*}
    G_k\defn\{(i,\tau):R_1(i,\tau)\text{ is }k\text{-great}\}
\end{equation*}
to be the set of $k$-great boxes.
\end{definition}

Later we will see that the walker has to traverse a feasible path. 
The next lemma will be used to say that if the walker traverses a $k$-box that is good with a large neighborhood of good $k$-boxes, then it is necessarily the case that the walker has to traverse enough $k$-great boxes. 
Such great boxes will be the places where we will attemp a simple random walk coupling later. 
\begin{lemma}
   \label{gbcross}
   Let $\mathcal{P}$ be a feasible path such that $(\mathcal{P}(\tau t_k),\tau t_k)\in \partial_\rt^- R_k^\core(i,\tau)$ for some $(i,\tau)$.  
   Assume that either $\tau>0$ or $|\cC_{\cP(\tau t_k)}(\tau t_k)|\leq \sqrt{\ell}\log^2\ell$.
   Then there exists $C_\newcte{cte:gbc}>0$ such that letting $r=C_{\ref*{cte:gbc}}\frac{t_k}{t_1}$ we can find times $s_1<\dots<s_r$ and distinct space-time indices $(i_1,\tau_1),\dots,(i_r,\tau_r)$ 
   such that the following all hold:
   \begin{itemize}
      \item For all $j$, $(\mathcal{P}(s_j),s_j)\in \partial_\rt^-R_1^\core(i_j,\tau_j)$ and $\mathcal{P}$ exits $R_1(i_j,\tau_j)$ from the time boundary.
      \item $R_1(i_j,\tau_j)\subset R_k(i,\tau)$ are $k$-great for all $j$.
      \item $\tau_1 t_1\geq \tau t_k$ and $\tau_j \geq \tau_{j-1}+2$ for all $j\in\lrc{2,3,\ldots,r}$.
   \end{itemize}
\end{lemma}

\begin{proof}
First, note that if a box $R_1(\tilde i, \tilde \tau)$ is $(k-1)$-great and is contained in $R_k(i,\tau)$ then it is also $k$-great.
We will prove the statement of the lemma replacing $C_{\ref*{cte:gbc}}$ with $c_k$, some function of $k$. 
Then the lemma follows by showing that there is a universal value $C_{\ref*{cte:gbc}}$ such that $0<C_{\ref*{cte:gbc}}\le c_k$ for all $k$. 
We will do a proof by induction on $k$. Case $k=1$ is trivially verified by choosing $c_1=1$ because in this case $r=c_1=1$
and we take $(i_1,\tau_1)=(i,\tau)$.

Now, for $k\geq2$, assume the lemma is true up to scale $k-1$ 
and consider a feasible path that at time $\tau t_k$ is inside $\partial_\rt^- R_{k}^\core(i,\tau)$ such that every box of scale $k$ whose $2$-enlargement 
intersects $R_{k}(i,\tau)$ is good. 
By Remark~\ref{enl1shield}, the bad boxes of scale $k-1$ inside $R_{k}(i,\tau)$ (if there are any) 
are all contained in $R_{k-1}^{\enl}(i',\tau')$ for some $i',\tau'$. 
We then regard all boxes of scale $1$ which are in at least one of the $2$-enlargement of the boxes contained in $R_{k-1}^\enl(i',\tau')$ 
as potentially not $(k-1)$-great.
By Lemma \ref{timeboundlem} 
we know that $\mathcal{P}$ crosses $\partial_\rt^+R_{k}(i,\tau)$ before $\partial_\rs R_{k}(i,\tau)$. 
In words the path stays for time at least $2t_{k}$ in the box $S_{k}(i)$. 

Since the path starts from $\partial_\rt^- R_{k}^\core(i,\tau)$, 
it starts on $\partial_\rt^- R_{k-1}^\core(i'',\tau'')$ for some $i''$ and $\tau''$.
In this $(k-1)$-box we can apply the inductive hypothesis, so after time $2t_{k-1}$ the path has gone through at least 
$c_{k-1}\frac{t_{k-1}}{t_1}$ distinct $(k-1)$-great boxes. 
Since this path remains inside $R_{k}(i,\tau)$, 
we immediately obtain that such boxes are all $k$-great boxes.
When the path reaches $\partial_\rt^+R_{k-1}(i'',\tau'')$ we have that the path is now on 
$\partial_\rt^-R_{k-1}^\core(i''+j,\tau''+2)\subset R_k(i,\tau)$ for some $j\in\Z^d$ and from here we can reapply the inductive hypothesis. 
So it remains to count how many times we can iterate this procedure before $2t_{k}$ amount of time has passed. 

To do this, we first count how much time the path can spend inside the $2$-enlargement of a bad $(k-1)$-box. 
It suffices to count how much time is spanned by the boxes whose $2$-enlargements intersect $R_{k-1}^\enl(i',\tau')$, which is 
$|T_{k-1}^\enl(\cdot)|+2|T_{k-1}^\enlb(\cdot)|=57t_{k-1}$. Hence, the number of times the above procedure can be iterated is at least 
$$
   \frac{2t_{k}-57t_{k-1}}{2t_{k-1}} = \frac{t_{k}}{t_{k-1}}\lr{1-\frac{57}{2m(k-1)^2}}.
$$
From the inductive hypothesis, the path will traverse at least 
$$
   \frac{t_{k}}{t_{k-1}}\lr{1-\frac{57}{2m(k-1)^2}}c_{k-1}\frac{t_{k-1}}{t_1}
   =c_{k}\frac{t_{k}}{t_1}
$$
$(k-1)$-great boxes, by setting 
$c_{k}=\lr{1-\frac{57}{2m(k-1)^2}}c_{k-1}$.
These boxes are $k$-great by the properties of $R_k(i,\tau)$.
The lemma is then concluded by setting 
$$
   C_{\ref*{cte:gbc}} = \lim_{k\to\infty}c_k=\prod_{i=1}^\infty \lr{1-\frac{57}{2mi^2}}>0.
$$
\end{proof}


\section{Overview of the Proof}\label{sec:overview}

In this section we give a high-level description of the proof. 
Consider two processes $\{M_t^\star\}_{t\ge0}=\{X_t,\eta_t^\star\}_{t\ge0}$ and 
$\{\bar{M}_t^\star\}_{t\ge0}=\{\bar X_t,\bar{\eta}_t^\star\}_{t\ge 0}$ 
with starting states $M_0^\star,\bar{M}_0^\star\in\Omega^\star$. 
We will construct a coupling of the two processes so that for some time $T=\Delta_1+\Delta_2+\Delta_3$ of order $\frac{n^2}{\mu}$ the two configurations agree with positive probability. 
Since $\{M_t\}_{t\geq 0}$ can be recovered from $\{M_t^\star\}_{t\geq 0}$,
by sampling independently the edges with status $\star$, we will obtain our result.

The coupling will consist of three different phases which we will describe in a high level way below.
The coupling of each phase will have a small, albeit positive, probability of failing. If the coupling of a phase fails, 
we declare the whole three-phase procedure to have failed, let the two processes evolve arbitrarily until time $T$ and restart everything again from phase 1. 
The detailed analysis of each phase will be given in sections \ref{firstphase}, \ref{secondphase} and \ref{thirdphase}. 
Then in section \ref{proofteo}, we will put all phases together and complete the proof of Theorem \ref{thm:main}.

\subsection{First phase: the local coupling}
During the first phase we let the two processes evolve \emph{independently}, and wait for the first time the graphs of the two processes agree on a ball of radius $2\ell$ around the walkers, that is, we wait for a time $t$ such that 
\begin{equation*}
    \eta_t^\star(X_t+e)=\bar{\eta}_t^\star(\bar X_t+e),
\end{equation*}
for all edges $e\in E(B_{2\ell}^\infty(0))$, where $B_r^\infty(x)$ is the vertices inside the $L_\infty$ ball of radius $r$ around $x$. 
We will show in Lemma \ref{lemmaf1} that this will happen within time $\Delta_1$ with large enough probability, where $\Delta_1$ has order $\frac{\log^2 n}{\mu}$. 
This is the shortest of the three phases. 

If the first phase does not end within time $\Delta_1$, we declare the whole three-phase procedure to have failed. 
This phase will be handled in section \ref{firstphase}.

\subsection{Second phase: the non-Markovian coupling of the walkers}
This is the most involved phase. 
After the first phase has been completed successfully, the graphs of the two processes are the same on a ball of radius $2\ell$ around the walkers. Then, in the second phase we wish to couple the motion of the walkers. 
We use the tessellation to decide when to couple the walkers identically (so that they jump in the same way) 
and when to perform a better coupling aiming to decrease the distance between the walkers.

Intuitively, whenever the walkers are passing through a ``bad'' region of the environment (which in our case will be the 2-enlargement of a bad box) 
we will just do the identity coupling to make sure the distance between the walkers does not increase. 
In fact, we will only be able to do the identity coupling because 
we will use the annulus between the 2-enlargement of the bad box and the bad box itself (which is composed of good boxes) 
to give time for the graphs around the walkers to get coupled in both configurations, allowing the identity coupling to be carried out. 
If instead the two walkers are in a great box, then we try to do a better coupling, which we shall refer to as a \emph{simple random walk moment}.

More precisely, translate time so that the second phase starts from time $0$. 
Then, we create the multi-scale tessellation describe in Section~\ref{tessellation} up to time $\Delta_2+\Delta_3$ where $\Delta_2$ and $\Delta_3$ are of order $\frac{n^2}{\mu}$.
We will fix a largest scale $k_\rmax$ and will look at how many times the walkers enter $k_\rmax$-great boxes. 

When the walkers are in great boxes, Lemma \ref{srwmprob} will give that the environment is favourable enough so that with positive probability 
the displacement of the walkers will have the same distribution as that of a simple random walk on $\mathbb{T}_n^d$ (i.e., where all edges are open). 
Phase 2 ends at time $\Delta_2$ where we check whether the walkers are coupled and the graphs are coupled on a ball of radius $2\ell$ around the walkers. 

Lemma~\ref{gbcross} says that the walkers will cross an order of $n^2$ great boxes during $[0,\Delta_2]$ and, therefore, by time $\Delta_2$ the walkers are expected to have done an order of $n^2$ simple random walk steps. 
Since two simple random walkers on $\mathbb{T}_n^d$ can be coupled in a way that they coalesce after a time of order $n^2$, we can ensure that with high probability phase 2 ends successfully. 
The details are carried out in section \ref{secondphase}.

\subsection{Third phase: the coupling of the graphs}
The third phase starts are time $\Delta_2$; as before we translate time so that the second phase starts at time 0. 
At the beginning of the third phase the walkers are coupled and the graphs are coupled as well on a ball of radius $2\ell$ around them. 
The idea of this phase is to keep performing identity coupling until the graphs couple together everywhere. We will show that this simple idea works. 

There is one tricky issue. During the second phase, we needed to construct the tessellation all the way to time $\Delta_2+\Delta_3$, 
while the second phase ends already at time $\Delta_2$. The reason for this is that, in order to know whether we can perform a simple random walk moment, 
we need to observe a little bit of future information about the environment. 
Therefore, as we performed the second phase, we observed some information from the updates after the end of phase two. 

So the goal of the third phase is simply to let time pass until we get to a point where no information regarding future times has been observed, meanwhile doing identity coupling.
With this, during the third phase we aim to keep the walkers coupled at all times, 
while we finish to couple the graphs before time $\Delta_2+\Delta_3$. 

We do not use any further information from the tessellation than what we already observed for phase 2. 
The delicate point is that in order to apply identity coupling of the walkers, as we explained in the second phase, we have to ensure that the graphs around the walkers are coupled. 
How large a region we require to be coupled depend on the environment of good and bad boxes that is ahead of the walker, 
but now we cannot observe anything beyond what we have already observed in phase two; 
otherwise we would keep observing future information.

As hinted above, we just proceed with the identity coupling ``blindly''. That is, we perform identity coupling up to time $\Delta_2+\Delta_3$ assuming that any information that we have not yet observed is ``good'', 
and simply ``hope for the best''. It will turn out that this procedure succeeds with large probability leaving the two processes completely coupled (both the graphs and the walkers) by time $\Delta_2+\Delta_3$.
The details of this phase are given in Section~\ref{thirdphase}.
\subsection{What if a phase fails?}
If any of the three phases does not successfully end, we let the two processes run independently (modulo what has already been observed) 
until the end of the third phase. 
This is needed as we might have observed some information about the environment up to that time. 
After that, we repeat the procedure from phase $1$. 
Since the three phases succeed with positive probability, we only need to repeat the whole procedure a constant number of times. 
The end of the proof of the upper bound is given in section \ref{proofteo}.

\section{First Phase}\label{firstphase}

During the first phase we let the processes $ M_t^\star=(X_t, \eta_t^\star)$ and $\bar M_t^\star=(\bar X_t, \bar{\eta}_t^\star)$ evolve independently.
Let $\Psi_t:V\to V$ be the translation that maps $X_t$ into $\bar X_t$; 
we will abuse notation and use the same $\Psi_t$ to denote the corresponding translation map $E\to E$ of the edges. 
For any $i\in \Z_+\cup\lrc{\infty}$, we define
\begin{equation*}
    E(B_r^i(v))=\{(v_1,v_2)\in E:v_1,v_2\in B^i_r(v)\}
    \quad\text{and}\quad
    B^i_r(v)=\{v_1\in V:\|v-v_1\|_i\leq  r\};
\end{equation*}
thus $E(B_r^i(v))$ is the set of edges in the ball of radius $r$ around $v$ according to the norm $L_i$. We omit $i$ from the superscript whenever $i=1$. Define the event
\begin{align}
    \cB_t\defn
    &\lrc{\forall e\in E\lr{B_{1}(X_{t})},\,{\eta}_{t}^\star(e)=\bar{\eta}_{t}^\star\lr{\Psi_{t}(e)}=0}\nonumber\\
    &\cap\lrc{\forall e\in E(B_{2\ell}^\infty(X_{t}))\setminus E(B_{1}(X_{t})),\,{\eta}_{t}^\star(e)=\bar{\eta}_{t}^\star(\Psi_{t}(e))=\star} \label{eq:eventb}
\end{align}
that the edges in a $L_\infty$ ball of radius $2\ell$ around the walkers are all $\star$ at time $t$, except for the ones adjacent to the walkers which are closed; recall that $\ell$ is the size of the 
core of boxes of scale $1$, whose value is given in~\eqref{ell}. Let 
\begin{equation}
\label{taubdef}
    \tau_B\defn\inf\left\{t\geq 0:\mathcal{B}_t\text{ holds}\right\}.
\end{equation}
Note that $\tau_B$ is a stopping time.
Define $\Delta_1\defn\frac{C_{\newcte{cte:p1}}\log^2 n}{\mu}$, for some constant $C_{\ref*{cte:p1}}(p)>0$, and define the event
\begin{equation}
    F_1\defn\{\tau_B<\Delta_1\},
    \label{eq:f1}
\end{equation}
which we shall take as the event that phase 1 succeeds. This event is a bit more restricted than the one announced in the previous section, but this will be convenient for us in the next phase. 

We then run phase 1 until $\tau_B$ or $\Delta_1$, whichever occurs first.
If it turns out that $F_1$ does not occur, phase 1 is then stopped at time $\Delta_1$ and we declare the whole procedure to have failed at time $\Delta_1$.
In this case, we do not proceed to the second phase, and define $\Delta_1$ as the \emph{failing time} of the procedure and, as we will explain more thoroughly in Section~\ref{proofteo}, we will
restart from phase 1 from $\Delta_1$.

The following lemma establishes the probability that the first phase is successful.
\begin{lemma}[Phase 1 success probability]
\label{lemmaf1}
   For any $\delta>0$, there exists $p_0=p_0(d,\delta)>0$ such that for any $p<p_0$, there exists $C_{\ref*{cte:p1}}(p,d,\delta)>0$ in the definition of $\Delta_1$ so that 
   for any initial configurations ${\eta}_0^\star,\bar{\eta}_0^\star\in \Omega^\star$ we obtain 
   \begin{equation*}
       \mathbb{P}\left(F_1^c\right)\leq  \delta
   \end{equation*}
   for all large enough $n$.
\end{lemma}

Before showing that phase 1 succeeds with good probability, we need to establish a simple result on percolation. 
We then prove Lemma~\ref{lemmaf1} in Section~\ref{sec:lemmaf1}.

\subsection{Percolation on cylinders and open upwards paths}
Let $G=(V,E)$ be a finite graph whose maximum degree is $d_\rmax$; in our case, it would be enough to take $G$ to be the $d$-dimensional torus $\T_n^d$ of side length $n$, where nearest-neighbors are defined according to the 
$\ell_\infty$ norm. We consider the discrete cylinder $V_\cyl=V\times \Z_+$ and define a site percolation process on $V_\cyl$ with parameter 
$\varrho>0$. In other words, we declare each site of $V_\cyl$ to be \emph{open} with probability $\varrho$, independently of one another; a vertex that is not open is said to be \emph{closed}.
For two vertices of $x,y\in V$, we write $x\sim y$ to denote that the graph distance between $x$ and $y$ is at most $1$ in $G$; thus, for example, $x\sim x$ for all $x\in V$.
\begin{definition}[open upwards path]
   An \emph{open upwards path} in $V_\cyl$ is a sequence of sites $(i_0,\tau_0),(i_1,\tau_1),(i_2,\tau_2),\ldots,(i_r,\tau_r)$ such that $i_j\in V$, $\tau_j\in\Z_+$, 
   $i_j \sim i_{j+1}$ and the following holds for all $j$. If $(i_j,\tau_j)$ is open, then $\tau_{j+1}=\tau_j+1$; otherwise, $\tau_{j+1}-\tau_j\in\lrc{0,1}$. 
   In other words, the path is compelled to move ``upwards'' in the cylinder when it visits open sites.
\end{definition}

Note that an open upwards path is allowed to visit a vertex more than once.
We say that an open upwards path $(i_0,\tau_0),(i_1,\tau_1),(i_2,\tau_2),\ldots,(i_r,\tau_r)$ \emph{traverses $m$ levels} if $\tau_r-\tau_0=m$.
When $\varrho$ is close to $1$, an open upwards path cannot visit too many closed sites. This is quantified in the next lemma.
\begin{lemma}[Open upwards path]\label{lem:oup}
   Let $(i_0,0)\in V_\cyl$ be fixed. 
   For any $\alpha>0$, there exists $\varrho'=\varrho'(\alpha,d_\rmax)\in(0,1)$ such that if $\varrho> \varrho'$ then the probability that there exists an open upwards path from $(i_0,0)$ that traverses 
   $m$ levels and visits at least $\alpha m$ distinct closed sites is at most $e^{-cm}$ for some constant $c=c(\alpha,d_\rmax)>0$. 
\end{lemma}

Before proving the above result, we need the following estimate in the number of subgraphs of $G$ that contain a given vertex.
\begin{lemma}\label{lem:animal}
   Given a vertex $v\in V$, let $A_r$ be the number of induced connected subgraphs of $V$ containing $v$ and having $r$ vertices.
   There exists a constant $c=c(d_\rmax)>0$ such that $A_r \leq e^{c r}$ for all $r$. 
\end{lemma}
\begin{proof}
   This proof is quite standard and a version for the lattice can be found in~\cite[Proof of Theorem 4.20]{Grimmett}; we include a proof here for the sake of completeness.
   Let $A_{r,s}$ be the number of induced connected subgraphs of $V$ containing $v$ and having $r$ vertices and $s$ boundary vertices, where a boundary vertex is a vertex that does not belong
   to the subgraph but has a neighbor who does. Hence, $A_r = \sum_s A_{r,s}$. 
   Note that for any $\varrho\in(0,1)$, if we perform percolation on $V$, we obtain 
   \begin{equation}
      \sum_{r,s} A_{r,s} \varrho^r (1-\varrho)^s=1. 
      \label{eq:animal}
   \end{equation}
   For any vertex $u\in V$ denote $d(u)$ its degree in $G$. 
   Let $S\in V$ be one subgraph counted in $A_{r,s}$, denote $m(S)$ the number of edges between vertices of $S$ and $\partial S$ the number of edges between vertices in $S$ and vertices in $V\setminus S$.
   Then, 
   $$
      d_\rmax r \geq \sum_{u \in S} d(u) = 2 m(S) + \partial S
      \geq 2 \lr{r-1} + s.
   $$
   Thus, $s\leq \lr{d_\rmax -2}r+2$. Plugging this result into~\eqref{eq:animal} and taking $\varrho=\frac{1}{2}$ we obtain
   $$
      1 \geq \sum_{r,s} A_{r,s} \varrho^r (1-\varrho)^{\lr{d_\rmax -2}r+2}
      =\sum_{r,s} A_{r,s} 2^{-\lr{d_\rmax -1}r-2}
      =\sum_{r} A_{r} 2^{-\lr{d_\rmax -1}r-2}.
   $$
   Thus, $A_{r}\leq 2^{\lr{d_\rmax -1}r+2}$ for each $r$.
\end{proof}

\begin{proof}[Proof of Lemma~\ref{lem:oup}]
   Let $m$ be an integer and, for convenience, set $\tau_0=0$. 
   Consider an open upwards path $(i_0,\tau_0),(i_1,\tau_1),(i_2,\tau_2),\ldots,(i_{r},\tau_{r})$ such that $\tau_{r}-\tau_0=m$; that is, the 
   path traverses $m$ levels. Let $s_1$ be the number of distinct closed sites visited by the path before it traverses $1$ level, 
   and for $j\geq 2$ let $s_j$ be the number of distinct closed sites visited by the path after having traversed $j-1$ levels and before traversing $j$ levels. 
   So $\sum_{j=1}^m s_j$ is the total number of distinct closed sites visited by the path. 
   Note that, the sites counted in each $s_j$ must be of the form $(\cdot, j-1)$ and must form a connected set with respect to the relation $\sim$ over $G$.
   Using Lemma~\ref{lem:animal}, given $s_1,s_2,\ldots,s_m$, the number of possibles ways to pick the set of distinct sites within $(i_0,\tau_0),(i_1,\tau_1),\ldots$ is 
   $$
      \lr{d_\rmax+1}^{m} \prod_{j=1}^m s_je^{c' s_j},
   $$
   where $c'$ is the constant from Lemma~\ref{lem:animal}, and the term $s_j$ in the product counts the number of sites at level $j$ that can be selected
   to be the last vertex visited by the path before going to level $j+1$. 
   Then, $\lr{d_\rmax+1}^{m}$ accounts for the number of ways to choose the first site at level $j$ given the last site at level $j-1$; this amounts to at most $d_\rmax+1$ choices per level. 
   If we fix $\sum_{j=1}^m s_j=S$, the number of ways to select the $s_j$ is $\binom{S+m-1}{S}$. Finally, given all sites in the path with $s_j$ as  defined above, the probability that this path is an 
   open upwards path is at most $\prod_{j=1}^m(1-\varrho)^{s_j}$ since each site counted in the $s_j$ must be closed. Therefore,       
   the expected number of open upwards paths that traverse $m$ levels and visit at least $\alpha m$ closed sites is at most 
   \begin{align*}
      &\sum_{S\geq \alpha m}\binom{S+m-1}{S} \lr{d_\rmax+1}^{m}\prod_{j=1}^m e^{c's_j}s_j (1-\varrho)^{s_j}\\
      &=\lr{d_\rmax+1}^{m}\sum_{S\geq \alpha m}\binom{S+m-1}{S} \lr{c''e^{c''} (1-\varrho)}^{S},
   \end{align*}
   where we used that given $c$ there exists a constant $c''$ such that $ze^{c'z}\leq c''e^{c''z}$ for all $z$. 
   It is enough to use the trivial bound $\binom{S+m-1}{S}\leq 2^{S+m-1}$ in the above expression to obtain the upper bound 
   $$
      \lr{2(d_\rmax+1)}^m \sum_{S\geq \alpha m}\lr{2c''e^{c''} (1-\varrho)}^{S}
      \leq 2 \lr{2(d_\rmax+1)}^m\lr{2c''e^{c''} (1-\varrho)}^{\alpha m},
   $$   
   with the inequality hold whenever $\varrho$ is close enough to $1$ so that $2c''e^{c''} (1-\varrho)\leq \frac{1}{2}$.
   Then the lemma holds by setting $\varrho$ further closer to $1$ so that 
   $4(d_\rmax+1)\lr{2c''e^{c''} (1-\varrho)}^{\alpha}\leq e^{-c}$.
\end{proof}

By using a result by Liggett, Schonmann and Stacey~\cite{LSS}, the above result can be extended to percolation on $\T_d\times \Z_+$ with bounded dependences.
\begin{lemma}\label{lem:oup2}
   Let $C\geq 1$ be a constant. Consider a site percolation process on $\T_d\times \Z_+$ where the probability that a given site is open depends on at most $C$ other sites.
   Then Lemma~\ref{lem:oup} holds with the lower bound on $\varrho$ depending on $C$.
\end{lemma}
\begin{proof}
   For any $\tilde \varrho$, provided $\varrho$ is large enough we can apply Liggett, Schonmann and Stacey~\cite[Theorem~0.0]{LSS} to obtain that the dependent site percolation process stochastically dominates an independent site percolation process 
   of parameter $\tilde \varrho$. The lemma then follows by applying Lemma~\ref{lem:oup} to this independent site percolation process.
\end{proof}

\subsection{Proof of Lemma~\ref{lemmaf1}}\label{sec:lemmaf1}
Now we are in a position to establish the occurence of the first phase.
\begin{proof}[Proof of Lemma~\ref{lemmaf1}]
   Let $\tau_B'$ be the first time $t\geq t_1$ such that $X_t$ and $\bar X_t$ are both isolated, meaning that all edges adjacent to them are closed.
   We will show that $\tau_B'$ occurs before time $\Delta_1-t_1+\bar t_1$.
   
   For each process $M_t^\star$ and $\bar M_t^\star$ we create a tessellation of 
   $\mathbb{T}_n^d \times [0,\Delta_1]$ into boxes of scale $1$ using the values for $\ell$ and $t_1$ from 
   Section~\ref{tessellation}.
   The event that a given box is good is defined as in Definition~\ref{def:goodbox}.
   We let $M_t^\star$ and $\bar M_t^\star$ evolve independently of one another until a stopping time $st_1$ where $X_{st_1}$ and $\bar X_{st_1}$ are both in good boxes and $s\geq 1$.
   Note that good $1$-boxes form a dependent site percolation process on $\T^d_n \times \Z_+$ so that we can apply Lemma~\ref{lem:oup2}.
   Let $(i_0,1)$ and $(\bar i_0,1)$ be the boxes visited by $X_t$ and $\bar X_t$ at time $t= t_1$. Now, since random walks must traverse a feasible path, and since feasible paths leave
   good boxes from the time boundary (cf.\ Lemma~\ref{timeboundlem1}), we obtain that from $(i_0,1)$ and $(\bar i_0,1)$ the random walks $X_t$ and $\bar X_t$ must traverse an open upwards path.
   Therefore, the probability that up to level $m=\Delta_1/t_1-1$ we have that $X_t$ and $\bar X_t$ each visited more than $\frac{m}{3}$ bad $1$-boxes is at most 
   $2e^{-cm}$ provided $p$ is small enough (which makes the probability that a $1$-box being good large enough). Under this event, there must exist $\frac{m}{3}$ instances of time $s\leq m$ at which $X_{st_1}$ and 
   $\bar X_{st_1}$ are both in good $1$-boxes. when this happens, at time $st_1+\bar t_1$ both $X_{st_1+\bar t_1}$ and $\bar X_{st_1+\bar t_1}$ are isolated in a vertex (i.e., all edges adjacent to them are closed).
   Therefore,
   $$
      \PR\lr{\tau_B' \geq \Delta_1-t_1+\bar t_1}
      \leq n^{2d} \PR\lr{s > m}
      \leq 2ne^{-cm},
   $$
   where the term $n^{2d}$ accounts for the number of choices for $i_0$ and $\bar i_0$.

   Now let $\mathcal{F}$ be the $\sigma$-algebra generated by $M_t^\star$ and $\bar M_t^\star$ during $t\in[0,\tau_B']$. We want to establish a lower bound on the probability that $\cB_{\tau_B'+\bar t_1}$ given 
   $\cF$. Since $X_t$ and $\bar X_t$ are isolated in $t=\tau_B'$, it is enough to compute the 
   probability that all edges inside a $L_\infty$ ball of radius $2\ell$ around the walkers do a $\star$-update but no non-$\star$ update, 
   and the edges adjacent to the walkers do not open or do a non $\star$-update during $[\tau_B',\tau_B'+\bar t_1]$. This probability is 
   $$
      \exp\lr{-\bar t_1 \mu (1-p_\star)2\lr{4\ell}^d} \lr{1-\exp\lr{-\bar t_1 \mu p_\star}}^{2\lr{4\ell}^d} \exp\lr{-\bar t_1 \mu p_\rmin 4d},
   $$
   where the first term is the probability that no edge in the $L_\infty$ ball of radius $2\ell$ around the walkers does a non $\star$-update, 
   the second term is the probability that those edges do a $\star$-update and the last term 
   is the probability that the edges adjacent to the walkers do not open. Therefore, 
   \begin{align*}
      \PR\lr{F_1^\compl}
      &\leq 2ne^{-cm} + 1- \exp\lr{-\bar t_1 \mu (1-p_\star)2\lr{4\ell}^d-\bar t_1 \mu p_\rmin 4d} \lr{1-\exp\lr{-\bar t_1 \mu p_\star}}^{2\lr{4\ell}^d}\\
      &\leq 2ne^{-cm} + 1- \exp\lr{-\bar t_1 \mu (1-p_\star)2\lr{4\ell}^d-\bar t_1 \mu p_\rmin 4d} - 2\lr{4\ell}^d\exp\lr{-\bar t_1 \mu p_\star}.
   \end{align*}
   Recall that $\Delta_1=\frac{C_{\ref*{cte:p1}} \log^2n}{\mu}$, $\bar t_1 = \frac{\log^2\ell}{\mu}$ and $t_1=\frac{\sqrt{\ell}}{\mu}$, where $\ell$ is just a large enough constant that is set before letting $p$ be small enough.
   Now we show that we can make the above smaller than $\delta$.
   We start with the term $2\lr{4\ell}^d\exp\lr{-\bar t_1 \mu p_\star}$, which can be made, say, smaller than $\frac{\delta}{3}$. We will do this by adjusting $\ell$ only, but this term involves also $p$ though $p_\star$.
   However, note that $p_\star$ goes to $1$ as $p$ goes to $0$. So, since $\bar t_1\mu$ is of order $\log^2\ell$, we can choose $\ell$ large enough so that $2\lr{4\ell}^d\exp\lr{-\bar t_1 \mu p_\star}\leq \frac{\delta}{3}$ for all 
   $p$ so that $p_\star\geq \frac{1}{2}$. After fixing $\ell$, we can take $p$ close enough to $0$, which makes $p_\rmin$ goes to $0$ and $p_\star$ goes to $1$, so that 
   $\exp\lr{-\bar t_1 \mu (1-p_\star)2\lr{4\ell}^d-\bar t_1 \mu p_\rmin 4d}\geq 1-\frac{\delta}{3}$. Finally, after fixing $\ell$ and $p$, we can take $n$ large enough so that    
   $2ne^{-cm}\leq \frac{\delta}{3}$ since 
   $m=\frac{\Delta_1}{t_1}-1$ is of order $\log^2n$ as a function of $n$.
   This concludes the first phase.

\end{proof}

\section{The Second Phase: non Markovian Coupling}
\label{secondphase}
To describe the coupling during the second phase we will use the full multi-scale space-time tessellation described in section \ref{tessellation}.
For simplicity, we translate time so that this phase starts at time $0$ and that $X_0$ is at the origin. 
Hence, $\bar X_0$ can be arbitrary, and $\eta_0^\star$ and $\bar \eta_0^\star$ can be any configuration for which the event 
$\cB_0$ from~\eqref{eq:eventb} holds.

\subsection{Largest scale}
We begin by creating the multi-scale space-time tessellation of $\mathbb{T}_n^d\times[0,\Delta_2]$ and with largest scale
\begin{equation}
\label{eq:kmax}
k_\rmax\defn\log_2\log{n}.
\end{equation}
We consider a positive constant $C_\newcte{cte:delta2}(p)>0$ to be chosen later so that $t_{k_\rmax}$ divides $C_{\ref*{cte:delta2}}\frac{n^2}{\mu}$, and define
\begin{equation}
   \Delta_3\defn \Delta_2+\frac{n^2}{\mu} 
   \quad\text{and}\quad
   \Delta_2\defn C_{\ref*{cte:delta2}}\frac{n^2}{\mu}.
   \label{eq:delta}
\end{equation}

The following Lemma shows that with large probability there are no bad boxes of scale $k_\rmax$ or larger. This will allow us to restrict our analysis to boxes of scale at most $k_\rmax$. 
We will consider all the boxes contained into the tessellation $\mathbb{T}_n^d\times[0,\Delta_3]$, 
which in particular are all the boxes intersecting the tessellation of $\mathbb{T}_n^d\times[0,\Delta_2]$.

\begin{lemma}
   \label{lemmakmax}
   For any $\delta>0$, there exists $p_0=p_0(\delta,d)>0$ such that for all $p<p_0$ and $n$ large enough
   \begin{equation*}
       \PR\lr{R_k(i,\tau)\text{ is bad}\text{ for some }R_k(i,\tau)\subset\mathbb{T}_n^d\times[0,\Delta_3],\text{ with }k\geq  k_\rmax}
       \leq \rho_1^{2^{k_\rmax-3}}.
   \end{equation*}
\end{lemma}
\begin{proof}
   The number $\zeta_{k}$ of boxes of scale $k$ in $\mathbb{T}_n^d\times[0,\Delta_3]$ is trivially bounded as
   \begin{align*}
       \zeta_{k}\leq &\left(\frac{n}{\ell_k}\right)^d\frac{\Delta_3}{t_k}
          \leq (C_{\ref*{cte:delta2}}+1)n^{d+2}.
   \end{align*}
   Using Lemma \ref{lemrho} the probability that there exists a box of scale $k_\rmax$ or bigger that is bad is bounded above by
   \begin{equation*}
       \sum_{k\geq  k_\rmax}\zeta_k\rho_k
       \leq (C_{\ref*{cte:delta2}}+1)n^{d+2} \sum_{k\geq  k_\rmax}\rho_1^{2^{k-2}}
       \leq 2 (C_{\ref*{cte:delta2}}+1)n^{d+2} \rho_1^{2^{k_\rmax-2}}.
   \end{equation*}
   Using the value of $k_\rmax$ and the fact that $\rho_1$ can be made arbitrarily small by taking $p$ small concludes the proof.
\end{proof}

\subsection{The coupling}
Recall the map $\Psi_t$ introduced in Section~\ref{firstphase} which maps $X_t$ into $\bar X_t$.
In order to define the coupling of the two processes, we will use a different map $\Phi_t$. The idea is that our new map will be equal to $\Psi_t$ in good parts of the environment, but when the walker enters the enlargement 
of a bad box, we will stop changing $\Phi_t$ and will keep it ``frozen'' until the walkers exit the enlargements of all bad boxes. The idea is that in the enlargement of bad boxes we want to couple the graphs
in a large region around the walkers so that if the walkers enter a bad box, then they do so with their graphs coupled within the box. We stop updating $\Phi_t$ because when $\Phi_t$ changes many edges uncouple.

More precisely, given a time $t$, denote with 
\begin{equation*}
\bar{s}_t=\sup\{s\leq  t\,:\, (X_s,s) \text{ is inside a $k_\rmax$-great box}\}
\end{equation*}
the last time before $t$ the walker is in a $k_\rmax$ great box. We will consider the new map $\Phi_t$ defined as
\begin{equation*}
    \Phi_t\defn\Psi_{\bar{s}_t}.
\end{equation*}

We will show that this change of map actually will not create any problems; in fact, we will show that $\Phi_t\equiv\Psi_t$ for all $t$ because in the way we construct the coupling, 
when the walkers are in the enlargement of a bad box, 
we will succeed in applying identity coupling, hence the translation map remains constant. So, the introduction of $\Phi_t$ here is a formalism so that the coupling procedure is well defined. This will imply that 
our application of identity coupling later on will be successful, which in turn implies that $\Phi_t\equiv \Psi_t$. 

As soon as the second phase begins we check whether the box $R_1(i,0)$, such that $(X_{0},0)\in R_1^\core(i,0)$, is $k_\rmax$-great (the reason we do this will be clarified later, see Remark \ref{whyenl}). 
If that is the case then we can begin the coupling procedure relative to the second phase.
The coupling is composed of two parts: the coupling of the graphs (that is, the coupling of $\eta_t^\star$ and $\bar{\eta}_t^\star$) and the coupling of the walkers.

\subsubsection{Coupling of the graphs}\label{sec:coupgraph}
We let the process $\{\eta_t^\star\}_{t\ge 0}$ evolve. 
Denote with $\mathcal{C}_v(t)$ (resp., $\bar{\mathcal{C}}_v(t)$) the cluster that contains vertex $v$ at time $t$ in the process $\eta^\star_t$ (resp., $\bar\eta^\star_t$). 
When an update $(s,U',U)$ occurs at an edge $e$ in $\eta_s^\star$ we update the process $\bar\eta^\star_s$ as follows.
\begin{itemize}
    \item If the update is a $\star$-update we refrain from looking at $U$ and instead simply set $\eta^\star_s(e)=\star$ and $\bar\eta^\star_s(\Phi_s(e))=\star$.
    \item If the update is not a $\star$-update we must check in both configurations $\eta^\star_s$ and $\bar\eta^\star_s$ whether $e$ is a cut-edge or not. We do this by looking at the connected components of the endpoints $v_1,v_2$ of the edge $e$. 
       If an edge $e'$ is such that $\eta^\star_s(e')=\star$ and $e'$ is incident to a vertex in $\mathcal{C}_{v_1}(s)\cup\mathcal{C}_{v_2}(s)$, 
          we sample its current status, open or closed, according to its last update. Note that this last update is itself a tuple $(\bar{s}, \bar{U}', \bar{U})$, so this step boils down to checking the value of $\bar{U}$. 
          If $\bar\eta^\star_s(\Phi_s(e'))=\star$ we set $\bar\eta^\star_s(\Phi_s(e'))=\eta^\star_s(e')$ as well. 
          We continue this procedure until the components of $v_1$ and $v_2$ have been fully explored in $\eta^\star_s$ and proceed analogously for the process $\bar\eta^\star_s$ until the components of $\Phi_s(v_1)$ 
          and $\Phi_s(v_2)$ have been fully explored. 
          A potential disagreement $\eta^\star_s(e)\ne\bar\eta^\star_s(\Phi_s(e))$ can happen only if, by revealing the components of $v_1$, $v_2$, $\Phi_s(v_1)$ and $\Phi_s(v_2)$, 
          we find that $e$ is a cut-edge in $\eta^\star_s$ but $\Phi_s(e)$ is not a cut-edge in $\bar\eta^\star_s$, or vice-versa.  
\end{itemize}
In this way edges whose status is $\star$ can always be coupled equivalently whereas non $\star$-updates cause the reveal of the status of other edges, potentially creating disagreements between the two configurations.

\begin{remark}[Momentaneous change of coupling]\label{rem:newcoupgraph}
   At some times we will carry out a different coupling of the environment. This will be done by simply introducing another map $\tilde \Phi$ of the environments, and the coupling of the graphs 
   will go as described above with $\Phi_t$ replaced with $\tilde \Phi$ until we specify that $\Phi_t$ is again the map to be used. 
\end{remark}

\subsubsection{Coupling of the walkers}\label{sec:coupwalker}
During this discussion the reader should refer to Figure~\ref{fig:coupling}.
\begin{figure}
    \centering
    \includegraphics[width=.9\linewidth]{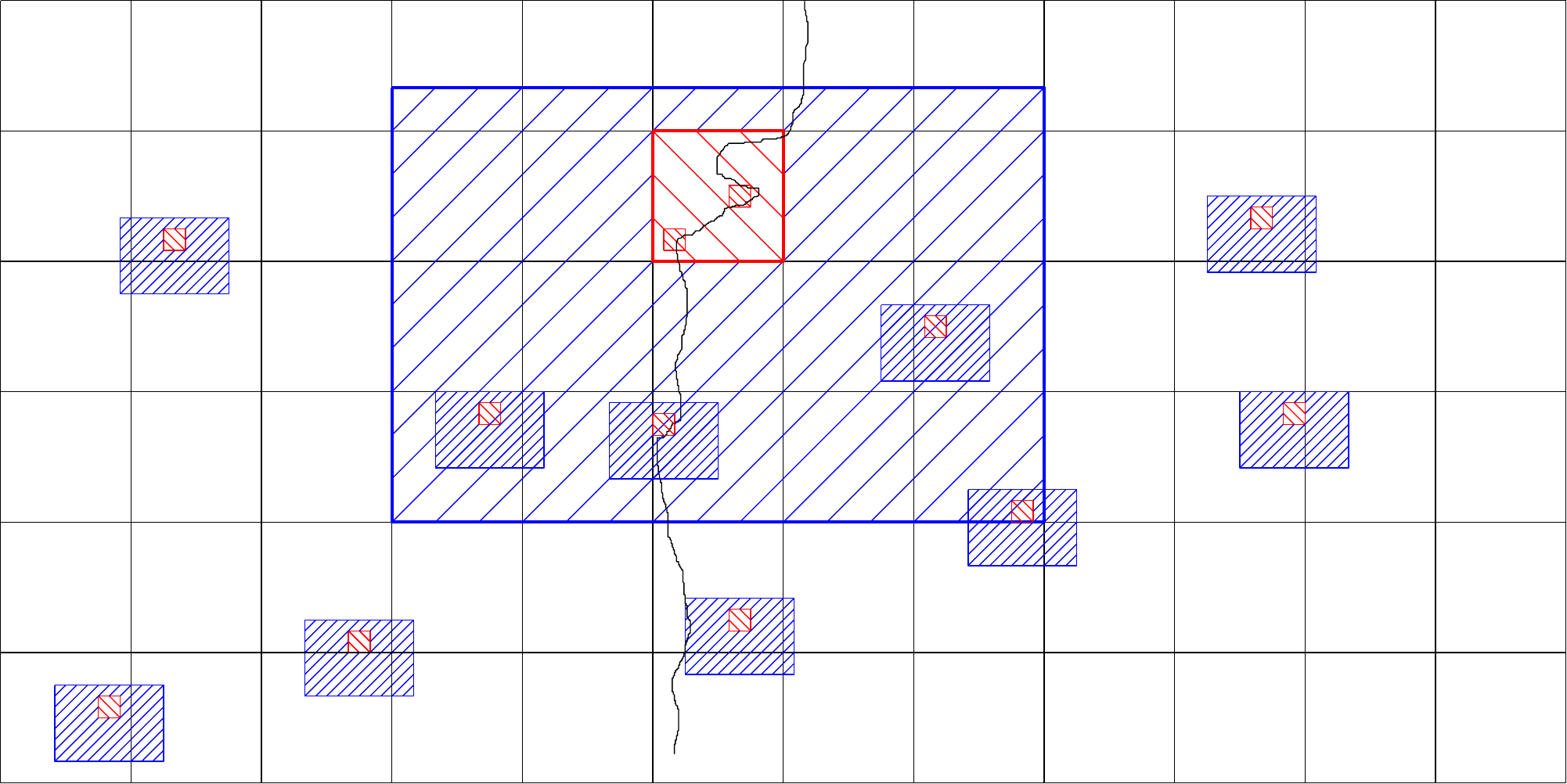}
\caption{In red the bad boxes, in blue their enlargement, in black the tessellation and the walker's trajectory. In bad boxes there is no control over the displacement of the walker, whereas in good boxes, the walker always leaves the box from its time boundary. Whenever the walker enters the enlargement of a bad box we start doing identity coupling, otherwise, in great boxes, if a SRWM occurs we do a simple random walk coupling, if not we keep doing identity coupling.}
\label{fig:coupling}
\end{figure}
Our goal is to define a coupling that can bring the walkers together. For this we will use the multi-scale tessellation. 
The coupling of the walkers will be composed of two different couplings. 
When the walker $X_s$ enters the core of a great box $R_1^\core(i,\tau)$, we will try to take advantage of the nice environment that a great box provides to perform a coupling that we refer to as a \textit{simple random walk moment}. 
This coupling aims to change the distance between the walkers, so that eventually the walkers may find themselves at the same site.

On the other hand, whenever $X_s$ is not in a great box, then we do not have a good enough control on the environment around the walker to do a simple random walk moment.
In such cases, we will just resort to a simple identity coupling that keeps the distance between the walkers unchanged. An identity coupling will only be able to be performed if the environment around the walkers are the same.
For this, we define the following event:
\begin{equation}
    \cB_t'\defn\left\{\forall e\in E(B_{\ell/2}^\infty(X_{t})),\,{\eta}_{t}^\star(e)=\bar{\eta}_{t}^\star(\Phi_{t}(e))\right\}.
    \label{eq:eventbp}
\end{equation}
If $B_s'$ holds for all $s\in(s_1,s_2)$, then in this time interval the walkers can perform the same jumps and not change their relative distance. 
In other words, we identity coupling is successful. 
In fact if the environment around the walkers is the same (as a matter of fact we only need the environments to agree on a ball of radius $1$ around the walkers), 
by doing identity coupling the walkers are able to perform the same jumps.

So the proof is now split into three steps. Since $\Phi_t$ does not change when the walker enters the $2$-enlargement of a bad box, we will show in Section~\ref{sec:phiunchanged} that when
$\Phi_t$ does not change the graph couples. Next, we deal with showing that identity coupling can be successfully implemented as the walker enters the $2$-environment of a bad box (i.e., when the walker is not in a great box).
This is carried out in Section~\ref{sec:identity}. Then in Section~\ref{sec:srwm} we deal with the simple random walk moments.

\subsection{Coupling of the graphs with $\Phi_t$ unchanged}\label{sec:phiunchanged}
Given $I\subset \Z^d$ and $k\geq 1$, let 
$$
   \cS_k(I)=\bigcup_{i\in I} S_k(i)
   \quad\text{and}\quad
   \cS_k^\diamond(I)=\bigcup_{i\in I} S_k^\diamond(i),
$$
for any $\diamond\in\lrc{\core,\enl,\enlb}$.
Recall the value $m$ in the definition of $\ell_k$ in~\eqref{ellkdeff}. 
Recall also $\bar t_k$ from~\eqref{eq:bt1}.
Then, for $k\geq 2$, we define 
\begin{equation}
   \bar t_k = \frac{6 t_k}{(k-1)^2m}.
   \label{eq:bartk}
\end{equation}

We start this section showing that the graph gets coupled in regions of good boxes if $\Phi_t$ does not change.

\begin{lemma}[Graphs couple in good boxes]\label{lem:graphcoup}
   Let $m$ be large enough, and then let $\ell$ be large enough with respect to $m$.
   Let $R_k(i,\tau)$ be a good box, and let $s_1$ be any time instance so that $[s_1,s_1+2\bar t_k]\subset T_k(\tau)$. 
   If $\Phi_t$ does not change during $[s_1,s_1+2\bar t_k]$, then  
   \begin{equation}
      \text{there exists $t\in [s_1,s_1+2\bar t_k]$ such that }\eta_{t}^\star(e)=\bar\eta_{t}^\star(\Phi_{t}(e)) \text{ for all $e\in S_k(i)$}.
      \label{eq:coupenlb}
   \end{equation}
\end{lemma}
\begin{proof}
   If $k=1$ then the proof follows since each edge of $S_k(i)$ receives only $\star$-updates and gets updated at least once during $[s_1,s_1+2\bar t_k]$.
   For $k\geq 2$, we assume that the statement of the lemma holds up to scale $k-1$.
   Let $s_2=\max T_k(\tau)$. 
   Let $\tau'$ be the first time index such that $\tau't_{k-1}\in [s_1,s_2]$ and all boxes $R_{k-1}(\cdot,\tau')\subset R_k(i,\tau)$ are good.
   Let $I$ be the set of indices containing all $(k-1)$-boxes that are inside $S_k(i)$; more precisely,
   $$
       I = \lrc{i' \colon S_{k-1}(i')\subset S_k(i)}.
   $$
   Then, by induction, by time $\tau't_{k-1}+2\bar t_{k-1}$ we obtain that $\cS_{k-1}(I)=S_k(i)$ has been coupled. 
   
   Now it remains to show that $\tau't_{k-1}+2\bar t_{k-1}\leq s_1+2 \bar t_k$. 
   Note that since $R_k(i,\tau)$ is a good box, there exists $\hat \imath, \hat\tau$ such that all $(k-1)$-bad boxes contained in $R_k(i,\tau)$ are contained in 
   $R_{k-1}^\enl(\hat\imath,\hat\tau)$. Since the amount of time spanned by the enlargement at scale $k-1$ is $9t_{k-1}$, 
   we obtain that $\tau't_{k-1}\leq s_1 + 9t_{k-1} +t_{k-1}$, where the last $t_{k-1}$ is to account for the possibility that $s_1$ is not a multiple of $t_{k-1}$. 
   Hence, using the notation $a_+ = \max\lrc{a, 1}$ for consistency with the case $k=2$, and noting that 
   $\bar t_1 \leq \frac{6t_1}{(k-2)^2_+ m}$ provided $\ell$ is made large enough once $m$ has been fixed, we have
   \begin{align*}
       \tau't_{k-1}+2\bar t_{k-1}
       &\leq s_1 + 10 t_{k-1}+2\bar t_{k-1}\\
       &\leq s_1 + t_{k-1}\lr{10+2\frac{6}{(k-2)^2_+m}}\\
       &= s_1 + \frac{t_{k}}{m(k-1)^2}\lr{10+\frac{12}{(k-2)^2_+m}}\\
       &= s_1 + \frac{\bar t_k}{6}\lr{10+\frac{12}{(k-2)^2_+m}}\\
       &\leq s_1 + 2\bar t_k.
   \end{align*}
\end{proof}

Recall the definition of $S_1^\inn(i)$ from~\eqref{eq:sinn}. For $k\geq2$, define
\begin{equation}
   S_k^\inn(i) = \bigcup_{j \colon S_{k-1}^\enlb(i)\subset S_k(j)}S^\core_{k-1}(j).
   \label{eq:sinnk}
\end{equation}
For a set of indices $I$, we write 
$$
   \cS_k^\inn(I) = \bigcup_{j \in I}S_{k}^\inn(j).
$$
Note that by taking $m$ large enough, then $\cS_k^\core(I)\subset \cS_k^\inn(I)\subset \cS_k(I)$.
We start with a simple result about the connected component of a vertex.
\begin{lemma}\label{lem:compsize}
   Let $m\geq 2$ and let $\ell$ be large enough with respect to $m$.
   Let $I\subset \Z^d$ be a set of indices, $k\geq 1$ a scale and $\tau\geq 1$ a time index such that $R_k(i,\tau)$ is a good box for all $i\in I$. 
   Then, for any $v\in \cS_k^\inn(I)$ and any $t\in T_k(\tau)$, the connected component of $v$ is contained in $B_{5\ell_k/m}^\infty(v)$,
   where we recall that $B^\infty_r(v)$ is the $L_\infty$ ball of radius $r$ around $v$.
\end{lemma}
\begin{proof}
   For $k=1$, the result follows by the fact that components have size at most $\log^2\ell$ in good $1$-boxes when $\tau\geq 1$, and $\ell$ is large enough so $\log^2\ell \leq 5\ell/m$. 
   For $k\geq 2$, let $(i',\tau')$ be such that $v\in S_{k-1}^\core(i')$ and $t\in T_{k-1}(\tau')\subset T_k(\tau)$; there could be more than one choice for $\tau'$, it is irrelevant which one we pick.
   Note that since $v\in \cS_k^\inn(I)$
   $$
      S_{k-1}(i')\subset S_{k-1}^\enlb(i')\subset S_k(i)\text{ for some $i\in I$}. 
   $$ 
   If $R_{k-1}(i',\tau')$ is good, then the connected component of $v$ is contained in $B_{5\ell_{k-1}/m}(v)\subset B_{5\ell_{k}/m}(v)$ 
   by applying the induction hypothesis at scale $k-1$ and set of indices $\lrc{i'}$. Otherwise, note that by Remark~\ref{enl1shield} we have that $R_{k-1}^\enl(i',\tau')$ contains all bad boxes in $R_k(i,\tau)$. 
   If the connected component of $v$ is contained in $S_{k-1}^\enl(i')$ then it is contained in $B_{5\ell_{k-1}}(v)\supset S_{k-1}^\enl(i')$. 
   Since $5\ell_{k-1}=5 \frac{\ell_k}{m(k-1)^2}\leq \frac{5\ell_k}{m}$ the lemma holds on this case as well.
   In the final case, when the connected component of $v$ is not contained in $S_{k-1}^\enl(i')$, it may sound contradictory but we can get an even smaller bound for the component of $v$.
   The reason is that there must exist 
   $i''$ such that $v$ is at the same component of a vertex $u$ with 
   $u\in S_{k-1}^\core(i'')$ and $S_{k-1}^\core(i'')\cap S_{k-1}^\enl(i')=\emptyset$ but $S_{k-1}(i'')\cap S_{k-1}^\enl(i')\neq\emptyset$. 
   But since $(u,t)$ is in the box $R_{k-1}(i'',\tau')$, and $S_{k-1}(i'')\subset S_{k-1}^\enlb(i')\subset S_k(i)$, we have that $R_{k-1}(i'',\tau'')$ is a good box. Thus, 
   by induction we obtain that the connected component of $v$ is inside $B_{5\ell_{k-1}/m}(u)\subset B_{10\ell_{k-1}/m}(v)$.
   Since $10\frac{\ell_{k-1}}{m} = 10 \frac{\ell_k}{m^2 (k-1)^2} \leq \frac{5\ell_k}{m}$ for all $k$ as long as $m\geq 2$, the proof is completed.
%
\end{proof}

With the help of the above lemma, we can show that the graph cannot uncouple in regions surrounded by good boxes. 
\begin{lemma}[Graphs remain coupled if $\Phi_t$ does not change]\label{lem:keepcoupled}
   Let $m$ be large, and let $\ell$ be large enough with respect to $m$.
   Let $R_k(i,\tau)$ be a good box, and let $s_1<s_2$ with $s_1,s_2\in T_k(\tau)$ and $s_1\leq (\tau+1) t_k$.
   If $\Phi_t$ does not change during $t\in [s_1,s_2]$ and
   $$
      \eta_{s_1}^\star(e)=\bar \eta_{s_1}^\star(\Phi_{s_1}(e)) \text{ for all $e\in E(S_k(i))$},
   $$
   then $\eta_{t}^\star(e)=\bar \eta_{t}^\star(\Phi_{t}(e))$ for all $e\in E(S_k^\inn(i))$ and all $t\in[s_1,s_2]$.
\end{lemma}
\begin{proof}
   For $k=1$ the lemma is obvious, since for any $e\in E(S_k(i))$, $e$ only receives $\star$-updates during $T_k(\tau)$. Therefore, $\eta_t^\star(e)=\bar\eta_t^\star(\Phi_t(e))$ for all $t\in[s_1,s_2]$. 
   For $k\geq 2$, assume the lemma holds up to scale $k-1$. 
   Let 
   $$
      \cT=\lrc{\tau' \colon T_{k-1}(\tau')\cap[s_1,s_2]\neq\emptyset \text{ and } T_{k-1}(\tau')\subset T_k(\tau)}.
   $$
   Let $\tau_1=\min \cT$. Note that either 
   \begin{equation}
      s_1\in T_{k-1}^\core(\tau_1) 
      \quad\text{or}\quad
      s_1\in T_{k-1}^\core(\tau_1-1),
      \label{eq:s1prop}
   \end{equation}
   where the latter happens when $s_1$ is near the starting time of $T_k(\tau)$. Because $s_1$ cannot be near the ending time
   of $T_k(\tau)$ due to the condition $s_1\leq (\tau+1) t_k$, we obtain that $\cT$ is not empty.
   We will first show that 
   \begin{align}
      \text{$S_{k-1}^\inn(j)$ is coupled during $[s_1,s_2]$ for all $j$ such that $S_{k-1}(j)\subset S_k(i)$}\nonumber\\
      \text{and for which $R_{k-1}(j,\tau')$ is good for all $\tau'\in \cT$}.
      \label{eq:sinncoup}
   \end{align}
   To see this, let $r_1=\sup T_{k-1}(\tau_1)$ and note that $r_1 \geq s_1 + t_{k-1}$ because of~\eqref{eq:s1prop}.
   Now, induction gives that $S_{k-1}^\inn(j)$ remains coupled up to time $r_1$. 
   We would like to reapply the induction hypothesis on the box $S_{k-1}(j)$ in the next time step, but for this we need $S_{k-1}(j)$ to be coupled, not only $S_{k-1}^\inn(j)$. 
   Thus, we first apply Lemma~\ref{lem:graphcoup} from time $r_1-2\bar t_{k-1}$ to obtain that there exists a time $r_1'\in [r_1-2\bar t_{k-1},r_1]$ 
   for which the whole of $S_{k-1}(j)$ is coupled. Let $\tau_2$ be such that $r_1'\in T_{k-1}^\core(\tau_2)$ and note that $\tau_2\geq \tau_1+1$. Thus, we repeat the induction hypothesis and the application of 
   Lemma~\ref{lem:graphcoup} to obtain a sequence of $\tau_\iota$, $r_\iota$ and $r_\iota'$ until a certain value $r_\iota'\in [s_2-2\bar t_{k-1},s_2]$. At that time, the induction hypothesis gives that 
   $S_{k-1}^\inn(j)$ is coupled at time $s_2$, establishing~\eqref{eq:sinncoup}. 
   
   Now we turn to establish the lemma. 
   If $R_k(i,\tau)$ has no bad $(k-1)$-box intersecting the time interval $[s_1,s_2]$, then~\eqref{eq:sinncoup} and the fact that $(k-1)$-boxes overlap give that 
   $\bigcup_{j\colon S_{k-1}(j)\subset S_k(i)}S_{k-1}^\inn(j)\supset S_k^\inn(i)$ is coupled during $[s_1,s_2]$.
    
   Now assume that $R_k(i,\tau)$ contains bad $(k-1)$-boxes that intersect $[s_1,s_2]$. From Remark~\ref{enl1shield}, there exists $R_{k-1}(i',\tau')$ so that all $(k-1)$-bad boxes contained in $R_k(i,\tau)$ are contained in 
   $R_{k-1}^\enl(i',\tau')$. Let 
   $$
      \cJ = \lrc{j \colon S_{k-1}(j)\subset S_k(i) \text{ and } S_{k-1}(j)\not \subset S_{k-1}^\enl(i')},
   $$
   and note that $R_{k-1}(j,\tau'')$ is good for all $j\in \cJ$ and $\tau''\in \cT$. Therefore,~\eqref{eq:sinncoup} gives that $S_{k-1}^\inn(j)$ is coupled during $[s_1,s_2]$ for all $j\in\cJ$. 
   The remaining of the proof is split into two cases. 
   First assume that $S_{k-1}^\enl(i')$ is separated from infinity by $\cJ$, which means that any path from $S_{k-1}^\enl(i')$ to the outside of $S_k(i)$ must enter $S_{k-1}^\core(j)$ for some $j\in\cJ$. 
   In fact, letting 
   $$
      \cJ' = \lrc{j\in\cJ \colon S_{k-1}^\core(j)\subset S_{k-1}^\enl(i')},
   $$
   we get that the path must enter $S_{k-1}^\core(j)$ for some $j\in \cJ'$.
   Besides, Lemma~\ref{lem:compsize} gives that for all $v\in \cS_{k-1}^\inn(\cJ)$ and all $s\in T_{k-1}^\enl(\tau')$ we have that the connected component of $v$ is contained in $B_{5\ell_{k-1}/m}^\infty(v)$. 
   Therefore, all connected components intersecting $S_{k-1}^\enl(i')$ must be contained in $\bigcup_{v \in S_{k-1}^\enl(i')} B_{5\ell_{k-1}/m}^\infty(v)$, which is a spatial region contained in the interior of 
   $\cS_k^\inn(J')$. 
   Therefore, since $\cS_{k-1}^\inn(J)\supset \cS_{k-1}^\inn(J')$ remains coupled throughout $[s_1,s_2]$ by~\eqref{eq:sinncoup}, non-$\star$ updates inside $S_{k-1}^\enl(i')$ cannot uncouple the graph.
   
   Turning to the second case, we assume that 
   $S_{k-1}^\enl(i')$ is not separated from infinity by $\cJ$. This means that $S_{k-1}^\enl(i')$ is so close to the boundary of $S_k(i)$ that it does not intersect $S_k^\inn(i)$. 
   More formally, for any $v\in S_{k-1}^\enl(i)$ we have that $B_{10\ell_{k-1}}^\infty(v)$ cannot be contained in $S_k(i)$. But this implies that any $i''$ with $S_{k-1}^\core(i'')\subset S_{k-1}^\enl(i')$ we have that 
   $S_{k-1}^\enlb(i'')\not\subset S_k^\inn(i)$. Therefore, applying~\eqref{eq:sinncoup} to the boxes in $\cJ$ already gives that $S_k^\inn(i)$ is coupled during $[s_1,s_2]$.
\end{proof}

\subsection{Identity coupling}\label{sec:identity}
We prove that, by doing identity coupling, as long as the particle $X_t$ is in a point $(v,t)\in \mathbb{T}_n^d\times\mathbb{R}^+$ in space-time that is part of a $1$-box $R_1(\cdot,\cdot)$ that is good, 
it is always possible to keep the distance between $X_t$ and $\bar X_t$ constant.
Recall the event $\cB'_t$ from~\eqref{eq:eventbp}, the event $\cB_t$ from~\eqref{eq:eventb}, and the definition of the spatial core of a box in~\eqref{eq:score}.
We will need a weaker version of $\cB'_t$ which we define as 
\begin{equation}
   \cB_t'' = \lrc{\forall e\in E\lr{B_{\ell/3}^\infty(X_t)},\;\eta_t^\star(e) = \bar\eta_t^\star(\Phi_t(e))}.
   \label{eq:cbpp}
\end{equation}
Recall that in the second phase we assume that $R_0(0,0)$ is a $k_\rmax$-great box and $\cB_0$ holds; we do not restate these conditions on the lemmas.
\begin{lemma}[Identity coupling succeeds in good boxes]
   \label{lemic1}
   Let $s_1$ be a time so that $\cB_{s_1}''$ holds and $(X_{s_1},s_1)\in R_1^\score(i,\tau)$ with $R_1(i,\tau)$ being a good box. 
   Let $s_2\in T_1(\tau)$, $s_2\geq  s_1$. If we attempt to do identity coupling for the entire time interval $[s_1,s_2]$, then 
   the coupling succeeds and $\Psi_{s_1}\equiv\Psi_{t}$ for all $t\in[s_1,s_2]$. 
\end{lemma}
\begin{proof}
   Let $B$ be the $L_\infty$ ball of radius $\ell/3$ around $X_{s_1}$; $B$ is a fixed region in space, not changing in time. The edges in $E(B)$ are coupled at time $s_1$ since $\cB''_{s_1}$ holds. 
   By Lemma \ref{timeboundlem} the walker never leaves $S_1(i)\subset B$ during the time interval $[s_1,s_2]$; if $s_1 \leq \bar t_1$, then we know that the component of the walker is at most $\log^2\ell$ since 
   $\cB_0$ holds and the box $R_1(0,0)$ is $k_\rmax$-great by the properties of the second phase. 
   Since there is no non-$\star$ update in $E(B)$ during $[s_1,s_2]$, $B$ remains coupled throughout and identity coupling is successful.
\end{proof}

The lemma below is a composition of the previous lemma when the walker traverses a sequence of good $1$-boxes. We assume that the stronger event $\cB_t'$ holds at the start time to be able to guarantee that $\cB_t''$ holds 
during the entire time interval covered by the lemma.
\begin{lemma}[Identity coupling succeeds in sequences of good boxes]
   \label{lemic1seq}
   Let $s_1$ be a time so that $\cB_{s_1}'$ holds. Let $s_2>s_1$ be such that during $[s_1,s_2]$ the walker only traverses $1$-boxes that are good. 
   Then, if we attempt to do identity coupling for the entire time interval $[s_1,s_2]$, the coupling succeeds, $\cB''_t$ holds and 
   $\Psi_{s_1}\equiv\Psi_{t}$ for all $t\in[s_1,s_2]$. Moreover, $\cB'_t$ holds for all $t\in[s_1+2\bar t_1,s_2]$.
\end{lemma}
\begin{proof}
   Let $R_1(i,\tau)$ be the box the walker is in its core at time $s_1$. Since $R_1(i,\tau)$ is a good box, Lemma~\ref{lemic1} gives that identity coupling works up to the end of $T_1(\tau)$ and 
   Lemma~\ref{lem:graphcoup} gives that $S_1(i)$ couples at some time during $[s_1,s_1+2\bar t_1]$. 
   Moreover, for any $t\in[s_1,s_1+2\bar t_1]$, $\cB_t''$ holds since $\cB_{s_1}'$ holds. 
   For $t\in T_1(\tau)\cap [s_1+2\bar t_1,\infty)$ we have that $\cB'_t\subset \cB''_t$ holds by Lemma~\ref{lem:keepcoupled}.  
   Hence, if $R_1(i',\tau')$ is the box whose core the walker is in when exitting $R_1(i,\tau)$, 
   we can apply Lemma~\ref{lemic1} again to show that identity coupling succeeds. Repeating this argument over and over again establishes the lemma.
\end{proof}

Now we analyze what happens in the neighborhood around a bad $1$-box, supposing that the walker 
enters the 2-enlargement of that box. 
Two things can happen, either the walker enters the 2-enlargement of the box from the space boundary $\partial_\rs R_1^\enlb(\cdot,\cdot)$ or 
it enters from the time boundary $\partial_\rt^-R_1^\enlb(\cdot,\cdot)$.
If it is from the space boundary, then the walker does not get too close to the bad box and $\mathcal{B}_t'$ would still be verified for all $t$. 
Moreover, as long as the walker is in the 2-enlargement, identity coupling can be applied successfully. 
In the other case, if the walker enters from the time boundary, then it could eventually reach the bad box but the environment in the 2-enlargement of the bad box will be coupled before that. 
In particular, the environments will be coupled at all times in $T_1^\enl(\cdot,\cdot)$ thanks to the abundance of $\star$-updates in 
$T_1^\enlb(\cdot,\cdot)\setminus T_1^\enl(\cdot,\cdot)$. 
This reasoning gives that the relative distance between the walkers does not change and the graphs remain coupled in $E(S_1^\enlb(i))$ when the walker cross a bad box of scale $1$.

In the lemma below we will require $m$ to be large enough so that the following holds: 
\begin{align}
   \text{For any $(k,i,\tau)$ and any $(i',\tau')$ so that $R_{k+1}^\core(i',\tau')$ intersects $R_{k}(i,\tau)$ we obtain that}\nonumber\\  
   \text{$R_{k+1}(i',\tau')$ contains all $k$-boxes that intersects $R_{k}^\enlb(i,\tau)$, and $S_{k+1}^\inn(i')$ contains $S_{k}^\enlb(i)$.}
   \label{eq:largem}
\end{align}
\begin{lemma}[Identity coupling in enlargement of bad boxes]
   \label{distnotchange}
   Let $m$ be large enough so that~\eqref{eq:largem} holds.
   Let $R_k(i,\tau)$ be a bad box of scale $k$ such that $R_{k+1}(i',\tau')$ is good for some $(k+1)$-box for which $R_{k+1}^\core(i',\tau')\cap R_k(i,\tau)\neq \emptyset$.
   Denote with $\tau_k^+=\max T_k^\enlb(\tau)$, and with $\tau_k^-=\min T_k^\enlb(\tau)$. 
   Let $s_c$ be a time at which the walker enters $R_k^\enlb(i,\tau)$ so $X_{s_c}\in S_k^\enlb(i)$ but $X_{s_c-}\not\in S_k^\enlb(i)$ or $s_c=\tau_k^-$.
   Let $s_e\defn\inf\{t \in (s_c,\tau_k^+] \colon X_t\notin S_k^\enlb(i)\}$ 
   the first time the walker exits $R_k^\enlb(i,\tau)$ after $s_c$; we take the convention that 
   $s_e=\tau_k^+$ if $X_t\in S_k^\enlb(i)$ for all $t\in [s_c,\tau_k^+]$. Thus
   \begin{equation}
   \label{eq:psinotchange}
   \text{if $\mathcal{B}_{s_c}'$ holds, then for all $t\in[s_c,s_e]$ $\Psi_t$ remains unchanged and $\cB''_t$ holds};
   \end{equation}
   consequently, identity coupling succeeds during $[s_c,s_e]$.
   Moreover, 
   \begin{equation}
   \label{eq:walkerenters}
   \text{if $s_c>\tau_k^-$, the walker does not enter $R_k^\enl(i,\tau)$}.
   \end{equation}
   Ultimately, letting $\cJ=\lrc{j\colon S_k(j)\cap S_k^\enlb(i)\neq \emptyset}$, 
   \begin{align}
      \text{if $s_c=\tau_k^-$, then }\eta_t^\star(e)=\bar\eta_t^\star(\Phi_t(e)) \text{ for all $e\in E(\cS_k^\inn(J))$}\nonumber\\
      \text{and all $t\in [s_c,s_e]$ with $t\geq \min T_k^\enl(\tau)$}.
   \label{eq:giveaname}
   \end{align}
\end{lemma}
The proof uses induction on $k$, so we treat the case $k=1$ separately.
\begin{proof}[Proof of Lemma~\ref{distnotchange} for $k=1$]
   We start with the case $s_c>\tau_1^-$, meaning that the walker entered the $2$-enlargement of the bad box from $\partial_\rs R_1^\enlb(i,\tau)$. 
   We need to establish~\eqref{eq:psinotchange} and~\eqref{eq:walkerenters} in this case.
   We establish~\eqref{eq:walkerenters} by showing that the walker never gets closer than $12\ell$ from $S_1^\enl(i)$. 
   To see this, from~\eqref{eq:largem} we have that $R_2(i',\tau')$ contains the $2$-enlargement of $R_1(i,\tau)$, and $R_2(i',\tau')$ is a good box.
   Moreover, Remark~\ref{enl1shield} gives that the $1$-enlargement of $R_1(i,\tau)$ contains all bad $1$-boxes inside $R_2(i',\tau')$, and Lemma~\ref{spaceenl} gives that 
   the distance between the walker and $S_1^\enl(i)$ is at least $12\ell$, establishing~\eqref{eq:walkerenters}. 
   To establish~\eqref{eq:psinotchange}, note that the walker only traverses good boxes during $[s_c,s_e]$, so~\eqref{eq:psinotchange} follows from~\ref{lemic1seq}.
   
   Now we consider the case $s_c=\tau_1^-$, and need to establish~\eqref{eq:psinotchange} and~\eqref{eq:giveaname}. 
   The idea in this case is to use the time interval between $s_c$ and $\min T_1^\enl(\tau)$, which is large enough for the graphs to couple. 
   In fact, applying Lemma~\ref{lem:graphcoup} to the box $R_2(i',\tau')$ from time $s_c$, we obtain a time $s\in[s_c,s_c+2\bar t_2]$ so that $S_2(i')$ is coupled. From this time onwards Lemma~\ref{lem:keepcoupled} 
   gives that $S_2^\inn(i')\supset S_1^\enlb(i)$ remains coupled up to time $s_e$. 
   From Lemma~\ref{timeboundlem} we know that the walker does not leave $S_2^\inn(i')$ during $[s_c,s_e]$. So if identity coupling succeeds up to time 
   $s_c+2\bar t_2$, then it succeeds up to time $s_e$. Moreover, note that $2\bar t_2=12 t_2/m=12 t_1$ is smaller than the distance between $s_c$ and $\min T_1^\enl(\tau)$, which is $15t_1$. 
   So $S_2^\inn(i')$ couples before the walker can enter $R_1^\enl(i,\tau)$ and~\eqref{eq:giveaname} is established.
   
   It remains to show that the coupling succeeds and $\cB_t''$ holds for all $t\in[s_c,s_c+2\bar t_2]$, completing the proof of~\eqref{eq:psinotchange}. 
   For this, we only need to note that during this time interval the walker only traverses good $1$-boxes, so~\eqref{eq:psinotchange} follows from Lemma~\ref{lemic1seq}.
\end{proof}

\begin{proof}[Proof of Lemma~\ref{distnotchange} for $k\geq2$]
   We have already established the case $k=1$. Now we proceed via induction. 
   Assume all claims of the lemma are proved up to scale $k-1$. Let $R_{k}(i,\tau)$ be a bad box and $R_{k+1}(i',\tau')$ as in the statement of the lemma be a good box. 
   All bad boxes in $R_{k+1}(i',\tau')$ are contained in $R_k^\enl(i,\tau)$. 
   
   We first prove the case $s_c>\tau_k^-$, which requires establishing~\eqref{eq:psinotchange} and~\eqref{eq:walkerenters}. 
   In this case we use the same argument as in the case $k=1$; that is,~\eqref{eq:walkerenters} follows from Lemma~\ref{spaceenl}.
   To show that identity coupling can be performed and $\cB_t''$ holds, notice that if at time $s_c$ the walker is inside a bad box $R_{k''}(i'',\tau'')$ for some $k''<k$, then since $R_1(0,0)$ is $k_\rmax$-great, 
   we have that in a previous time the walker was in the boundary of 
   $R_{k''}^\enlb(i'',\tau'')$. 
   If there are more than one tuple $(k'',i'',\tau'')$ satisfying the property above, we take the one with the largest $k''$ (breaking ties arbitrarily if there still are more than one such tuples).
   Since the walker must have entered the $2$-enlargement $R_{k''}^{\enlb}(i'',\tau'')$ at some time $s_c''$, 
   we obtain by induction that while traversing the bad box $R_{k''}(i'',\tau'')$ identity coupling is successful and $\cB'_t$ holds up to the end of $T_{k''}(\tau'')$, since~\eqref{eq:giveaname} implies $\cB'_t$. Therefore, 
   when the walker leaves $R_{k''}(i'',\tau'')$, we can apply the induction hypothesis again if the walker is inside another bad box.  
   It remains to check that identity coupling can be performed while the walker passes through space-time locations that belong to good boxes at all scale, in particular, while the walker passes through good $1$-boxes.
   But since $\cB'_t$ holds at that time, identity coupling succeeds by Lemma~\ref{lemic1seq}, concluding the proof of~\eqref{eq:psinotchange}.
   
   We now prove the case $s_c=\tau_k^-$, which requires establishing~\eqref{eq:psinotchange} and~\eqref{eq:giveaname}. Assume that $s_e\geq \min T_k^\enl(\tau)$, otherwise~\eqref{eq:psinotchange} follows from the same
   argument above and~\eqref{eq:giveaname} is irrelevant.
   We can do the same argument as for $k=1$; i.e., we show that the time interval between $s_c$ and $\min T_k^\enl(\tau)$ is large enough for the graphs to couple. 
   By Lemma~\ref{lem:graphcoup} we obtain a time $s\in[s_c,s_c+2\bar t_{k+1}]$ so that $S_{k+1}(i')$ is coupled and, by Lemma~\ref{lem:keepcoupled}, 
   $S_{k+1}^\inn(i')\supset S_k^\enlb(i)$ remains coupled until $s_e$. 
   Since Lemma~\ref{timeboundlem} gives that the walker does not leave $S_{k+1}^\inn(i')$ during $[s_c,s_e]$, if identity coupling succeeds up to time 
   $s_c+2\bar t_{k+1}$, then it succeeds up to time $s_e$. Besides, $2\bar t_{k+1}=12 \frac{t_{k+1}}{k^2 m}= 12t_k$ is smaller than the distance between $s_c$ and $\min T_k^\enl(\tau)$, which is $15t_k$. 
   So $S_{k+1}^\inn(i')$ couples before the walker can enter $R_k^\enl(i,\tau)$ and~\eqref{eq:giveaname} is established.
   To establish~\eqref{eq:psinotchange}, we need to show that $\cB_t''$ holds for all $t\in[s_c,s_c+2\bar t_{k+1}]$, but during this time the walker only traverses good $1$-boxes, 
   so~\eqref{eq:psinotchange} follows from Lemma~\ref{lemic1seq}.
\end{proof}

\begin{remark}
\label{whyenl}
The 2-enlargement of a bad box is chosen so that whenever the walker crosses it, by doing identity coupling the two processes have time to couple the environment before the walker crosses the bad box. 
For this exact reason we want the first box whose core the walker is at, at the beginning of the second phase, to be $k_\rmax$-great, so we know that the walker does not start inside the enlargement of a bad box, 
meaning that if the walker encountersa bad box during the second phase, it must first traverse its enlargement.
\end{remark}

%
\subsection{Simple random walk moment}\label{sec:srwm}

Now we handle the case when the walker traverses great boxes, during which we do not perform identity coupling but try a different coupling. 
This coupling will be based on what we call \emph{a simple random walk moment} (SRWM), which is a given condition of the evolution of the environment that makes the walker performs a simple random walk step.

\begin{definition}[Simple random walk moment]
\label{srwm}
Let $R_1(i,\tau)$ be a great box such that $(X_{\tau t_1},\tau t_1)\in R_1^\core(i,\tau)$. We consider three consecutive intervals $I_1,I_2,I_3$ of lengths
$$
   |I_1| = \frac{t_1}{2}
   \quad\text{and}\quad
   |I_2| = |I_3|= \frac{1}{\mu},
$$
such that $I_1$ begins at time $\tau t_1=\min T_1^\core(\tau)$; note that $\tau t_1+\sum_{j=1}^3|I_j|<(\tau+2)t_1=\max T_1^\core(\tau)$. 
Let $v\in S_1(i)$ be the position of the walker $X_{\tau t_1}$; note that since $R_1(i,\tau)$ is a good box then all edges adjacent to $v$ at time $\tau t_1$ are closed. 
All the events below consider only $\star$-updates during $I_1\cup I_2\cup I_3$, ignoring all non-$\star$ updates. 
Then, a \emph{simple random walk moment} (SRWM) is said to occur in $R_1(i,\tau)$ if the following events happen consecutively:
\begin{enumerate}[start=1,label={(\bfseries $E_\arabic*$)}]
    \item During $I_1$, one of the edges adjacent to $v$, say $e=(v,u)$, receives an update to become open, and the edges adjacent to $u$ with status $\star$ are sampled closed. 
       Moreover, the other edges adjacent to $v$ or $u$ do not open during $I_1$, 
        and after $e$ opens, $e$ does not close for at least time $\frac{C_\newcte{cte:i1}}{\mu}$. 
    \item During $I_2$, edge $e$ closes and does not open, while the edges adjacent to $e$, that were closed, do not open; note that at the end of $I_2$, the walker is in either $u$ or $v$.
    \item During $I_3$, the edges adjacent to $u$ or $v$ do a $\star$-update, and the edges adjacent to the walker do not open.
\end{enumerate}
\end{definition}
See Figure~\ref{fig:srwm} for an illustrative realization of a simple random walk moment. 
Define 
\begin{equation}
   \SRWM{i}{\tau} \text{ be the indicator for the event that SRWM occurs in $R_1(i,\tau)$.}
   \label{eq:srwm}
\end{equation}

\begin{remark}
\label{rem:srwm}
   Given $v$, the position of the walker at time $\tau t_1$, the event SRWM depends only on the updates in $E(S_1(i))$ during the time interval $I_1\cup I_2 \cup I_3$.
   In particular, it does not depend on the jumps of the walkers during $I_1\cup I_2\cup I_3$, and does not depend on non-$\star$ updates that could occur during $I_1\cup I_2\cup I_3$.
\end{remark}

\begin{figure}
    \centering
    \includegraphics[width=.9\linewidth]{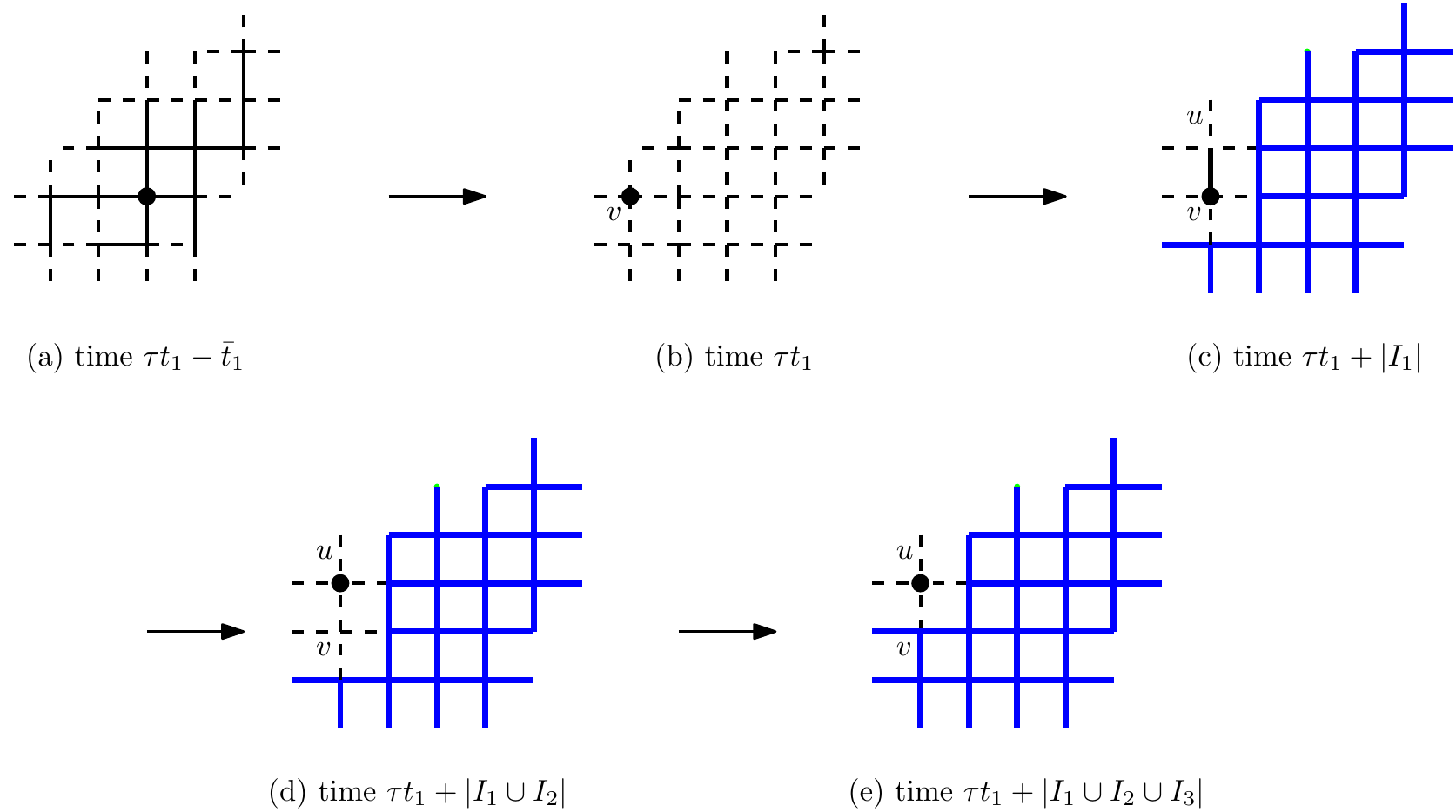}
   \caption{A possible realization of SRWM in a $k_{\rmax}$-great box. 
      (a) configuration at time $\tau t_1-\bar t_1$, with dashed lines representing closed edges, solid lines representing open edges, and the black ball representing the walker. 
      (b) During $[\tau t_1 -\bar t_1,\tau t_1]$ all edges close, trapping the walker in a vertex $v$. 
      (c) During $I_1$, edge $(u,v)$ adjacent to the walker opens, edges adjacent to $u$ are closed, and the other edges receive a $\star$-update. Blue lines represent edges that are updated $\star$. 
      (d) Edge $(u,v)$ closes at some time during $I_2$, trapping the walker in one of its endpoints, in this case endpoint $u$. 
      (e) During $I_3$ all edges adjacent to $v$ receive a $\star$-update, while the edges adjacent to the walker do not open.}
   \label{fig:srwm}
\end{figure}

Note that from $\tau t_1$ to time $\tau t_1 + |I_1\cup I_2 \cup I_3|$ the walker essentially performed a simple random walk step since the edge $e$ adjacent to $v$ that is chosen to open during $I_1$ is a uniformly random edge.
Now assume that the walker enters $R_1^\core(i,\tau)$ with $R_1(i,\tau)$ being a $k_\rmax$-great box; i.e., $(X_{\tau t_1},\tau t_1)\in R_1^\core(i,\tau)$. 
We define the coupling we employ in this situation.

\begin{definition}[Coupling on great boxes]
   At time $\tau t_1$ both walkers are trapped at some vertices $v=X_{\tau t_1}\in S_1^\core(i)$ and $\bar v=\Phi_{\tau t_1}(v)$. 
   Then we perform the following steps.
   \begin{enumerate}
      \item Sample whether a simple random walk moment occurs in $R_1(i,\tau)$. If not, sample the updates of the edges in $S_1(i)$ during $I_1\cup I_2 \cup I_3$ from the distribution conditioned on $\SRWM{i}{\tau}=0$, apply the 
         coupling of the graphs from Section~\ref{sec:coupgraph} and apply identity coupling for the walkers. Identity coupling succeeds since the graphs are coupled inside $S_1(i)$ and we obtain that 
         $\Phi_t$ does not change during $I_1\cup I_2 \cup I_3$. This concludes the coupling when $\SRWM{i}{\tau}=0$.
      \item If $\SRWM{i}{\tau}=1$, choose a coordinate $j\in\lrc{1,2,\ldots,d}$ and a sign $s\in\lrc{-1,+1}$ uniformly at random. 
         If $v$ and $\bar v$ agree in that coordinate, let $e=(v,v+se_j)$ and $\bar e=(\bar v, \bar v+se_j)$ 
         be the edges chosen to open during $I_1$ in the configurations $\eta^\star$ and $\bar \eta^\star$, respectively, where $e_1,e_2,\ldots,e_d$ stands for the standard basis of $\Z^d$. 
         In this case, during $I_2$, we let the walkers perform the same jumps across $e$ and $\bar e$ (i.e., we perform identity coupling), 
            and note that $\Phi_t$ maps $e$ into $\bar e$ during this time. Then we couple the graphs using $\Phi_t$, as described 
            in Section~\ref{sec:coupgraph}, until the end of $I_3$. In this case, the map $\Phi_t$ does not change during $I_1\cup I_2 \cup I_3$.           
      \item If $\SRWM{i}{\tau}=1$, and $v$ and $\bar v$ do not agree in the $j$th coordinate, we set $e=(v,v+se_j)$ and $\bar e=(\bar v, \bar v-se_j)$. This is the most delicate case as we will need to change the coupling of the
       graphs from the time $e$ opens to the end of $I_1\cup I_2 \cup I_3$. For this, we will use the map $\tilde \Phi$ which maps $v$ to $\bar v$ and is a translation map in all coordinates but the $j$th one, 
       where it is a reflection  map around $e$. In particular, $\tilde \Phi$ maps $e$ onto $\bar e$. Then the graphs will be coupled as in Remark~\ref{rem:newcoupgraph}; 
       that is, the graphs are coupled as in Section~\ref{sec:coupgraph} but using the 
       map $\tilde \Phi$ instead of $\Phi_t$. Note that any update to $e$ translates to an update of $\bar e$, so they open at the same time and close at the same time. 
       Let $\zeta$ be the time that $e$ and $\bar e$ open for the first time during $I_1$. Then, they remain open during $[\zeta,\zeta+C_{\ref*{cte:i1}}/\mu]$ since $\SRWM{i}{\tau}=1$. 
       We couple the position of the walkers at time $\zeta+C_{\ref*{cte:i1}}/\mu$ as follows. Let $\delta$ be the probability that $X_{\zeta+C_{\ref*{cte:i1}}/\mu}=v$ and $1-\delta$ be the probability that 
       $X_{\zeta+C_{\ref*{cte:i1}}/\mu}=u$. Then we make $X_{\zeta+C_{\ref*{cte:i1}}/\mu}=\bar X_{\zeta+C_{\ref*{cte:i1}}/\mu}$ with probability $\min\lrc{\delta,1-\delta}$; otherwise, we sample them accordingly.  
       Then, we let the graph and the walkers evolve up to the end of the interval $I_1\cup I_2\cup I_3$, coupling the jumps of the walkers so that they jump at the same times after time 
       $\zeta+C_{\ref*{cte:i1}}/\mu$; note that the walkers do not move after $e$ and $\bar e$ close for the first time after $\zeta+C_{\ref*{cte:i1}}/\mu$.       
   \end{enumerate}
\end{definition}

Now, let $s=\tau t_1 + |I_1\cup I_2\cup I_3|$ be the end time of the simple random walk moment. Note that if SRWM occurs then 
$\|X_{s}-\bar X_{s}\|_1$ may differ from $\|X_{\tau t_1}^1-X_{\tau t_1}^2\|_1$, and as a result the translation map $\Phi_s$ may be different from $\Phi_{\tau t_1}$ as well. So 
it could be the case that an edge that was coupled before the simple random walk moment (in the sense that ${\eta}_{\tau t_1}(e')=\bar{\eta}_{\tau t_1}(\Phi_{\tau t_1}(e'))$) may get uncoupled because the map $\Phi$ changes. 
On the other hand, after $I_1$ all the edges in the box receive a $\star$ update. So at the end of the SRWM, all edges in $S_1(i)$ are $\star$ with the only exception being the edges adjacent to the walker which are closed.  
So the configurations are coupled locally, in particular, $\cB_s'$ holds. 
Moreover, as $R_1(i,\tau)$ is great (so it is also good) the particles will stay in $S_1(i)$ for the whole time interval $T_1(\tau)$. 
In other words we obtain that the edges in $S_1(i)$, where the random walk moment is occurring, are coupled after the simple random walk moment ends.

More formally, we will implement this by assigning a ``hidden'' random variable to each $1$-box, which tells whether the box will undergo a SRWM should the walker pass there. We will not try the above coupling at each 
great box the walker enters, since we do need a bit of time separation between two simple random walk moments because of the overlapping of the boxes. But whenever we decide to attempt a simple random walk moment inside a 
great box the walker is in, the hidden random variable will tell whether SRWM will occurs. 
The main point is that we can obtain a lower bound on $\PR\lr{\SRWM{i}{\tau}=1}$ that is uniform on the location of the walker at time $\tau t_1$. 
Because of this uniform bound, we can couple the outcome of the hidden variable with the evolution of the processes ${M}_t^\star$ and $\bar M_t^\star$ 
so that the simple random walk moment takes place, regardless of the location of the walker within the box. 
The content of the hidden variable is just a Bernoulli random variable of parameter $C_{\ref*{cte:srwm}} p^{\frac{6d-1}{6d}}$, which is the bound we derive in Lemma \ref{srwmprob} below, so  
the event of successfully performing a SRWM stochastically dominates the hidden variable. 
Whenever we decide to look at the hidden variable of a box, we perform the coupling described above. Otherwise, we just do identity coupling.

Before establishing a bound on $\PR\lr{\SRWM{i}{\tau}=1}$ we need to show that the environments recouple locally after a SRWM.
\begin{lemma}[Recoupling the graphs after SRWM]\label{lem:recoupgraph}
   Let $R_1(i,\tau)$ be a $k_\rmax$-great box such that $X_{\tau t_1}\in S_1^\core(i)$ and $\cB''_{\tau t_1}$ holds. 
   Suppose the walkers perform successfully a simple random walk moment. 
   Then 
   $$
      \eta_{t}^\star(e)=\bar\eta_{t}^\star(\Phi_{t}(e)) \text{ for all $e\in E(S_1^\enlb(i))$ and all $t\in [s,(\tau+1)t_1],$}
   $$
   where $s=\tau t_1+|I_1\cup I_2\cup I_3|+\bar t_1$.
\end{lemma}
\begin{proof}
   $R_1(i,\tau)$ is $k_\rmax$-great, and in particular $1$-great. 
   Thus, every $1$-box $R(j,\tau)$ such that $S_1(j)\cap S_1^\enlb(i)\neq\emptyset$ is good. 
   After $s-\bar t_1=\tau t_1 + |I_1\cup I_2\cup I_3|$, we have that the edges in $S_1(i)$ are coupled and we start performing identity coupling of the walkers. 
   The coupling is succeessful so $\Phi_t$ does not change from that moment onwards and all edges in $S_1(j)$ with $S_1(j)\cap S_1^\enlb(i)\neq\emptyset$ receives a $\star$-update and couples.
\end{proof}

Now we bound the probability of a SRWM. 
Recall from Definition~\ref{def:goodbox} that the event that a box $R_1(i,\tau)$ is good is based on events $\cG_{12}(i,\tau)$ and $\hat\cG_{34}(i,\tau)$. 
Define $\cJ$ to be the set of all tuples $(i,\tau)$ such that $R_1(i,\tau)$ is a box of the tessellation of the second phase.
Let $\Sigma\defn\{0,1\}^{2\cJ}$ be the set of all possible assignments of occurrence or non occurrence to the events $\cG_{12}(i,\tau)$ and $\hat\cG_{34}(i,\tau)$. 
Then for each $\sigma\in\Sigma$ and each $(i,\tau)$, the values $\sigma_{12}(i,\tau)$ and $\sigma_{34}(i,\tau)$ will be used to specify whether the events $\cG_{12}(i,\tau)$ and $\hat\cG_{34}(i,\tau)$ occur, respectively.  
In this way, given $\sigma\in\Sigma$, we abuse notation and denote by $\sigma$ the event that the realizations of $\cG_{12}(i,\tau)$ and $\hat\cG_{34}(i,\tau)$ match the values of $\sigma_{12}(i,\tau)$ and
$\sigma_{34}(i,\tau)$ for each $(i,\tau)\in\cJ$, 
and write $\PR(\cdot \mid \sigma)$ for the corresponding conditional probability. More precisely, 
$$
   \PR\lr{\cdot \mid \sigma} = \PR\lr{\cdot \Mid \bigcap_{(i,\tau)\in\cJ} \lrc{\sigma_{12}(i,\tau)=\ind{\cG_{12}(i,\tau)}}\cap\lrc{\sigma_{34}(i,\tau)=\ind{\hat\cG_{34}(i,\tau)}}}.
$$

Note that once we condition on some $\sigma\in \Sigma$, then which boxes of all scales are good or bad is a deterministic function of $\sigma$. 
Let $\cF_{t}$ be the $\sigma$-algebra generated by the trajectory of the walker $X_s$ and the value of the map $\Psi_s$, $s\in[0,t]$, and all the updates of the graph up to time $t$.
Let $\Sigma_{i,\tau}\subset \Sigma$ be the set of assignments $\sigma$ for which $R_1(i,\tau)$ is a $k_\rmax$-great box.
\begin{lemma}\label{srwmprob}
   Let $(i,\tau)$ be such that $R(i,\tau)$ is a $k_\rmax$-great box. 
   There exists $p_0>0$ and $C_\newcte{cte:srwm}>0$ such that for all $p<p_0$, for all $\sigma \in\Sigma_{i,\tau}$, and all $F\in \cF_{\tau t_1}$ for which $\PR\lr{\sigma\cap F}>0$, 
   then the probability of performing a simple random walk moment in $R_1(i,\,\tau)$ is
   \begin{equation}
       \PR\lr{\SRWM{i}{\tau}=1 \mid F\cap \sigma}\geq  C_{\ref*{cte:srwm}} p^{\frac{6d-1}{6d}}.
   \end{equation}
\end{lemma}
\begin{proof}
   Start with the following simplification of $\sigma$. Recall the definition of $j(\tau)$ from~\eqref{eq:jtau}. 
   So $j(\tau)$ and $j(\tau+1)$ are the first and last interval of the type $\bar{T}_1(\cdot)$ inside $T_1(\tau)$. 
   Recall that $\sigma_{34}(\cdot,\cdot)$ correspond to the events $\hat\cG_{34}(\cdot,\cdot)$, which are i.i.d.\ events coupled with the events $\cG_{34}(\cdot,\cdot)$. 
   Since $\cG_{34}(\cdot,\cdot)$ are independent of $\cG_{12}(\cdot,\cdot)$ by Lemma~\ref{lem:badboxI}, we have that also $\hat\cG_{34}(\cdot,\cdot)$ are independent of $\cG_{12}(\cdot,\cdot)$. 
   Moreover, for any $x$, we have that $\cG_{34}(i',\tau)$ is independent of $\cF_{\tau t_1}$ since $\cG_{34}(i',\tau)$ only considers updates on the edges during the interval $T_1^\core(\tau)\setminus T_1(\tau+1)$. 
   Since for any fixed $x$ we have that $\cG_{34}(i',\tau')$ are independent for different $s$, we have that $\hat\cG_{34}(i',\tau)$ is independent of $\cF_{\tau t_1}$. 
   So now we collect in $\cJ_{34}$ all tuples from $\cJ$ for which $\SRWM{i}{\tau}$ depend on $\sigma_{34}(\cdot,\cdot)$:
   $$
      \cJ_{34} = \lrc{(i',\tau) \colon S_1(i')\cap S_1(i)\neq \emptyset}
      \quad\text{and}\quad
      \cJ' = \cJ \setminus \cJ_{34}.
   $$
   We will not need to split $\sigma_{12}(\cdot,\cdot)$ into two groups since those events are already independent of $\SRWM{i}{\tau}$. 
   
   For any $\sigma\in\Sigma$ denote
   \begin{align*}
      \cS &= \cS(\sigma) = \bigcap_{(i',\tau')\in \cJ} \sigma_{12}(i',\tau')\bigcap_{(i',\tau')\in \cJ'} \sigma_{34}(i',\tau')\\
      \cS_{34} &= \cS_{34}(\sigma) = \bigcap_{(i',\tau')\in \cJ_{34}} \sigma_{34}(i',\tau').
   \end{align*}
   Then, 
   \begin{align*}
       \PR\lr{\SRWM{i}{\tau}=1 \mid F\cap \sigma}
       &= \PR\lr{\SRWM{i}{\tau}=1 \Mid F\cap \cS \cap \cS_{34}} \\
       &\geq \frac{\PR\lr{\SRWM{i}{\tau}=1 \Mid F\cap \cS}-\PR\lr{\cS_{34}^\compl \Mid F\cap \cS}}
                  {\PR\lr{\cS_{34} \Mid F\cap \cS}}\\
       &\geq \PR\lr{\SRWM{i}{\tau}=1 \Mid F\cap \cS}-\PR\lr{\cS_{34}^\compl \Mid F\cap \cS}.
   \end{align*}
   Note that $T_1^\core(\tau)\setminus T_1(\tau+1)\supset I_1\cup I_2 \cup I_3$, so $\SRWM{i}{\tau}$ does not depend on $F\cap \cS$ given the position of the walker at time $\tau t_1$. 
   Letting 
   $S_1'(i)=\bigcup_{u\in S_1^\core(i)} B_{\log^2\ell}^\infty(u)$, which are the places where the walker can be at time $\tau t_1$, we write 
   \begin{align*}
      \PR\lr{\SRWM{i}{\tau}=1 \Mid F\cap \cS}
      &= \sum_{v\in S_1'(i)}\PR\lr{\SRWM{i}{\tau}=1 \Mid F\cap \cS\cap\lrc{X_{\tau t_1}=v}}\PR\lr{X_{\tau t_1}=v \Mid F\cap \cS}\\
      &= \sum_{v\in S_1'(i)}\PR\lr{\SRWM{i}{\tau}=1 \Mid X_{\tau t_1}=v}\PR\lr{X_{\tau t_1}=v \Mid F\cap \cS}
   \end{align*}
   We are left with the following lower bound on $\PR\lr{\SRWM{i}{\tau}=1 \mid F\cap \sigma}$:
   \begin{align}
       &\sum_{v\in S_1'(i)}\PR\lr{\SRWM{i}{\tau}=1 \Mid X_{\tau t_1}=v}\PR\lr{X_{\tau t_1}=v \Mid F\cap \cS}-\PR\lr{\cS_{34}^\compl \Mid F\cap \cS}\nonumber\\
       &\geq \inf_{v\in S_1'(i)}\PR\lr{\SRWM{i}{\tau}=1 \Mid X_{\tau t_1}=v}-\PR\lr{\cS_{34}^\compl \Mid F\cap \cS}.
       \label{eq:mainsrwm}
   \end{align}
   
   We start with the first term in~\eqref{eq:mainsrwm}; that is, we derive a lower bound on $\PR\lr{\SRWM{i}{\tau}=1 \Mid X_{\tau t_1}=v}$ that is uniform in $v$. 
   Since $\SRWM{i}{\tau}$ is composed of the events $E_1$, $E_2$ and $E_3$, which are independent of one another since they involve disjoint time intervals, we 
   will derive a lower bound for each of them. For the event $E_1$, we will require that an edge adjacent to $v$ (call it $e$) opens during the first half of $I_1$, 
   so that $e$ has time to remains open for time $C_{\ref*{cte:i1}}/\mu$ during $I_1$. Recall that $I_1$ has length $t_1/2$, so its first half has length $t_1/4$, and the rate at which an edge opens due to a $\star$-update 
   is $\mu p_\star \frac{p_\rmin}{p_\star}=\mu p_\rmin$, and the rate at which an edge close due to a $\star$-update is $1-p_\rmax$. We obtain 
   \begin{align*}
      \PR\lr{E_1 \Mid X_{\tau t_1}=v}
      = \lr{1-e^{-2d\mu p_\rmin \frac{t_1}{4}}}\lr{\frac{1-p_\rmax}{p_\star}}^{2d-1}e^{-(4d-2)\mu p_\rmin \frac{t_1}{2}}e^{-\mu(1-p_\rmax)\frac{C_{\ref*{cte:i1}}}{\mu}}.
   \end{align*}
   In the product above, the first term corresponds to an edge adjacent to the walker (call it $e$) opening during the first half of $I_1$, the second term is the probability that all $2d-1$ edges adjacent to $e$ are 
   closed at that time, the third term is the probability that none of the $4d-2$ edges adjacent to $e$ open until the end of $I_1$, and the fourth term is the probability that $e$ remains open for at least time $C_{\ref*{cte:i1}}/\mu$.
   Recalling that $t_1=\frac{\sqrt{\ell}}{\mu}$ and that $\ell=p^{-\frac{1}{3d}}$ we obtain
   \begin{align*}
      \PR\lr{E_1 \Mid X_{\tau t_1}=v}
      = \lr{1-e^{-d p_\rmin \frac{\sqrt{\ell}}{2}}}\lr{\frac{1-p_\rmax}{p_\star}}^{2d-1}e^{-(2d-1)p_\rmin \sqrt{\ell}}e^{-C_{\ref*{cte:i1}}(1-p_\rmax)}.
   \end{align*}
   Using that $p_\star\leq 1$ in the second term, $p_\rmin\in \lrb{\frac{p}{1+q},p}$ in the first and third terms, and
   $p_\rmax\geq 0$ in the fourth term, and then making $p$ small enough so that $p_\rmin\leq p_\rmax\leq \frac{1}{2}$ and 
   $e^{-d \frac{p}{1+q}\frac{\sqrt{\ell}}{2}}\leq 1- \frac{d p\sqrt{\ell}}{4(1+q)}$ we obtain
   \begin{align*}
      \PR\lr{E_1 \Mid X_{\tau t_1}=v}
      &\geq \lr{1-e^{-d \frac{p}{1+q}\frac{\sqrt{\ell}}{2}}}\lr{1-p_\rmax}^{2d-1}e^{-(2d-1)p \sqrt{\ell}}e^{-C_{\ref*{cte:i1}}}\\
      &\geq \frac{d p \sqrt{\ell}}{4(1+q)}2^{-2d+1}e^{-(2d-1)p \sqrt{\ell}}e^{-C_{\ref*{cte:i1}}}\\
      &=\frac{e^{-C_{\ref*{cte:i1}}}d}{2^{2d+1}(1+q)}p \sqrt{\ell}e^{-(2d-1)p \sqrt{\ell}}.
   \end{align*}
   Now note that $p\sqrt{\ell}=p^{1-\frac{1}{6d}}$ goes to $0$ as $p\to0$. Thus, we can take $p$ small enough so that $p \sqrt{\ell}e^{-(2d-1)p \sqrt{\ell}} \geq \frac{p\sqrt{\ell}}{2}$ to otain
   \begin{align}
      \PR\lr{E_1 \Mid X_{\tau t_1}=v}
      \geq \frac{e^{-C_{\ref*{cte:i1}}}d}{2^{2d+2}(1+q)}p\sqrt{\ell}
      = \frac{e^{-C_{\ref*{cte:i1}}}d}{2^{2d+2}(1+q)}p^{1-\frac{1}{6d}}.
      \label{eq:srwme1}
   \end{align}
   Event $E_1$ is the main one governing the probability that SRWM occurs, since it involves the opening of an edge, which has small probability. For $E_2$ and $E_3$ we will just derive simple bounds that will not go to $0$
   as $p\to0$. Recall that $I_2$ and $I_3$ are time intervals of length $1/\mu$, so
   $$
      \PR\lr{E_2 \Mid X_{\tau t_1}=v}
      =\lr{1-e^{-\mu (1-p_\rmax)\frac{1}{\mu}}}e^{-\mu p_\rmin\frac{1}{\mu}}e^{-(4d-2)\mu p_\rmin\frac{1}{\mu}},
   $$  
   where the first term is the probability that $e$ has a $\star$-update to close, the second term is the probability that $e$ does not get a $\star$-update to open, and the final term is the probability that 
   all $4d-2$ edges adjacent to $e$ do not receive a $\star$-update to open. 
   Recall that $p_\rmin$ and $p_\rmax$ both go to $0$ as $p\to 0$, so we obtain that 
   \begin{align}
      \PR\lr{E_2 \Mid X_{\tau t_1}=v}
      \geq 1-\frac{1}{2e}.
      \label{eq:srwme2}
   \end{align}
   Regarding $E_3$, we obtain
   \begin{align}
      \PR\lr{E_3 \Mid X_{\tau t_1}=v}
      =\lr{1-e^{-(2d-1)\mu p_\star\frac{1}{\mu}}}e^{-2d\mu p_\rmin\frac{1}{\mu}}
      \geq 1-\frac{1}{2e^{2d-1}},
      \label{eq:srwme3}
   \end{align}
   where the inequality follows for all small enough $p$ since $p_\star\to 1$ and $p_\rmin\to0$ as $p\to0$.
   Putting~\eqref{eq:srwme1},~\eqref{eq:srwme2} and~\eqref{eq:srwme3} together we have a constant $c=c(d,q)$ so that for all small enough $p$ we obtain
   $$
      \PR\lr{\SRWM{i}{\tau}=1 \Mid X_{\tau t_1}=v}
      \geq c p^{1-\frac{1}{6d}}.
   $$
   Plugging the bound above into~\eqref{eq:mainsrwm}, we obtain
   \begin{equation}
       \PR\lr{\SRWM{i}{\tau}=1 \mid F\cap \sigma} 
       \geq c p^{1-\frac{1}{6d}}-\PR\lr{\cS_{34}^\compl \Mid F\cap \cS}.
       \label{eq:main2srwm}
   \end{equation}
   Now as we explained in the beginning of the proof, $\cS_{34}$ is independent of $\cF_{\tau t_1}$ and of $\cS$. Moreover, $\cS_{34}$ is composed of an intersection of independent events $\hat\cG_{34}(\cdot,\tau)$ since 
   $\sigma\in \Sigma_{i,\tau}$ so that $R_1(i,\tau)$ is $k_\rmax$-great. Therefore, 
   \begin{align*}
       \PR\lr{\SRWM{i}{\tau}=1 \mid F\cap \sigma} 
       &\geq c p^{1-\frac{1}{6d}}-\PR\lr{\bigcup_{(i',\tau)\in \cJ_{34}}\hat\cG^\compl_{34}(i',\tau)}\\
       &\geq c p^{1-\frac{1}{6d}}-\sum_{(i',\tau)\in \cJ_{34}} \PR\lr{\hat\cG^\compl_{34}(i',\tau)}\\
       &\geq c p^{1-\frac{1}{6d}}-5^d \exp\lr{-C_{\ref*{cte:bblss}} \log^2\ell},
       \label{eq:main2srwm}
   \end{align*}
   where the last inequality follows from Lemma~\ref{lem:badboxlss}. Since as $p\to 0$ the second term is much smaller than the first one, the lemma follows.
\end{proof}

\subsection{Concluding the second phase}

Recall that for simplicity we are assuming that $(X_0,0)=(0,0)$, and recall the value of $\Delta_2$ from~\eqref{eq:delta}. Denote with $I_d:V\to V$ the identity map, then we define
\begin{equation}
    F_2\defn\{(0,0) \text{ is $k_\rmax$-great}\}\cap\{\Phi_{\Delta_2}=I_d\}\cap\mathcal{B}_{\Delta_2}'.
    \label{eq:f2}
\end{equation}
If $F_2$ is verified, the second phase is successful and the third phase can start, 
otherwise we let the two processes evolve independently until the end of \emph{phase 3}, and only then restart the coupling from phase 1.

\begin{lemma}
   \label{lemmaf2}
   Assume $F_1$ is verified at time $0$. For any $\delta>0$ and for all $p$ small enough, there exists $C_{\ref*{cte:delta2}}=C_{\ref*{cte:delta2}}(d,p,\delta)>0$ in the definition of $\Delta_2$ 
   and $n_0<\infty$ such that for all $n>n_0$
   \begin{equation*}
       \mathbb{P}(F_2)\geq  1-\delta.
   \end{equation*}
\end{lemma}
\begin{proof}
   Let $\mathcal{P}$ be any feasible path and consider
   \begin{align*}
      \Upsilon_1^\mathcal{P}&\defn\inf\lrc{\tau>0 \colon (\mathcal{P}(\tau t_1),\tau t_1)\in R_1^\core(i,\tau) \text{ where $R_1(i,\tau)$ is a $k_\rmax$-great box}},\\
      \Upsilon_j^\mathcal{P}&\defn\inf\lrc{\tau > \Upsilon_{j-1}^\mathcal{P} \colon (\mathcal{P}(\tau t_1),\tau t_1)\in R_1^\core(i,\tau) \text{ where $R_1(i,\tau)$ is a $k_\rmax$-great box}},
   \end{align*}
   for $j\geq 2$.
   For any feasible path $\mathcal{P}$ we let $\kappa_\mathcal{P}$ be the largest value such that $\Upsilon_{\kappa_\mathcal{P}}^\mathcal{P}\leq  \frac{\Delta_2}{t_1}-2$.
   Recall that $\Sigma$ represents the set of all possible realizations of occurrences and non occurrences for the events $\cG_{12}(\cdot,\cdot)$ and $\hat\cG_{34}(\cdot,\cdot)$, so the good and bad boxes at all scales
   are deterministic functions of $\sigma$. 
   Let $\fF(\sigma)$ be the set of all feasible paths for a given $\sigma$. 
   Given the uniform bound from Lemma~\ref{srwmprob}, we let $Y_1,Y_2,\ldots$ be a sequence of i.i.d.\ Bernoulli random variables of parameter $C_{\ref*{cte:srwm}} p^{\frac{6d-1}{6d}}$ where $Y_j$ gives whether the $j$th SRWM will
   succeed when we try to perform it during the coupling.
   Let
   \begin{equation}
      \zeta\defn \frac{C_{\ref*{cte:gbc}}C_{\ref*{cte:delta2}}}{4\sqrt{\ell}}n^2
      = \frac{C_{\ref*{cte:gbc}}}{4t_1}\Delta_2,
      \label{eq:zeta}
   \end{equation}
   where $C_{\ref*{cte:gbc}}$ is from Lemma~\ref{gbcross} and $C_{\ref*{cte:delta2}}$ from the definition of $\Delta_2$ in~\eqref{eq:delta}.
   Define the following events
   \begin{align*}
       &E_1 \defn \lrc{\sum\nolimits_{j=1}^\zeta Y_j \geq c_0n^2}, \text{ with $c_0$ to be chosen later, and}\\
       &E_2 \defn \lrc{\kappa_\mathcal{P}\geq \zeta \text{ for all feasible paths $\mathcal{P}\in\fF(\sigma)$}}.
   \end{align*}
   In this stage we want to couple the position of the walkers. 
   From Lemma \ref{distnotchange}, by doing identity coupling whenever the walkers are not in a great box, their relative distance does not change. 
   Their relative distance changes only when they are in a great box and a simple random walk moment is successfully performed. 
   Let $E_\coup=\{\Phi_{\Delta_2}=I_d\}\cap \mathcal{B}_{\Delta_2}'$.
   Hence
   \begin{align*}
      \mathbb{P}(F_2^c)
      &\leq \PR\lr{(0,0) \text{ is not $k_\rmax$-great}} + \PR\lr{E_1\cap E_2\cap E_\mathrm{coup}^c}+\PR\lr{E_1^c}+\PR\lr{E_2^c}.
   \end{align*}
   
   We start by bounding the first term. 
   Notice that $\lrc{(0,0) \text{ is not $k_\rmax$-great}}$ does not depend on the configuration at time 0. 
   Moreover, at time 0, the walkers are stuck in a vertex, so $X_s$ has to leave $R_1(0,0)$ from the time boundary if $R_1(0,0)$ is a good box. 
   Using Lemmas~\ref{lem:badbox} and~\ref{lemrho} to bound $\rho_j$, we obtain
   \begin{align*}
       \PR\lr{(0,0) \text{ is not $k_\rmax$-great}}
       \leq  c_d\sum_{j=1}^{k_\rmax}\rho_j
       \leq  c_d\lr{\rho_1+\sum_{j=2}^{k_\rmax}\rho_1^{2^{k-2}}}
       \leq  3c_d\rho_1
       \leq  \frac{\delta}{4},
   \end{align*}
   where $c_d$ is a constant that counts the number of boxes whose 2-enlargement intersects $R_1(0,0)$, and the last inequality follows for all $p$ small enough.
   Next we bound 
   $$
      \PR\lr{E_1\cap E_2\cap E_\mathrm{coup}^c}\leq \frac{\delta}{4}.
   $$ 
   Under $E_1\cap E_2$, we know we performed at least $c_0 n^2$ simple random walk moments. So, 
   $\PR\lr{E_1\cap E_\mathrm{coup}^c}$ can be bounded by the probability that two random walkers performing SRW on $\mathbb{T}_n^d$ are not coupled after 
   $c_0 n^2$ steps. Taking $c_0=c_0(d,\delta)$ large enough we obtain that they have coupled with probability at least $1-\frac{\delta}{4}$.
   
   Next, we bound $\PR\lr{E_2^c}$. 
   From Lemma \ref{lemmakmax}, with probability at least $1-\rho_1^{2^{k_\rmax-3}}$, all $k_\rmax$-boxes in the tessellation are good.
   Thus, Lemma~\ref{gbcross} gives that while traversing the first good $k_\rmax$-great box any feasible paths will traverse at least 
   $$
      C_{\ref*{cte:gbc}}\frac{t_{k_\rmax}}{t_1}
   $$ 
   $k_\rmax$-great $1$-boxes. After the feasible path exits the first $k_\rmax$-great box, it enters into another one and we obtain again another set of $k_\rmax$-great $1$-boxes.
   The total number of steps we can iterate this procedure up to reaching time $\Delta_2$ is $\frac{\Delta_2}{2t_{k_\rmax}}-1\geq \frac{\Delta_2}{4t_{k_\rmax}}$. Therefore, any feasible path must traverse at least 
   $$
      C_{\ref*{cte:gbc}}\frac{t_{k_\rmax}}{t_1} \frac{\Delta_2}{4t_{k_\rmax}} = C_{\ref*{cte:gbc}}\frac{\Delta_2}{4t_1}=\zeta
   $$ 
   $k_\rmax$-great boxes.
   Hence,
   \begin{equation*}
       \PR\lr{E_2^c}\leq  \rho_1^{2^{k_\rmax-3}} \leq \frac{\delta}{4},
   \end{equation*}
   by simply having $n$ large enough.
   
   Finally we bound $\PR\lr{E_1^\compl}$. 
   This is a simple Chernoff bound for the sum of independent Bernoulli random variables, where $\PR(Y_j) = C_{\ref*{cte:srwm}} p^{\frac{6d-1}{6d}}$. Since  
   $$
      \E\lr{\sum_{j=1}^\zeta Y_j}=C_{\ref*{cte:srwm}} p^{\frac{6d-1}{6d}}\zeta 
      = \frac{C_{\ref*{cte:gbc}}C_{\ref*{cte:delta2}}C_{\ref*{cte:srwm}}}{4} \frac{p^{\frac{6d-1}{6d}}}{\sqrt{\ell}}n^2
      =\frac{C_{\ref*{cte:gbc}}C_{\ref*{cte:delta2}}C_{\ref*{cte:srwm}}}{4} p^{\frac{1}{6d}+\frac{6d-1}{6d}}n^2
      = \frac{C_{\ref*{cte:gbc}}C_{\ref*{cte:delta2}}C_{\ref*{cte:srwm}}}{4} pn^2.
   $$
   Now we take $C_{\ref*{cte:delta2}}$ large enough so that the above is larger than $2c_0n^2$, which gives a constant $c$ so that 
   $$
      \PR\lr{E_1^\compl} \leq \exp\lr{-c \E\lr{\sum_{j=1}^\zeta Y_j}} \leq \exp\lr{-c 2c_0n^2} \leq \frac{\delta}{4},
   $$
   where the last inequality follows by taking $n$ large.
\end{proof}

To conclude the second phase, once the walkers are coupled after one SRWM, we just perform identity coupling up to time $\Delta_2$. 
If $t$ is a time where a SRWM ended, notice that Lemma~\ref{lem:recoupgraph} gives that $\cB_t'$ holds. So we succeed performing identity coupling up to $\Delta_2$ by
Lemmas~\ref{lemic1},~\ref{lemic1seq} and~\ref{distnotchange}.

\section{Third Phase}\label{thirdphase}
The third phase starts at time $\Delta_2$, at which time the walkers are coupled and $\cB_{\Delta_2}'$ holds. 
During the third phase we let $\bar{M}_t^\star$ mimic the evolution of ${M}_t^\star$ by doing identity coupling on both the motion of the walkers and the updates of the edges. 
We now check whether the processes are fully coupled by time $\Delta_3=\Delta_2+\frac{n^2}{\mu}$.

Define
\begin{equation}
    F_3\defn \lrc{X_{\Delta_3}=\bar X_{\Delta_3} \text{ and } {\eta}_{\Delta_3}^\star(e)=\bar{\eta}_{\Delta_3}^\star(e)\,\forall e\in E(\bT)_n^d}.
    \label{eq:f3}
\end{equation}
If $F_3$ is not verified, we restart the coupling at time $\Delta_3$ from phase 1. 
\begin{lemma}
\label{lemmaf3}
   For any $\delta>0$, if $p$ is small enough and $n$ large enough, we obtain 
   \begin{equation}
       \PR\lr{F_3}\geq  1-\delta.
   \end{equation}
\end{lemma}
\begin{proof}
   Recall that boxes contained in $[0,\Delta_3]$ have been sampled as good or bad during the second phase. 
   By Lemmas~\ref{lemic1},~\ref{lemic1seq} and~\ref{distnotchange}, identity coupling is successful provided we cannot enter a bad box without first entering its 2-enlargement. 
   Therefore, for the walkers to get uncoupled during $[\Delta_2,\Delta_3]$, it must so happen that the walkers entered a bad box of some scale $k$ 
   whose $2$-enlargement intersects $[0,\Delta_2]$ and which was not observed during the second phase because it is not contained in $[0,\Delta_3]$. 
   We now count the number of such boxes. 
   
   We start by deriving bounds on $\ell_k$ and $t_k$, the size of the boxes of scale $k$, for which the above can happen. 
   When $k\geq  k_\rmax$, we can choose $n$ large enough so that for any $m,\ell$ fixed
   \begin{align*}    
       2\ell k^{2k}\leq \ell_{k}&=m^{k}(k!)^2\ell\leq \ell k^{3k},\\
       2\sqrt{\ell}k^{2k}\leq \mu t_{k}&=m^{k}(k!)^2\sqrt{\ell}\leq \sqrt{\ell}k^{3k}.
   \end{align*}
   Recall that $k_\rmax=\log_2\log n$ from~\eqref{eq:kmax}. Then, $\mu t_{k_\rmax}\leq \sqrt{\ell} k_\rmax^{3\log_2\log n}$ is much smaller than a polynomial in $n$. 
   Therefore, any box whose enlargement intersects $[0,\Delta_2]$ and is not contained in $[0,\Delta_3]$ must be of scale larger than $k_\rmax$. So
   \begin{align*}
    \PR\lr{F_3}
       &\geq  1-\PR\lr{\exists\,k>k_\rmax \colon R_k(i,\tau) \text{ is bad and } \Delta_3\in T_k^\enlb(\tau)}.
   \end{align*}
   Next, using the bounds we derived above for $\ell_k$ and $t_k$, the number $\zeta_{k}$ of boxes of scale $k$ that intersect $\mathbb{T}_n^d\times \Delta_3$ is bounded above and below by
   \begin{align*}
       \zeta_{k}\geq &\left(\frac{n}{3\ell_k}\right)^d 24 \geq \frac{24n^{d}}{3^{d}\ell^{d} k^{3dk}},\\
       \zeta_{k}\leq &1+24\left(\frac{n}{\ell_k}\right)^d\leq 1+\frac{24n^{d}}{\ell^{d} k^{2dk}}.
   \end{align*}
   In the upper bound of $\zeta_k$ we add a $1$ to the fraction to consider the case when $k$ is so large that we cannot find a box all contained in the tessellation.
   Using Lemma \ref{lemrho} the probability that there exists a box of scale $k_\rmax$ or bigger that is bad is bounded above by
   \begin{equation*}
       \sum_{k\geq  k_\rmax}\zeta_k\rho_k\leq \sum_{k\geq  k_\rmax}\zeta_k\rho_1^{2^{k-2}},
   \end{equation*}
   moreover using the inequalities above for $\zeta_k$ it is easy to see that, for any $k\geq  k_\rmax$, 
   \begin{equation*}
       \sum_{k\geq  k_\rmax}\zeta_k\rho_k\leq 2 \zeta_{k_{\rmax}}\rho_1^{2^{k_{\rmax}-2}}.
   \end{equation*}
   Since $2^{k_\rmax}=\log n$, by taking $p$ small enough we make $\rho_1$ small enough, which gives that 
   \begin{align*}
       \PR(F_3)\geq 1-\delta.
   \end{align*}
\end{proof}


\section{Completing the proof of Theorem~\ref{thm:main}}
\label{proofteo}

\begin{proof}[Proof of Theorem~\ref{thm:main}]
   Let $\{{M}_t^\star\}_{t\geq 0}$ and $\{\bar{M}_t^\star\}_{t\geq 0}$ denote two copies of the process, 
   each starting from an arbitrary configuration in $\mathbb{T}_n^d\times\{0, 1\}^{E(\mathbb{T}_n^d)}$. 
   Recall the events $F_1$, $F_2$ and $F_3$ from~\eqref{eq:f1},~\eqref{eq:f2} and~\eqref{eq:f3}. 
   If the three events hold then $X_{\Delta_3}= \bar X_{\Delta_3}$ and $\eta_{\Delta_3}^\star\equiv \bar \eta_{\Delta_3}$. So from $\Delta_3$ onwards we can keep the processes coupled.
   We can now set $\delta=\frac{1}{12}$ so that $F_1\cap F_2\cap F_3$ all hold with probability at least $\frac{3}{4}$.
   If any of the above fails, we just let the processes evolve independently up to time $\Delta_3$ and restart from scratch. Since $\Delta_3$ is of order $n^2/\mu$ from~\eqref{eq:delta}, we obtain that the mixing time 
   is of order $n^2$ concluding the proof.
\end{proof}


\section{Proof of the lower bound (Theorems~\ref{thm:lb} and~\ref{thm:lb2})}
\label{prooflb}

The proof of the lower bounds are identical to the ones in~\cite{peres2015random}. We add them here for completion.

\begin{proof}[Proof of Theorem~\ref{thm:lb}]
   First we introduce a discrete time Markov chain $\tilde{M}_k= (\tilde X_k,\tilde \eta_k)$ which is defined by sampling the continuous time chain $M_t=(X_t,\eta_t)$ on intervals of length $\delta$; that is,
   $$
      \tilde X_k = X_{k\delta}
      \quad\text{and}\quad
      \tilde \eta_k = \eta_{k\delta},
   $$
   where $\delta$ is given from~\eqref{eq:assump2}.
   Let $\tilde \gamma=\tilde\gamma(\tilde M)$ and $\gamma=\gamma(M)$ be the spectral gaps of the discrete time and continuous time chain, respectively. We obtain
   $$
      1-\tilde \gamma = \exp\lr{- \delta\gamma}.
   $$
   For all $\tilde\gamma\leq 1/2$ we simply use the bound $\gamma \leq \frac{2\tilde\gamma}{\delta}$.
   The lower bound on the relaxation time follows by taking the function $f(x,\xi)=d(x,0)$, so $f(X_t,\eta_t)$ is the distance between the walker and the origin of $\T_n^d$. 
   Since the stationary distribution of the walker is uniform by~\eqref{eq:assump1}, it follows that $\Var(f)\geq c n^2$ for some constant $c>0$. Moreover, from~\eqref{eq:assump2}, we have
   \begin{align*}
      \cE(f,f) 
      &= \frac{1}{2}\sum_{x,\xi}\pi(x)\nu(\xi) \sum_{x',\xi'} P((x,\xi),(x',\xi')) \lr{d(x,0)-d(x',0)}^2\\
      &\leq \frac{1}{2}\E_{x\sim \pi}\lr{D_{x,\delta}^2}
      \leq \frac{1}{2}C_\cref{cte:assump2},
   \end{align*}
   where $\E_{x\sim \pi}$ denotes the expectation where $x$ is a random variable sampled according to $\pi$, the uniform measure on $\T_n^d$.
   From the above we obtain
   $$
      \tilde\gamma\leq \frac{C_\cref{cte:assump2}}{2c n^2}.
   $$
   If the above is at most $1/2$ we obtain
   $$
      \gamma 
      \leq \frac{C_\cref{cte:assump2}}{c n^2\delta}.
   $$
   Otherwise, if $\tilde\gamma\geq 1/2$ we obtain that $\gamma$ is of order $1/\delta$.
   The above establishes the relaxation time of the chain. 
\end{proof}
\begin{proof}[Proof of Theorem~\ref{thm:lb2}]
   We use the following nice result from~\cite{Naor2006Jul}, which appeared implicitly already in
   \cite{Ball1992Jun}.
   \begin{lemma}
      Let $\lrc{Y_k}_{k\in\Z}$ be a discrete-time, stationary, reversible Markov chain with finite state space $\cS$, and let $h\colon S\to \R^m$ for some $m\in \Z_+$. Then, for each $k\geq0$
      $$
         \E\lr{\|h(Y_k)-h(Y_0)\|_{L_2}^2}\leq k \E\lr{\|h(Y_1)-h(Y_0)\|_{L_2}^2},
      $$
      where $\|\cdot\|_{L_2}$ denotes the Euclidean norm on $\R^m$.
   \end{lemma}
   Letting $g_n \colon \T_n^d \to \R^{2d}$ the function 
   $$
      g_n(x_1,x_2,\ldots,x_d)=\lr{n\cos\lr{2\pi x_1/n},n\sin\lr{2\pi x_1/n},\ldots,n\cos\lr{2\pi x_d/n},n\sin\lr{2\pi x_d/n}}.
   $$
   For $x\in \T_n^d$ and $\xi\in\lrc{0,1}^{E\lr{\T_n^d}}$ we let $h(x,\xi)=g_n(x)$. Then, noting that $g_n$ is bi-Lipschitz with some constant $c$ we have 
   \begin{align*}
      \E_{\pi\times\nu}\lr{\|\tilde X_k-\tilde X_0\|_1^2}
      &\leq c^2 \E_{\pi\times\nu}\lr{\lr{g_n(\tilde X_k)-g_n(\tilde X_0)}^2}\\
      &\leq c^2 k \E_{\pi\times\nu}\lr{\lr{g_n(\tilde X_1)-g_n(\tilde X_0)}^2}\\
      &\leq c^4 k\E_{\pi\times\nu}\lr{\|\tilde X_1-\tilde X_0)\|_2^2}
      \leq c^4 k\E_{\pi\times\nu}\lr{D_{\tilde X_0,\delta}^2}
      \leq c^4 C_\cref{cte:assump2} k.
   \end{align*}
   Hence for any $t\geq \delta$ we have
   $$
      \E_{\pi\times\nu}\lr{\|X_{t} - X_0\|_1^2}\leq c^4 C_\cref{cte:assump2} \left\lceil\tfrac{t}{\delta}\right\rceil
      \leq 2c^4 C_\cref{cte:assump2} \frac{t}{\delta}.
   $$
   Now for the total variation starting from a stationary environment, we simply make
   \begin{align*}
      \|\upsilon_t - \pi\times \nu\|_\TV
      &\geq \PR\lr{\|X_t-X_0\|_1 \leq \epsilon^{1/d} n}\lr{1-\frac{2\epsilon}{3}}\\
      &= \lr{1-\PR\lr{\|X_t-X_0\|_1 \leq \epsilon^{1/d} n}}\lr{1-\frac{2\epsilon}{3}}\\
      &\geq \lr{1-\frac{\E\lr{\|X_t-X_0\|_1^2}}{\epsilon^{2/d} n^2}}\lr{1-\frac{2\epsilon}{3}}\\
      &\leq \lr{1-\frac{2c^4 C_\cref{cte:assump2}t}{\delta\epsilon^{2/d} n^2}}\lr{1-\frac{2\epsilon}{3}}.
   \end{align*}
   Therefore, if $t\leq \frac{\epsilon^{\frac{2+d}{d}}\delta n^2}{6 c^4 C_\cref{cte:assump2}}$ we have that $\|\upsilon_t - \pi\times \nu\|_\TV\geq 1-\epsilon$.
\end{proof}

\section{Proof of Corollary~\ref{cor:lb}}
\label{sec:cor}

In order to apply the above to the random walk on dynamical random cluster model, we first need a certain sprinkling lemma for the 
random cluster model. 
Given $q\geq 1$ and $p>0$, let $\nu_{p,q}$ be the measure of a random cluster model with parameters $p,q$. 
Let $\eta$ be a configuration sampled from $\nu_{p,q}$. We construct a sprinkling by 
associating to each edge $e$ an independent Bernoulli random variable $Z(e)$ of parameter $\epsilon$. 
Define the configurations
$$
   \lr{\eta+Z}(e) = \ind{\eta(e)+Z(e) \geq 1} \text{ for $e\in E(\T_n^d)$}
$$
and 
$$
   \lr{\eta-Z}(e) = \ind{\eta(e)(1-Z(e)) = 1} \text{ for $e\in E(\T_n^d)$}.
$$
So $\eta+Z$ (resp., $\eta-Z$) is the configuration obtained from $\eta$ by opening (resp., closing) all edges $e$ with $Z(e)=1$.
Given two elements $\xi,\xi'$ of $\lrc{0,1}^{E(\T_n^d)}$ we say that $\xi \leq \xi'$ if $\xi(e)\leq \xi'(e)$ for all $e\in E(\T_n^d)$. 
\begin{lemma}[Sprinkling lemma]\label{lem:sprinkle}
   Let $q\geq 1$, $0<p<p'<1$ and $\epsilon>0$ be fixed. Let $\lrc{Z(e)\colon e\in E(\T_n^d)}$ be a collection of i.i.d.\ Bernoulli random variables of parameter $\epsilon$. 
   Let $\eta$ and $\eta'$ be random configurations with distributions $\nu_{p,q}$ and $\nu_{p',q}$, respectively. 
   If 
   \begin{equation}
      \epsilon+(1-\epsilon)p\leq p'
      \quad\text{and}\quad
      \epsilon+(1-\epsilon)\frac{p}{p+(1-p)q}\leq \frac{p'}{p'+(1-p')q}
      \label{eq:sprinkle}
   \end{equation}
   then there exists a coupling between $\nu,\nu',Z$ such that $\lr{\eta+Z}\leq \eta'$.
   Similarly, if 
   \begin{equation}
      (1-\epsilon)p'\geq p
      \quad\text{and}\quad
      (1-\epsilon)\frac{p'}{p'+(1-p')q}\geq \frac{p}{p+(1-p)q}
      \label{eq:sprinkle2}
   \end{equation}
   then there exists a coupling between $\nu,\nu',Z$ such that $\lr{\eta'-Z}\geq \eta$.
\end{lemma}
\begin{proof}
   Let $\lrc{\eta_t}_t$ and $\lrc{\eta'_t}_t$ be the single-site Glauber dynamics Markov chains on the random cluster model with parameters $(p,q)$ and $(p',q)$, respectively. 
   Let $\lrc{Z_t}_t$ be a Glauber dynamics Markov chain on the state space $\lrc{0,1}^{E(\T_n^d)}$ with stationary distribution given by a product of Bernoulli measures with parameter $\epsilon$.  
   Start with arbitrary configurations such that $\eta_0\equiv \eta'_0$ and $Z_0(e)=0$ for all $e\in E(\T_n^d)$. 
   Assume that $\eta_t+Z_t \leq \eta'_t$ at some time $t$. We will show that we can couple the next transition of the chains so that $\eta_{t+1}+Z_{t+1} \leq \eta'_{t+1}$. This establishes the lemma.
   For any edge $e$ and configuration $\eta$, let 
   $$
      \alpha(e,\eta) = \ind{\text{$e$ is a cut-edge in $\eta$}}.
   $$
   In the coupling we will choose the same edge to be updated in all chains. Let $e$ be such an edge.
   Then, note that 
   \begin{align*}
      \PR\lr{\lr{\eta_{t+1}+Z_{t+1}}(e)=1 \mid \eta_t,\eta'_t,Z_t}
      &= \epsilon + (1-\epsilon)\lr{\alpha(e,\eta_t)\frac{p}{p+(1-p)q}+\lr{1-\alpha(e,\eta_t)}p}.
   \end{align*}
   Since $\eta_t \leq \eta'_t$ we have that $\alpha(e,\eta_t)\geq \alpha(e,\eta_t')$. So if $\alpha(e,\eta_t)=0$ we have that $\alpha(e,\eta'_t)=0$, which gives 
   \begin{align*}
      \PR\lr{\lr{\eta_{t+1}+Z_{t+1}}(e)=1 \mid \eta_t,\eta'_t,Z_t}
      = \epsilon + (1-\epsilon)p
      \leq p'
      = \PR\lr{\eta_{t+1}'(e)=1 \mid \eta_t,\eta'_t,Z_t},
   \end{align*}
   where in the first inequality we used~\eqref{eq:sprinkle}.
   If $\alpha(e,\eta_t)=1$, then we use the second part of~\eqref{eq:sprinkle} and  that $\frac{p'}{p'+(1-p')q}\leq p'$ to write 
   \begin{align*}
      \PR\lr{\lr{\eta_{t+1}+Z_{t+1}}(e)=1 \mid \eta_t,\eta'_t,Z_t}
      &= \epsilon + (1-\epsilon)\frac{p}{p+(1-p)q}\\
      &\leq \frac{p'}{p'+(1-p')q}\\
      &\leq \alpha(e,\eta_t')\frac{p'}{p'+(1-p')q}+\lr{1-\alpha(e,\eta_t')}p'\\
      &= \PR\lr{\eta_{t+1}'(e)=1 \mid \eta_t,\eta'_t,Z_t}.
   \end{align*}
   Therefore, it follows that we can couple the next transition of the Markov chains so that $\eta_{t+1}\leq \lr{\eta_{t+1}+Z_{t+1}} \leq \eta_{t+1}'$. Consequently, we can couple the stationary measures of such chains to obtain
   that $\lr{\eta+Z}\leq \eta'$.
   
   For the second part of the lemma, we use the same strategy and analyze the transition probabilities for $\eta_t'-Z_t$. We have
   \begin{align*}
      \PR\lr{\lr{\eta'_{t+1}-Z_{t+1}}(e)=1 \mid \eta_t,\eta'_t,Z_t}
      &= (1-\epsilon)\lr{\alpha(e,\eta_t')\frac{p'}{p'+(1-p')q}+\lr{1-\alpha(e,\eta_t')}p'}.
   \end{align*}
   If $\alpha(e,\eta_t')=1$ then $\alpha(e,\eta_t)=1$, yielding
   \begin{align*}
      \PR\lr{\lr{\eta'_{t+1}-Z_{t+1}}(e)=1 \mid \eta_t,\eta'_t,Z_t}
      = (1-\epsilon)\frac{p'}{p'+(1-p')q}
      &\geq \frac{p}{p+(1-p)q}\\
      &=\PR\lr{\eta_{t+1}=1 \mid \eta_t,\eta'_t,Z_t}.
   \end{align*}
   If $\alpha(e,\eta_t')=0$ then 
   \begin{align*}
      \PR\lr{\lr{\eta'_{t+1}-Z_{t+1}}(e)=1 \mid \eta_t,\eta'_t,Z_t}
      = (1-\epsilon)p'
      &\geq p\\
      &\geq \alpha(e,\eta_t)\frac{p}{p+(1-p)q}+\lr{1-\alpha(e,\eta_t)}p\\
      &=\PR\lr{\eta_{t+1}=1 \mid \eta_t,\eta'_t,Z_t}.
   \end{align*}
   Therefore, there exists a coupling such that $\eta_{t+1}\leq \lr{\eta_{t+1}'-Z_{t+1}}$ and we obtain $\lr{\eta'-Z}\geq \eta$.
\end{proof}

\begin{proof}[Proof of Corollary~\ref{cor:lb}]
   We only need to check that assumptions~\eqref{eq:assump1} and~\eqref{eq:assump2} hold for the random walk on dynamical random cluster model.
   For any $q,p$ we have that~\eqref{eq:assump1} holds. For $q\geq 1$,~\eqref{eq:assump2} holds for all $p<p_\crit^q$ using the following
   argument. Take $p'=\frac{p+p_\crit^q}{2}\in(p,p_\crit^q)$. Take $\epsilon>0$ small enough so that~\eqref{eq:sprinkle} is satisfied. 
   We choose $\delta=\epsilon/\mu$ and take $\eta$ to be a random cluster configuration of parameters $p,q$. Note that the probability that a given edge gets refreshed during $[0,\delta]$ is
   $$
      1-e^{-\mu\delta} = 1-e^{-\epsilon} \leq \epsilon.      
   $$
   Therefore, if $Z(e)$ is a Bernoulli random variable of parameter of parameter $\epsilon$, we can couple $Z(e)$ with the refresh clocks of the dynamical random cluster so that if $e$ gets refreshed during $[0,\delta]$ then
   $Z(e)=1$. Therefore, this coupling gives that $\cC_x([0,\delta])$ is contained in the cluster of $x$ inside the configuration $\eta+Z$, which by Lemma~\ref{lem:sprinkle} is contained inside $\eta'$, a random cluster 
   configuration with parameters $p',q$. Then it follows by the sharpness of the phase transition~\cite{Duminil-Copin2019} that the cluster of $x$ in $\eta'$ has an exponential decay, establishing~\eqref{eq:assump2} and 
   allowing us to obtain the conclusions of Theorems~\ref{thm:lb} and~\ref{thm:lb2} for the random cluster model with $q\geq 1$. 
   
   Regarding the case $q<1$, one can deduce the exponential decay of the cluster $\eta+Z$ only when $p$ is small enough. This becames rather trivial as regardless of the state of the other edges, we obtain that an edge $e$ is open
   during $[0,\delta]$ with probability at most 
   $$
      \max\lrc{\frac{p}{p+(1-p)q},p}+\epsilon=\frac{p}{p+(1-p)q}+\epsilon
   $$
   which for small enough $p$ can be made smaller than $p_\crit$, the critical probability for independent percolation.
\end{proof}

\bibliographystyle{plain} 
\bibliography{Biblio}
\end{document}